\documentclass [11pt]{amsart}
\usepackage{amsmath}
\usepackage{amssymb}
\usepackage{amscd}
\usepackage{epsfig}
\usepackage{amsxtra}
\usepackage{pifont}
\usepackage{mathabx}
\usepackage{MnSymbol}
\usepackage{accents}
\usepackage{scalerel,stackengine}
\usepackage{tikz}
\usepackage{caption}
\usepackage{graphicx}
\graphicspath{{scans/}} 
\usepackage{pdfpages}
\usepackage{rotating}
\usepackage{adjustbox}
\usepackage{longtable}
\usetikzlibrary{shapes.geometric}
\usetikzlibrary{calc}
\usetikzlibrary{positioning}
\usepackage{hyperref}

\numberwithin{equation}{section}

\newcommand{\RR}{{\mathbb R}}

\newcommand{\ZZ}{{\mathbb Z}}

\newcommand{\NN}{\mathbb N}
\newcommand{\cT}{{\mathcal T}}

 \newtheorem{theorem}{Theorem}[section]
 \newtheorem{lemma}[theorem]{Lemma}
 \newtheorem{proposition}[theorem]{Proposition}

 \newtheorem{conj}[theorem]{Conjecture}
 \newtheorem{quest}[theorem]{Question}
  
 \newtheorem{definition}[theorem]{Definition}
 \newtheorem{example}[theorem]{Example}
  \newtheorem{remark}[theorem]{Remark}

\begin{document}

\title{The 1-2-3 conjecture for polygonal tilings}

\author[A. Charlesworth]{Alison Charlesworth}
\address{Department of Mathematics and Statistics, MacEwan University,  Edmonton, Alberta, Canada}
\email{charleswortha@mymacewan.ca}

\author[C. Ramsey]{Christopher Ramsey}
\address{Department of Mathematics and Statistics, MacEwan University,  Edmonton, Alberta, Canada}
\email{ramseyc5@macewan.ca}
\urladdr{https://sites.google.com/macewan.ca/chrisramsey/}

\author[N. Strungaru]{Nicolae Strungaru}
\address{Department of Mathematics and Statistics, MacEwan University, Edmonton, Alberta, Canada,
and
\newline \hspace*{\parindent}
Institute of Mathematics ``Simon Stoilow'', 
Bucharest, Romania}
\email{strungarun@macewan.ca}
\urladdr{https://sites.google.com/macewan.ca/nicolae-strungaru/home}

\begin{abstract} The 1-2-3 conjecture has been solved positively in 2024 for finite graphs and by extension for infinite graphs which are locally finite. The solution is non-constructive, and finding explicit solutions for large (or infinite) graphs is very hard. By exploiting the extra structure present in many non-periodic tilings, we find explicit solutions for the Chair (all three vertex placements), Non-Pinwheel, Pinwheel, Half-hex, Ammann-Beenker (two versions), Penrose Rhomb, and the Domino tilings. We prove that for any fully periodic tiling of the plane there exists a fully periodic solution, and provide an algorithm for finding such a solution. We give solutions for the fully periodic square, triangle and hexagonal lattices.  
\end{abstract}


\maketitle

\section{Introduction}

Let $G = (V, E)$ be a simple, undirected, locally finite graph. A $k$-edge labelling or weighting is a function $l: E\rightarrow \{1,\dots, k\}$ and the weighted degree of a vertex is 
\[
s_l(u) = \sum_{uv\in E} l(uv)\,.
\]
If for every $uv\in E$ we have $s_l(u)\neq s_l(v)$, then this edge labelling is called a vertex colouring. Karo\'{n}ski, \L{}uczak and Thomason conjectured in 2004 \cite{KLT} that every connected, finite, simple graph with at least two edges has a 3-edge labelling that produces a vertex colouring. This has been popularly called the {\bf 1-2-3 Conjecture}. This seemed a reasonable conjecture as they could prove in \cite{KLT} that it is true for all 3-colourable graphs and it is readily verified in any given small-enough example. 

There has been a lot of work over the past two decades around this and related conjectures. One approach was to find $k$ that worked for all graphs, with this version of the conjecture proven for, successively better, $k=30$ \cite{ADMRT}, $k=16$ \cite{ADR}, $k=13$ \cite{WY}, $k=6$ \cite{KKP1},  $k=5$ \cite{KKP2}, and $k=4$ \cite{Keusch0}. Finally, in 2024, Keusch \cite{Keusch} proved the 1-2-3 Conjecture to be true.

A triangle necessitates the use of $k=3$ to get a vertex-colouring, so a natural question is to describe the graphs which have vertex-colouring 2-edge labellings. This has been shown to be possible for all $d$-regular bipartite graphs for $d\geq 3$ \cite{CLWY}. A vertex-colouring 1-edge labelling is possible if and only if adjacent vertices have differing degrees.

The 1-2-3 Conjecture for infinite, locally finite graphs has been less studied as the finite version was difficult enough. Of course, the existence of a solution to the problem for finite graphs is enough to prove the solution for infinite graphs by way of a Zorn's Lemma argument \cite{Stawiski}. The proof in \cite{Keusch}, and by extension in \cite{Stawiski} is non-constructive, and in general finding explicit weights for large graphs is very complicated. While we do not know the complexity of this problem, it is well known that many similar graph colouring problems are NP complete. The problem becomes infinitely harder on infinite graphs, where in general one should not expect to find explicit solutions. An existential result is nice to know but it would be even nicer to have a constructive algorithm, which generally requires some structure.

A nicely structured infinite graph arises naturally as the edges and vertices of a polygonal tiling of the plane. Even these can be too complicated to come up with reasonable algorithms so one at least needs locally finite tilings. A special class of these are substitution tilings, all of which have finite local complexity (FLC) which is stronger than locally finite. See \cite{TAO} for a rigorous background on tilings and the Tilings Encyclopedia \cite{TilingEncyclopedia} for many beautiful non-periodic examples.

For many tilings of the plane with finitely many prototiles, there are extra properties which one can use to try to construct solutions to various colouring problems:
\begin{itemize}
\item{} There are usually finitely many local configurations, and it is often possible to solve this type of problem by assigning a weight/colour based on the local configuration around the edge/point.
\item{} For substitution tilings it is sometimes possible to solve the problem by finding a solution for the level $N$-supertile and showing that the solution must also work when the level $N$ supertiles are glued together.
\item{} We can sometimes find finitely many patches/local configurations, with the property that the tiling can be partitioned in these patches. Then, similarly to the above point, finding a solution for each individual patch, and showing that the solution still works when the pieces of the puzzle are put together can give a solution.
\end{itemize}
These techniques have been proven effective in finding optimal vertex-, edge- and face-colourings of various tilings of the plane \cite{ColouringsofAperiodicTilings}.

Due to the above, the 1-2-3 Conjecture is reformulated into Conjecture \ref{conj:123} which says that there exist 3-edge labellings on some collection of finite patches of a FLC polygonal tiling that leads to a vertex-colouring 3-edge labelling of the whole tiling. This is true for periodic tilings, Theorem \ref{thm:periodic}, but remains a conjecture for non-periodic tilings. However, we are able to establish the conjecture in quite a few non-periodic tilings. In fact, we obtain slightly more as the algorithms will work for any polygonal tiling whose local configurations match with the patch labeling. Let us emphasize here that, while for the tilings considered here we could find solutions, there are many examples where the number of finite local configurations of small size is very large, which makes it difficult to even find solutions for all small local configurations. Finding non-conflicting solutions for all small local configurations becomes exceedingly hard, and showing that they produce a global solution is tricky. Thus finding explicit solutions for tilings is not as easy as 1-2-3.  

The paper proceeds as follows, Section 2 contains all needed definitions and theory on tilings and graphs. Section 3 proves the 1-2-3 Conjecture for periodic polygonal tilings and illustrates the theory with worked-out examples for  the square, triangle, and hexagon periodic tilings. Finally, Section 4 proceeds with solutions to the conjecture for non-periodic tilings, namely: Chair (all three vertex placements), Non-Pinwheel, Pinwheel, Half-hex, Ammann-Beenker (two versions), Penrose Rhomb, and the Domino variant.

\section{Preliminaries}

We will be working with infinite graphs which arise from tilings. 

\begin{definition}
A {\bf tiling} in $\mathbb R^2$ is a collection of non-empty closed sets $\mathcal T = \{T_i \subset \mathbb R^2 : i\in \NN\}$ called a {\bf pattern} satisfying:
\begin{itemize}
    \item $\bigcup_{i\in\NN} T_i = \RR^2$,
    \item $T_i^\circ \cap T_j^\circ = \emptyset$ for all $i\neq j$,
    \item $T_i = \overline{T_i^\circ}$ (regular).
\end{itemize}
\end{definition}

The $T_i$ are called {\bf tiles} and the equivalence classes of all tiles up to congruence (translations, rotations, and reflections) are called {\bf prototiles}. Note that there are different usages of some of these terms, for instance prototile can mean tiles up to translation only. All of the examples we will be working with have connected tiles but this doesn't necessarily have to be true in general. The interested reader will find much, much more on this subject in \cite[Chapter 5 and 6]{TAO}.

\begin{definition}
A tiling is {\bf locally finite} if every compact set $K\subset \RR^2$ intersects only a finite number of tiles.
\end{definition}

In most cases we will abuse notation and refer to the tiling and its graph interchangeably. The 1-2-3 conjecture is true for nicely behaved tilings.

\begin{theorem}[Theorem 5, \cite{Stawiski}]
If $\mathcal T$ is a locally finite tiling with polygonal prototiles, then there is a solution to the 1-2-3 conjecture for $\mathcal T$.
\end{theorem}

As one would imagine, the proof uses a K\"{o}nig's Lemma/Zorn's Lemma type argument using the fact that any non-trivial patch of the tiling can be 1-2-3 solved by Keusch's result \cite{Keusch}. While this is a reasonable result it is completely non-constructive. To be able to locally construct a vertex-colouring 3-edge labelling we need more definitions.

\begin{definition}
A {\bf cluster} of a tiling is the collection of all tiles that non-trivially intersect a compact set. It is called a {\bf patch} if the compact set is also convex.
\end{definition}

\begin{definition}
A tiling $\mathcal T$ has {\bf FLC (finite local complexity)} if for every $R>0$ the set $\{\cT \cap B_R(t) : t\in \mathbb R\}$ is finite up to translations.
\end{definition}

FLC is strictly stronger than locally finite. For instance, any tiling with a finite number of polygonal tiles will be locally finite. However, the square tiling of the plane with each row shifted by the same irrational translation from the one below will not have FLC.

A tile and its nearest neighbours forms a particular kind of cluster called a {\bf collared tile} or {\bf collar}, see \cite[pg. 24]{Sadun}. Note that this is a somewhat vague definition as ``neighbours'' and ``nearest'' can be defined in multiple ways.

Often we will be interested in a {\bf local configuration} of the prototiles, which is a tiling of a compact set. This resembles a cluster or patch but may not be extendable to a tiling of the whole plane.

\begin{definition}
Given a tiling $\mathcal T$ the {\bf set of periods} is defined to be 
\[
\operatorname{per}(\mathcal T) = \{t\in\RR^2 : \mathcal T = t + \mathcal T = \{t + T_i : i\in\NN\}\}\,.
\]
The {\bf rank} of $\operatorname{per}(\mathcal T)$ is the dimension of its $\RR$-span (here 0, 1, or 2). A tiling is called {\bf periodic} if the rank of its periods is non-zero, {\bf fully periodic} if the rank is 2, and {\bf non-periodic} if the rank is 0. 
\end{definition}

\begin{definition}
To every locally finite tiling $\mathcal T$ whose prototiles are polygons we can associate a graph $G_\mathcal T = (V_\mathcal T, E_\mathcal T)$. The vertex set is the collection of vertices of the polygonal tiles and an edge exists between two vertices if they are both on a line segment in $\bigcap_{i\in\NN} \partial T_i$ with no intermediate vertices. 
\end{definition}

Note that a tile in the tiling may have more vertices on its boundary than given by the polygon because of how it adjoins other tiles. It is often desirable to ignore the degree 2 vertices and one could get rid of them in the graph by edge contraction. This will not be a concern if all tiles are convex polygons with vertices only at their corners.

We now turn to considering what finding a ``local'' solution to the 1-2-3 problem is in our context.

\begin{quest}\label{quest:123}
Let $\mathcal T$ be a locally finite tiling with polygonal prototiles. Does there exist a solution to the 1-2-3 problem which is locally derivable from $\mathcal T$? Meaning that there is some radius $r>0$ such that all patches that contain a ball of radius $r$ will be labeled identically with their translates.
\end{quest}

\begin{example} Let $\cT$ be the periodic square tiling of $\RR^2$. Let $l: E(\cT) \to  \RR$ be any weight function which is locally derivable from $\cT$.

Since $\cT$ is $\ZZ^2$-periodic, and since $l$ is locally derivable from $\cT$ we get that $l$ is $\ZZ^2$-periodic. This immediately implies that there exists $a,b \in \RR$ such that 
\[
l(e)=
\begin{cases}
a &\mbox{ if } e=\left( (n,m), (n+1,m) \right)\\
b &\mbox{ if } e=\left( (n,m), (n,m+1) \right) \,,
\end{cases}
\]
meaning that $l$ takes the value $a$ on all horizontal edges and $b$ on all vertical edges. But then 
\[
S_l(v)=2a+2b \qquad \forall v \in V(\cT) \,.
\]
Therefore, there is no solution to $\cT$ which is locally derivable from $\mathcal T$, answering the question in the negative.
\end{example}

Similarly, we can construct a non-periodic example by removing a single edge from the periodic square tiling. If the weight function was locally derivable then almost all of the tiling would be weighted the same as above, meaning that it will not produce a vertex colouring.

\begin{remark}
Even with the stronger assumption that $\cT$ is aperiodic, having only non-periodic tilings in its hull, the question still almost certainly has counterexamples.
\end{remark}

This implies that we need a more nuanced conjecture.

\begin{conj}\label{conj:123}
Let $\mathcal T$ be a locally finite tiling with a finite number of polygonal prototiles. Then, there exists some partition 
\[
\cT:= \bigcupdot_{k=1}^N \left(\varLambda_k+\mathcal{P}_k\right)
\]
into translations of finitely many (not necessarily distinct) finite subgraphs, and weights 
\[
l_k: \mathcal{P}_k \to \{1,2,3\}
\]
which produce a global solution to the 1-2-3 problem.
\end{conj}

\section{Periodic Tilings}

\begin{theorem}\label{thm:periodic}
Let $G$ be a fully periodic connected locally finite simple graph in $\RR^d$. Then, there exists a fully periodic solution to the 1-2-3 problem.  
\end{theorem}
\begin{proof}
Since $G$ is fully periodic and locally finite, $\operatorname{per}(G)$
is a lattice in $\RR^d$. Thus, there exist vectors $v_1, v_2 , \ldots, v_d\in \RR^d$ which are linearly independent over $\RR$ such that 
\[
\operatorname{per}(G) =\ZZ v_1 \oplus  \ZZ v_2 \oplus \ldots \oplus \ZZ v_d \,.
\]
Consider now the linear transformation 
\begin{align*}
T &: \RR^d \to \RR^d \\
T(v_i)&=e_i \,,
\end{align*}
where $e_1, e_2, \ldots, e_d$ is the standard basis in $\RR^d$.

Then $T$ is an automorphism of $\RR^d$ to itself and induces a graph isomorphism between $G$ and 
\[
G':=T(G) \,.
\]
It is obvious that 
\[
\operatorname{per}(G')=\ZZ^d
\]
and that any fully periodic solution to the 1-2-3 problem on $G'$ induces a fully periodic solution to the 1-2-3 problem on $G$.

\smallskip 
Therefore, without loss of generality we can assume that 
\[
\operatorname{per}(G)=\ZZ^d\,. 
\]

\smallskip

Let 
\[
\mathcal{F}:= V(G) \cap [0,1)^d \,.
\]
As well, let $\mathcal{J}$ be all edges incident to a vertex in $\mathcal{F}$. Since $G$ is locally finite, $\mathcal{F}$ and $\mathcal{J}$ are both finite. Since $G$ is connected and fully periodic, there are at least $2d$ edges in $\mathcal{J}$.

\smallskip 

Pick some $N \in \ZZ$ so that $\mathcal{J} \subset (-N,N)^d$.
Since $G$ is $\ZZ^d$-periodic, it is also $2N\ZZ^d$ periodic. Therefore, we can define a factor graph $G_{f}$
\[
G_f:= G/(2N\ZZ^d) \,.
\]
Formally, this is defined as follows: On the vertices $V(G)$ we can define an equivalence relation 
\[
v \equiv w \Leftrightarrow v-w \in 2N\ZZ^d \,.
\]
Since $G$ is $2N\ZZ^d$-periodic, $\equiv$ induces an equivalence relation on $E(G)$. 
Then, it is immediate the sets of equivalence classes
\begin{align*}
V(G_{f})&:= \{ [v] : v \in V(G) \} \\
E(G_{f})&:= \{ [v][w] : vw \in E(G) \} 
\end{align*}
is a graph.
Now, given the construction of the factor graph, for all $v \in V(G)$ the mapping $e \mapsto [e]$ is injective on $\mathcal{J}$ and so is injective on
$\{ e \in E(G): e \mbox{ is incident to } v \}$.

Since $G$ is locally finite and $2N\ZZ^d$-periodic, $G_f$ is a finite graph. Moreover, $G$ is connected and hence so is $G_{f}$. By the choice of $N$ no edge in $E(G_f)$ is a loop and so $G_f$ is simple.
Thus, $G_f$ is a finite connected simple graph with at least two edges. Therefore, there exists some 
\[
l: E(G_f) \to \{1,2,3\}
\]
which provides a solution to the 1-2-3 problem.

To weight the original graph, define 
\begin{align*}
w&: E(G) \to \{ 1,2,3\} \\
w(e)&:= l([e]) \,.
\end{align*}
Therefore, given that the factor maps are locally injective, we get
\begin{align*}
s_w(u)&= \sum_{uv \in E(G)} w(uv)=\sum_{[u][v] \in E(G_f)} l([u][v])= s_l([u]) \,.
\end{align*}
Since $s_l$ is a colouring of the vertices of $G_f$, then $s_w$ is a colouring of the vertices of $G$. Therefore, $w$ gives a $2N \ZZ^d$-periodic solution to the 1-2-3 problem.
\end{proof}

\begin{remark} The proof of Theorem~\ref{thm:periodic} is constructive in the sense that it gives a finite process for generating fully periodic solutions for the fully periodic tilings. As well, it answers Conjecture \ref{conj:123} positively in this case.
\end{remark}

Below, we provide some solutions for the most common fully periodic tilings in $\RR^2$, which are constructed geometrically without relying on the algorithm in Theorem~\ref{thm:periodic}.

\subsection{Square}
\begin{proposition} Consider the periodic tiling of the plane with the following $2 \times 2$ supertile. Assign 1-2-3 to the edges as below.

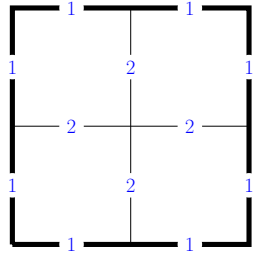
\begin{center}
\resizebox{100pt}{!}{
\begin{tikzpicture}[scale=1.7] 
\foreach \a/\b in {0/0}
{
\draw[line width=2.2mm] (0+\a,0+\b) -- (6+\a,0+\b) -- (6+\a,6+\b)--(0+\a,6+\b)--(0+\a,0+\b);
\draw (3+\a,0+\b)--(3+\a,6+\b);
\draw (0+\a,3+\b)--(6+\a,3+\b);
\foreach \x in  { 0, 3}
{
\draw (\x+1.5+\a,+\b) node[shape=circle,fill=white,font={\Huge}] {\color{blue}$1$};
\draw (\x+1.5+\a, 6+\b) node[shape=circle,fill=white,font={\Huge}] {\color{blue}$1$};
\draw (+\a,\x+1.5+\b) node[shape=circle,fill=white,font={\Huge}] {\color{blue}$1$};
\draw (6+\a, \x+1.5+\b) node[shape=circle,fill=white,font={\Huge}] {\color{blue}$1$};
\draw (\x+1.5+\a, 3+\b) node[shape=circle,fill=white,font={\Huge}] {\color{blue}$2$};
\draw (3+\a,\x+1.5+\b) node[shape=circle,fill=white,font={\Huge}] {\color{blue}$2$};
}}
\end{tikzpicture}}
\end{center}

This gives a solution to the 1-2-3 problem for the periodic tiling of the plane with squares.
\end{proposition}

\begin{proof}
The internal vertex of the supertile has a sum of 8. When we combine the supertiles along all edges the vertices adjacent to the internal vertex will all have two edges weights of 2 (the internal supertile edges) and two edges with weights of 1 (the external supertile edges) resulting in vertices that sum up to 6. The external corners of the supertile will combine with 3 other supertiles to give a vertex with edges each weighted 1 and sums to 4.

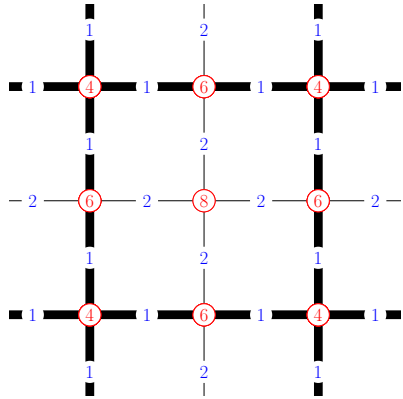
\tikzset{square/.pic={
       \foreach \a/\b in {0/0}
{
\draw[line width=4mm] (0+\a,0+\b) -- (6+\a,0+\b) -- (6+\a,6+\b)--(0+\a,6+\b)--(0+\a,0+\b);
\draw (3+\a,0+\b)--(3+\a,6+\b);
\draw (0+\a,3+\b)--(6+\a,3+\b);
\foreach \x in  { 0, 3}
{
\draw (\x+1.5+\a,+\b) node[shape=circle,fill=white,font={\Huge}] {\color{blue}$1$};
\draw (\x+1.5+\a, 6+\b) node[shape=circle,fill=white,font={\Huge}] {\color{blue}$1$};
\draw (+\a,\x+1.5+\b) node[shape=circle,fill=white,font={\Huge}] {\color{blue}$1$};
\draw (6+\a, \x+1.5+\b) node[shape=circle,fill=white,font={\Huge}] {\color{blue}$1$};
\draw (\x+1.5+\a, 3+\b) node[shape=circle,fill=white,font={\Huge}] {\color{blue}$2$};
\draw (3+\a,\x+1.5+\b) node[shape=circle,fill=white,font={\Huge}] {\color{blue}$2$};
\foreach \i/\j in {0/0,0/6,6/6,6/0}
{
\node[shape=circle, draw=red,fill=white,text=red,font={\Huge}] at (\i,\j) {4};
}
\foreach \i/\j in {0/3,3/6,6/3,3/0}
{
\node[shape=circle, draw=red,fill=white,text=red,font={\Huge}] at (\i,\j) {6};
}
\node[shape=circle, draw=red,fill=white,text=red,font={\Huge}] at (3,3) {8};
}}
    }}

\begin{center}

    \resizebox{150pt}{!}{ 
    \tikz{
    \clip (7,7) rectangle (25.5,25.5);
    \foreach \i/\j in {0/0,10.8/0,21.6/0,0/10.8,10.8/10.8,21.6/10.8,0/21.6,10.8/21.6,21.6/21.6}
    {\pic[scale=1.8] at (\i,\j) {square};}
    }
    }
\end{center}
\end{proof}

We notice that every individual square (prototile) will have a corner with a sum of 4, a corner with a sum of 8, and two corners with sums of 6. Therefore, 6 will occur with a frequency of 1/2 while 4 and 8 will each occur with a frequency of 1/4.

\subsection{Equilateral Triangle}
\begin{proposition}Consider the periodic tiling of the plane with rhombi consisting of 18 equilateral triangles. Assign 1-2-3 to the edges as below.

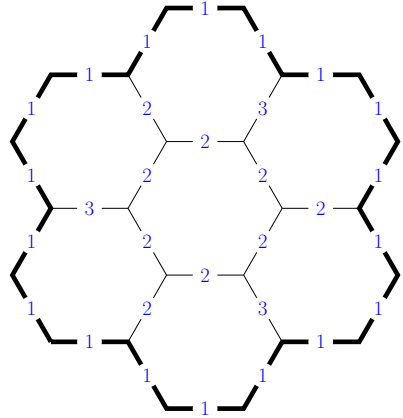
\begin{center}
    \resizebox{210pt}{!}{
    \begin{tikzpicture}[scale=1.2]
        \draw[line width=2.5mm] (0,0)--(6,10.4)--(18,10.4)--(12,0)--(0,0);
        \draw[line width=0.7mm] (4,6.92)--(16,6.92);
        \draw[line width=0.7mm] (2,3.46)--(14,3.46);
        \draw[line width=0.7mm] (2,3.46)--(4,0)--(10,10.4)--(14,3.46);
        \draw[line width=0.7mm] (4,6.92)--(8,0)--(14,10.4)--(16,6.92);
        \draw[line width=0.7mm] (6,10.4)--(12,0);
        \foreach \i/\j in {4/6.92,10/3.44,4/0,16/6.92}
        {
        \draw (\i+1,\j+1.73) node[shape=circle,fill=white,font={\Huge}]{\color{blue}$1$};
        }
        \foreach \i/\j in {2/3.46,8/6.92,8/0,14/3.46}
        {
        \draw (\i+1,\j+1.73) node[shape=circle,fill=white,font={\Huge}]{\color{blue}$2$};
        }
        \foreach \i/\j in {0/0,12/6.92,6/3.46,12/0}
        {
        \draw (\i+1,\j+1.73) node[shape=circle,fill=white,font={\Huge}]{\color{blue}$3$};
        }
        \foreach \i/\j in {12/6.92,6/3.46,16/6.92,12/0}
        {
        \draw (\i-1,\j+1.73) node[shape=circle,fill=white,font={\Huge}]{\color{blue}$1$};
        }
        \foreach \i/\j in {16/6.92,10/3.46,4/0}
        {
        \draw (\i-1,\j+1.73) node[shape=circle,fill=white,font={\Huge}]{\color{blue}$2$};
        }
        \foreach \i/\j in {8/6.92,14/3.46,8/0}
        {
        \draw (\i-1,\j+1.73) node[shape=circle,fill=white,font={\Huge}]{\color{blue}$3$};
        }
        \foreach \i/\j in {8/10.4,14/6.92,8/3.46,2/0}
        {
        \draw (\i,\j) node[shape=circle,fill=white,font={\Huge}]{\color{blue}$1$};
        }
        \foreach \i/\j in {12/10.4,6/6.92,12/3.46,6/0}
        {
        \draw (\i,\j) node[shape=circle,fill=white,font={\Huge}]{\color{blue}$2$};
        }
        \foreach \i/\j in {16/10.4,10/6.92,4/3.46,10/0}
        {
        \draw (\i,\j) node[shape=circle,fill=white,font={\Huge}]{\color{blue}$3$};
        }
    \end{tikzpicture}
    }
\end{center}
This gives a solution to 1-2-3 problem for the periodic tiling of the plane with equilateral triangles.

\end{proposition}

\begin{proof}
The sums of the four vertices within the rhombus supertile are simple to calculate and will remain unchanged as we tile the plane. The four corners of the supertile will sum to 12 as they join. The vertices along the edge of the supertile between the corners will sum up to 9 and 15 in that order.   

 \tikzset{rhombus/.pic={
        \draw[line width=4mm] (0,0)--(6,10.4)--(18,10.4)--(12,0)--(0,0);
        \draw[line width=0.7mm] (4,6.92)--(16,6.92);
        \draw[line width=0.7mm] (2,3.46)--(14,3.46);
        \draw[line width=0.7mm] (2,3.46)--(4,0)--(10,10.4)--(14,3.46);
        \draw[line width=0.7mm] (4,6.92)--(8,0)--(14,10.4)--(16,6.92);
        \draw[line width=0.7mm] (6,10.4)--(12,0);
        \foreach \i/\j in {4/6.92,10/3.44,4/0,16/6.92}
        {
        \draw (\i+1,\j+1.73) node[shape=circle,fill=white,font={\Huge}]{\color{blue}$1$};
        }
        \foreach \i/\j in {2/3.46,8/6.92,8/0,14/3.46}
        {
        \draw (\i+1,\j+1.73) node[shape=circle,fill=white,font={\Huge}]{\color{blue}$2$};
        }
        \foreach \i/\j in {0/0,12/6.92,6/3.46,12/0}
        {
        \draw (\i+1,\j+1.73) node[shape=circle,fill=white,font={\Huge}]{\color{blue}$3$};
        }
        \foreach \i/\j in {12/6.92,6/3.46,16/6.92,12/0}
        {
        \draw (\i-1,\j+1.73) node[shape=circle,fill=white,font={\Huge}]{\color{blue}$1$};
        }
        \foreach \i/\j in {16/6.92,10/3.46,4/0}
        {
        \draw (\i-1,\j+1.73) node[shape=circle,fill=white,font={\Huge}]{\color{blue}$2$};
        }
        \foreach \i/\j in {8/6.92,14/3.46,8/0}
        {
        \draw (\i-1,\j+1.73) node[shape=circle,fill=white,font={\Huge}]{\color{blue}$3$};
        }
        \foreach \i/\j in {8/10.4,14/6.92,8/3.46,2/0}
        {
        \draw (\i,\j) node[shape=circle,fill=white,font={\Huge}]{\color{blue}$1$};
        }
        \foreach \i/\j in {12/10.4,6/6.92,12/3.46,6/0}
        {
        \draw (\i,\j) node[shape=circle,fill=white,font={\Huge}]{\color{blue}$2$};
        }
        \foreach \i/\j in {16/10.4,10/6.92,4/3.46,10/0}
        {
        \draw (\i,\j) node[shape=circle,fill=white,font={\Huge}]{\color{blue}$3$};
        }
        \foreach \i/\j in {6/10.4,18/10.4,12/6.92,6/3.46,0/0,12/0}
        {
        \node[shape=circle, draw=red, fill=white, text=red, font={\Huge}] at (\i,\j) {12};
        }
        \foreach \i/\j in {10/10.4,4/6.92,16/6.92,10/3.46,4/0}
        {
        \node[shape=circle, draw=red, fill=white, text=red, font={\Huge}] at (\i,\j) {9};
        }
        \foreach \i/\j in {14/10.4,8/6.92,2/3.46,14/3.46,8/0}
        {
        \node[shape=circle, draw=red, fill=white, text=red, font={\Huge}] at (\i,\j) {15};
        }
    }}

\begin{center}

    \resizebox{270pt}{!}{ 
    \tikz{
    \clip (18.5,9.5) rectangle (47,28);
    \foreach \i/\j in {0/0,14.4/0,28.8/0,7.2/12.48,21.6/12.48,36/12.48,14.4/24.96,28.8/24.96,43.2/24.96}
    {\pic[scale=1.2] at (\i,\j) {rhombus};}
    }
    }
\end{center}
\end{proof}

We notice that the sums in each row alternate between 12, 9 and 15. Also, each triangle has a vertex of sum 12, 9, and 15. Therefore, each sum of 12, 9, and 15 occurs with a frequency of 1/3.

\subsection{Hexagon}
\begin{proposition}
    Consider the periodic tiling of the plane with the following honeycomb supertile consisting of 7 hexagons. Assign 1-2-3 to the edges as below.
    
\begin{center}
\resizebox{150pt}{!}{
    \begin{tikzpicture}[scale=1.1]
    \draw[line width=2mm] (-6,-5.2)--(-7.5,-2.6)--(-6,0)--(-7.5,2.6)--(-6,5.2)--(-3,5.2)--(-1.5,7.8)--(1.5,7.8)--(3,5.2)--(6,5.2)--(7.5,2.6)--(6,0)--(7.5,-2.6)--(6,-5.2)--(3,-5.2)--(1.5,-7.8)--(-1.5,-7.8)--(-3,-5.2)--(-6,-5.2);
    \draw[line width=0.2mm] (-3,5.2)--(-1.5,2.6)--(1.5,2.6)--(3,5.2);
    \draw[line width=0.2mm] (1.5,2.6)--(3,0)--(6,0);
    \draw[line width=0.2mm] (3,0)--(1.5,-2.6)--(3,-5.2);
    \draw[line width=0.2mm] (1.5,-2.6)--(-1.5,-2.6)--(-3,-5.2);
    \draw[line width=0.2mm] (-1.5,-2.6)--(-3,0)--(-6,0);
    \draw[line width=0.2mm] (-3,0)--(-1.5,2.6);
    \foreach \i/\j in {0/7.8,2.25/6.5,4.5/5.2,6.75/3.9,6.75/1.3,6.75/-1.3,6.75/-3.9,4.5/-5.2,2.25/-6.5,0/-7.8,-2.25/-6.5,-4.5/-5.2,-6.75/-3.9,-6.75/-1.3,-6.75/1.3,-6.75/3.9,-4.5/5.2,-2.25/6.5}
    {
    \draw (\i,\j) node[shape=circle,fill=white,font={\Huge}] {\color{blue}$1$};
    }
    \foreach \i/\j in {0/2.6,2.25/1.3,2.25/-1.3,0/-2.6,-2.25/-1.3,-2.25/1.3,
    -2.25/3.9,4.5/0,-2.25/-3.9}
    {
    \draw (\i,\j) node[shape=circle,fill=white,font={\Huge}] {\color{blue}$2$};
    }
    \foreach \i/\j in {2.25/3.9,2.25/-3.9,-4.5/0}
    {
    \draw (\i,\j) node[shape=circle,fill=white,font={\Huge}] {\color{blue}$3$};
    }
    \end{tikzpicture}
    }
\end{center}

This gives a solution to the 1-2-3 problem for the periodic tiling of the plane with hexagons.
\end{proposition}

\begin{proof}
    The internal vertices have alternating sums of 6 and 7. When we combine the supertiles along all edges the external supertile vertices will have either 3 external supertile edges or 2 external supertile edges and a single internal edge. The vertices adjacent to only other supertile edges will sum up to 3. The internal edges within the supertile that extend to the external vertices have alternate weights of 2 and 3. The external supertile vertices that are adjacent to the edges with a weighting of 2 will sum up to 4 and the vertices adjacent to the edges with a weighting of 3 will sum up to 5.

 \tikzset{honeycomb/.pic={
    \draw[line width=3.5mm] (-6,-5.2)--(-7.5,-2.6)--(-6,0)--(-7.5,2.6)--(-6,5.2)--(-3,5.2)--(-1.5,7.8)--(1.5,7.8)--(3,5.2)--(6,5.2)--(7.5,2.6)--(6,0)--(7.5,-2.6)--(6,-5.2)--(3,-5.2)--(1.5,-7.8)--(-1.5,-7.8)--(-3,-5.2)--(-6,-5.2);
    \draw[line width=0.2mm] (-3,5.2)--(-1.5,2.6)--(1.5,2.6)--(3,5.2);
    \draw[line width=0.2mm] (1.5,2.6)--(3,0)--(6,0);
    \draw[line width=0.2mm] (3,0)--(1.5,-2.6)--(3,-5.2);
    \draw[line width=0.2mm] (1.5,-2.6)--(-1.5,-2.6)--(-3,-5.2);
    \draw[line width=0.2mm] (-1.5,-2.6)--(-3,0)--(-6,0);
    \draw[line width=0.2mm] (-3,0)--(-1.5,2.6);
    \foreach \i/\j in {0/7.8,2.25/6.5,4.5/5.2,6.75/3.9,6.75/1.3,6.75/-1.3,6.75/-3.9,4.5/-5.2,2.25/-6.5,0/-7.8,-2.25/-6.5,-4.5/-5.2,-6.75/-3.9,-6.75/-1.3,-6.75/1.3,-6.75/3.9,-4.5/5.2,-2.25/6.5}
    {
    \draw (\i,\j) node[shape=circle,fill=white,font={\Huge}] {\color{blue}$1$};
    }
    \foreach \i/\j in {0/2.6,2.25/1.3,2.25/-1.3,0/-2.6,-2.25/-1.3,-2.25/1.3,
    -2.25/3.9,4.5/0,-2.25/-3.9}
    {
    \draw (\i,\j) node[shape=circle,fill=white,font={\Huge}] {\color{blue}$2$};
    }
    \foreach \i/\j in {2.25/3.9,2.25/-3.9,-4.5/0}
    {
    \draw (\i,\j) node[shape=circle,fill=white,font={\Huge}] {\color{blue}$3$};
    }  
    \foreach \i/\j in {-1.5/2.6,3/0,-1.5/-2.6}
    {   \node[shape=circle,draw=red,fill=white,text=red,font={\Huge}] at (\i,\j){6};
    }
    \foreach \i/\j in {1.5/2.6,1.5/-2.6,-3/0}
    {   \node[shape=circle,draw=red,fill=white,text=red,font={\Huge}] at (\i,\j){7};
    }
    \foreach \i/\j in {-1.5/7.8,6/5.2,7.5/-2.6,1.5/-7.8,-6/-5.2,-7.5/2.6}
    {   \node[shape=circle,draw=red,fill=white,text=red,font={\Huge}] at (\i,\j){3};
    }
    \foreach \i/\j in {1.5/7.8,6/0,6/-5.2,-3/-5.2,-7.5/-2.6,-3/5.2}
    {   \node[shape=circle,draw=red,fill=white,text=red,font={\Huge}] at (\i,\j){4};
    }
    \foreach \i/\j in {3/5.2,7.5/2.6,3/-5.2,-1.5/-7.8,-6/0,-6/5.2}
    {   \node[shape=circle,draw=red,fill=white,text=red,font={\Huge}] at (\i,\j){5};
    }
    }}

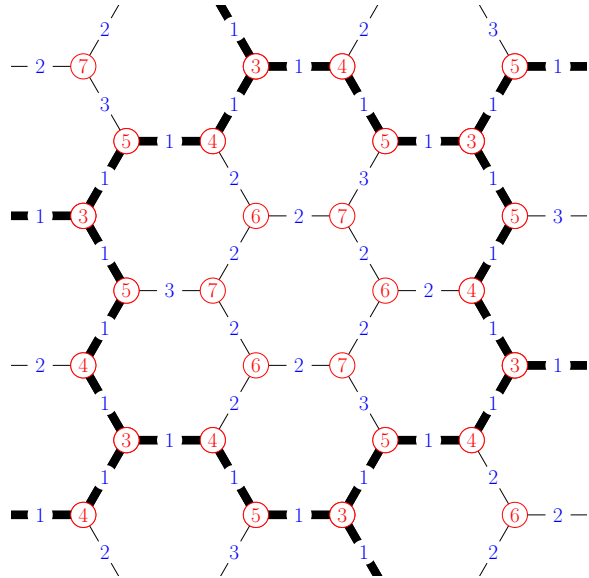
\begin{center}
    \resizebox{220pt}{!}{ 
    \tikz{
    \clip (-12,-11.9) rectangle (12,11.9);
    \foreach \i/\j in {0/0, 5.4/15.6, 16.2/3.12, 10.8/-12.48, -5.4/-15.6, -16.2/-3.12, -10.8/12.48}
    {\pic[scale=1.2] at (\i,\j) {honeycomb};}
    }
    }
\end{center}
\end{proof}

\tikzset{honeycombfreq/.pic={
    \draw[line width=3.5mm] (-6,-5.2)--(-7.5,-2.6)--(-6,0)--(-7.5,2.6)--(-6,5.2)--(-3,5.2)--(-1.5,7.8)--(1.5,7.8)--(3,5.2)--(6,5.2)--(7.5,2.6)--(6,0)--(7.5,-2.6)--(6,-5.2)--(3,-5.2)--(1.5,-7.8)--(-1.5,-7.8)--(-3,-5.2)--(-6,-5.2);
    \draw[line width=0.2mm] (-3,5.2)--(-1.5,2.6)--(1.5,2.6)--(3,5.2);
    \draw[line width=0.2mm] (1.5,2.6)--(3,0)--(6,0);
    \draw[line width=0.2mm] (3,0)--(1.5,-2.6)--(3,-5.2);
    \draw[line width=0.2mm] (1.5,-2.6)--(-1.5,-2.6)--(-3,-5.2);
    \draw[line width=0.2mm] (-1.5,-2.6)--(-3,0)--(-6,0);
    \draw[line width=0.2mm] (-3,0)--(-1.5,2.6);
    \foreach \i/\j in {0/-7.8,-2.25/-6.5,-4.5/-5.2,-6.75/-3.9,-6.75/-1.3,-6.75/1.3,-6.75/3.9,-4.5/5.2,-2.25/6.5}
    {
    \draw (\i,\j) node[shape=circle,fill=white,font={\Huge}] {\color{blue}$1$};
    }
    \foreach \i/\j in {0/2.6,2.25/1.3,2.25/-1.3,0/-2.6,-2.25/-1.3,-2.25/1.3,
    -2.25/3.9,4.5/0,-2.25/-3.9}
    {
    \draw (\i,\j) node[shape=circle,fill=white,font={\Huge}] {\color{blue}$2$};
    }
    \foreach \i/\j in {2.25/3.9,2.25/-3.9,-4.5/0}
    {
    \draw (\i,\j) node[shape=circle,fill=white,font={\Huge}] {\color{blue}$3$};
    }  
    \foreach \i/\j in {-1.5/2.6,3/0,-1.5/-2.6}
    {   \node[shape=circle,draw=red,fill=white,text=red,font={\Huge}] at (\i,\j){6};
    }
    \foreach \i/\j in {1.5/2.6,1.5/-2.6,-3/0}
    {   \node[shape=circle,draw=red,fill=white,text=red,font={\Huge}] at (\i,\j){7};
    }
    \foreach \i/\j in {-6/-5.2,-7.5/2.6}
    {   \node[shape=circle,draw=red,fill=white,text=red,font={\Huge}] at (\i,\j){3};
    }
    \foreach \i/\j in {-3/-5.2,-7.5/-2.6,-3/5.2}
    {   \node[shape=circle,draw=red,fill=white,text=red,font={\Huge}] at (\i,\j){4};
    }
    \foreach \i/\j in {-1.5/-7.8,-6/0,-6/5.2}
    {   \node[shape=circle,draw=red,fill=white,text=red,font={\Huge}] at (\i,\j){5};
    }
    }}

When we partially label the honeycomb supertile, as seen below, we can join them such that all edges are labeled without overlap. Therefore, the sums 4, 5, 6, and 7 each occur at a frequency of 3/14 and the sum 3 occurs at a frequency of 1/7.
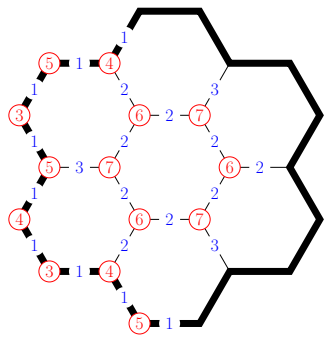
\begin{center}
    \hspace*{-1cm}
    \setlength{\tabcolsep}{10pt}
    \begin{tabular}{c c}
    \resizebox{120pt}{!}{
    \tikz[]{
    \pic[scale=1] at (0,0) {honeycombfreq};}}
    
    \end{tabular}\hspace*{-1cm}
\end{center}

\section{Non-periodic Tilings}

\subsection{Chair}
Consider the non-periodic tiling of the plane with the \href{https://tilings.math.uni-bielefeld.de/substitution/chair/}{chair} in \cite{TilingEncyclopedia}.
There are multiple ways in which to place the vertices for the chair tiling. We will look at three of them.

First, a short standard argument shows the following (see for example \cite{FKRS} for details).

\begin{proposition}\label{prop:chair-collared} The following are all level 1 collars in the chair tiling:

    \tikzset{chair1a/.pic=
{
        \draw[line width=0.7mm] (-3,3)--(0,3)--(0,0)--(3,0)--(3,-3)--(-3,-3)--(-3,3);
       
    }}

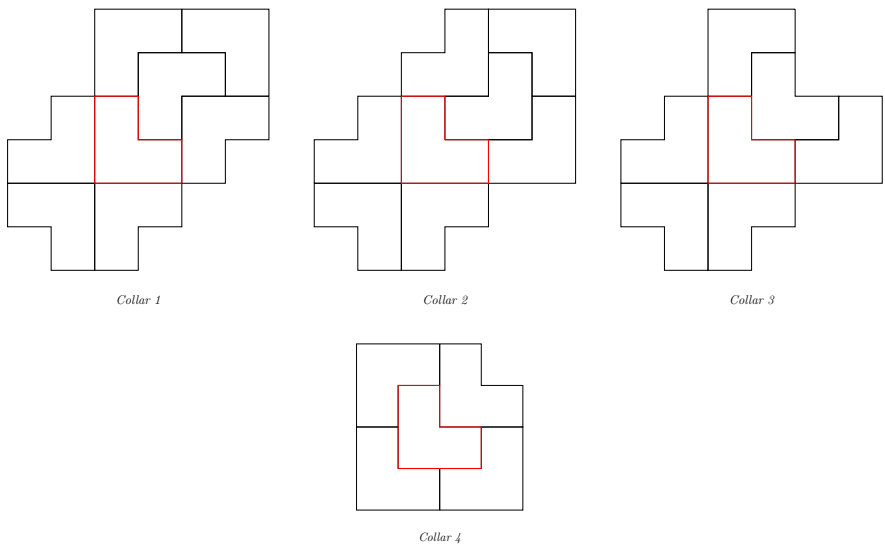
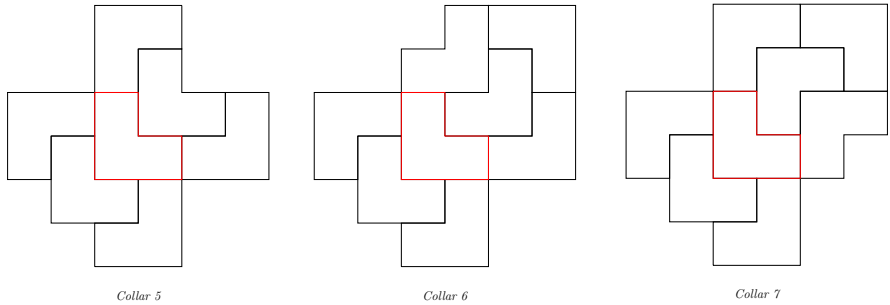
\begin{center}
    \hspace*{-3cm}
    \setlength{\tabcolsep}{8pt}
    \begin{longtable}[c]{c c c}
        
        \resizebox{100pt}{!}{
        \tikz{
        \pic[rotate=180] at (0,0) {chair1a};
        \pic[rotate=270] at (6,0) {chair1a};
         \pic[rotate=90] at (0,6) {chair1a};
        \pic[rotate=270] at (6,12) {chair1a};
        \pic[rotate=270] at (9,9) {chair1a};
        \pic[rotate=270] at (12,6) {chair1a};
        \pic[rotate=180] at (12,12) {chair1a};
        \pic[rotate=0,draw=red] at (6,6) {chair1a};
        

        \node[font={\Huge}, text=black] at (6,-5) {Collar 1};
        }}

&
        
        \resizebox{100pt}{!}{
        \tikz{
        \pic[rotate=180] at (0,0) {chair1a};
        \pic[rotate=270] at (6,0) {chair1a};
         \pic[rotate=90] at (0,6) {chair1a};
        \pic[rotate=90] at (6,12) {chair1a};
        \pic[rotate=90] at (9,9) {chair1a};
        \pic[rotate=90] at (12,6) {chair1a};
        \pic[rotate=180] at (12,12) {chair1a};
        \pic[rotate=0,draw=red] at (6,6) {chair1a};
        

        \node[font={\Huge}, text=black] at (6,-5) {Collar 2};
        }}

&
        
        \resizebox{100pt}{!}{
        \tikz{
        \pic[rotate=180] at (0,0) {chair1a};
        \pic[rotate=270] at (6,0) {chair1a};
         \pic[rotate=90] at (0,6) {chair1a};
        \pic[rotate=270] at (6,12) {chair1a};
        \pic[rotate=0] at (9,9) {chair1a};
        \pic[rotate=90] at (12,6) {chair1a};
        \pic[rotate=0,draw=red] at (6,6) {chair1a};
        

        \node[font={\Huge}, text=black] at (6,-5) {Collar 3};
        }}
       \bigskip \\

&
        
        \resizebox{64pt}{!}{
        \tikz{
        \pic[rotate=0] at (0,0) {chair1a};
        \pic[rotate=0] at (6,6) {chair1a};
         \pic[rotate=270] at (0,6) {chair1a};
        \pic[rotate=90] at (6,0) {chair1a};
        
        \pic[rotate=0,draw=red] at (3,3) {chair1a};
        

        \node[font={\Huge}, text=black] at (3,-5) {Collar 4};
        }}
       \bigskip\\

        
        \resizebox{100pt}{!}{
        \tikz{
        \pic[rotate=0] at (0,0) {chair1a};
        \pic[rotate=90] at (3,-3) {chair1a};
         \pic[rotate=270] at (-3,3) {chair1a};
        \pic[rotate=270] at (3,9) {chair1a};
        \pic[rotate=0] at (6,6) {chair1a};
        \pic[rotate=90] at (9,3) {chair1a};

        \pic[rotate=0,draw=red] at (3,3) {chair1a};
        

        \node[font={\Huge}, text=black] at (3,-8) {Collar 5};
        }}

& 
        
        \resizebox{100pt}{!}{
        \tikz{
        \pic[rotate=0] at (0,0) {chair1a};
        \pic[rotate=90] at (3,-3) {chair1a};
         \pic[rotate=270] at (-3,3) {chair1a};
        \pic[rotate=90] at (3,9) {chair1a};
        \pic[rotate=90] at (6,6) {chair1a};
        \pic[rotate=90] at (9,3) {chair1a};
        \pic[rotate=180] at (9,9) {chair1a};

        \pic[rotate=0,draw=red] at (3,3) {chair1a};
        

        \node[font={\Huge}, text=black] at (3,-8) {Collar 6};
        }}

& 
        
        \resizebox{100pt}{!}{
        \tikz{
        \pic[rotate=0] at (0,0) {chair1a};
        \pic[rotate=90] at (3,-3) {chair1a};
         \pic[rotate=270] at (-3,3) {chair1a};
        \pic[rotate=270] at (3,9) {chair1a};
        \pic[rotate=270] at (6,6) {chair1a};
        \pic[rotate=270] at (9,3) {chair1a};
        \pic[rotate=180] at (9,9) {chair1a};

        \pic[rotate=0,draw=red] at (3,3) {chair1a};
        

        \node[font={\Huge}, text=black] at (3,-8) {Collar 7};
        }}\\

    \end{longtable}
    \hspace*{-3cm}
\end{center}

\end{proposition}

\subsubsection{Vertex Placement 1}
The first vertex placement we will look at is where we place a vertex at every point on the grid so that each prototile has 8 vertices.
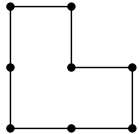
\begin{center}
\resizebox{50pt}{!}{
    \begin{tikzpicture}
        \draw[line width=0.7mm] (-3,3)--(0,3)--(0,0)--(3,0)--(3,-3)--(-3,-3)--(-3,3);

        \foreach \i/\j in {-3/-3,-3/0,-3/3,0/3,0/0,3/0,3/-3,0/-3}
            \node[shape=circle, draw=black, fill=black,font={\Huge}] at (\i,\j){$ $};
    \end{tikzpicture}}
\end{center}

\begin{proposition}
    Each vertex has a different degree than its adjacent vertices. Thus, by Lemma 3.1, assigning all edges a constant weight (1, 2, or 3) will provide a solution to the 1-2-3 problem.
\end{proposition}

\begin{proof}

    By Proposition~\ref{prop:chair-collared} there are 7 possible collared prototiles for the chair tiling. By inspection we see that for all collared tiles each vertex has a degree distinct from its adjacent vertices. The degrees are listed in the picture below. Thus, the result follows.

    \tikzset{chair1a/.pic=
{
        \draw[line width=0.7mm] (-3,3)--(0,3)--(0,0)--(3,0)--(3,-3)--(-3,-3)--(-3,3);
       
        \foreach \i/\j in {-3/-3,-3/0,-3/3,0/3,0/0,3/0,3/-3,0/-3}
            \node[shape=circle, draw=black, fill=black,font={\Huge}] at (\i,\j){$ $};
    }}

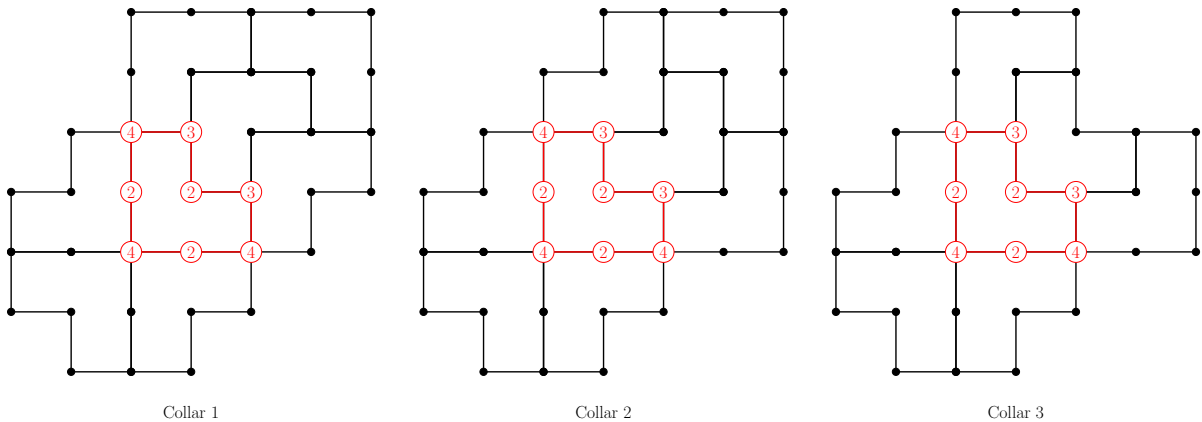
\begin{center}
    \hspace*{-3cm}
    \setlength{\tabcolsep}{8pt}
    \begin{longtable}[c]{c c c}
        
        \resizebox{140pt}{!}{
        \tikz{
        \pic[rotate=180] at (0,0) {chair1a};
        \pic[rotate=270] at (6,0) {chair1a};
         \pic[rotate=90] at (0,6) {chair1a};
        \pic[rotate=270] at (6,12) {chair1a};
        \pic[rotate=270] at (9,9) {chair1a};
        \pic[rotate=270] at (12,6) {chair1a};
        \pic[rotate=180] at (12,12) {chair1a};
        \pic[rotate=0,draw=red] at (6,6) {chair1a};
        
        \foreach \i/\j in {3/3,3/9,9/3}
     {\node[shape=circle, draw=red, fill=white, text=red,font={\Huge}] at (\i,\j) {4};
        }
        \foreach \i/\j in {9/6,6/9}
      {\node[shape=circle, draw=red, fill=white, text=red,font={\Huge}] at (\i,\j) {3};
        }
        \foreach \i/\j in {3/6,6/6,6/3}
     {\node[shape=circle, draw=red, fill=white, text=red,font={\Huge}] at (\i,\j) {2};
        }

        \node[font={\Huge}, text=black] at (6,-5) {Collar 1};
        }}

&
        
        \resizebox{140pt}{!}{
        \tikz{
        \pic[rotate=180] at (0,0) {chair1a};
        \pic[rotate=270] at (6,0) {chair1a};
         \pic[rotate=90] at (0,6) {chair1a};
        \pic[rotate=90] at (6,12) {chair1a};
        \pic[rotate=90] at (9,9) {chair1a};
        \pic[rotate=90] at (12,6) {chair1a};
        \pic[rotate=180] at (12,12) {chair1a};
        \pic[rotate=0,draw=red] at (6,6) {chair1a};
        
  \foreach \i/\j in {3/3,3/9,9/3}
     {\node[shape=circle, draw=red, fill=white, text=red,font={\Huge}] at (\i,\j) {4};
        }
        \foreach \i/\j in {9/6,6/9}
      {\node[shape=circle, draw=red, fill=white, text=red,font={\Huge}] at (\i,\j) {3};
        }
        \foreach \i/\j in {3/6,6/6,6/3}
     {\node[shape=circle, draw=red, fill=white, text=red,font={\Huge}] at (\i,\j) {2};
        }
        
        \node[font={\Huge}, text=black] at (6,-5) {Collar 2};
        }}

&
        
        \resizebox{140pt}{!}{
        \tikz{
        \pic[rotate=180] at (0,0) {chair1a};
        \pic[rotate=270] at (6,0) {chair1a};
         \pic[rotate=90] at (0,6) {chair1a};
        \pic[rotate=270] at (6,12) {chair1a};
        \pic[rotate=0] at (9,9) {chair1a};
        \pic[rotate=90] at (12,6) {chair1a};
        \pic[rotate=0,draw=red] at (6,6) {chair1a};
        
  \foreach \i/\j in {3/3,3/9,9/3}
     {\node[shape=circle, draw=red, fill=white, text=red,font={\Huge}] at (\i,\j) {4};
        }
        \foreach \i/\j in {9/6,6/9}
      {\node[shape=circle, draw=red, fill=white, text=red,font={\Huge}] at (\i,\j) {3};
        }
        \foreach \i/\j in {3/6,6/6,6/3}
     {\node[shape=circle, draw=red, fill=white, text=red,font={\Huge}] at (\i,\j) {2};
        }

        \node[font={\Huge}, text=black] at (6,-5) {Collar 3};
        }}
       \bigskip \\

&
        
        \resizebox{90pt}{!}{
        \tikz{
        \pic[rotate=0] at (0,0) {chair1a};
        \pic[rotate=0] at (6,6) {chair1a};
         \pic[rotate=270] at (0,6) {chair1a};
        \pic[rotate=90] at (6,0) {chair1a};
        
        \pic[rotate=0,draw=red] at (3,3) {chair1a};
        
        \foreach \i/\j in {}
        {\node[shape=circle, draw=red, fill=white, text=red,font={\Huge}] at (\i,\j) {4};
        }
        \foreach \i/\j in {3/0,0/3,3/6,6/3}
        {\node[shape=circle, draw=red, fill=white, text=red,font={\Huge}] at (\i,\j) {3};
        }
        \foreach \i/\j in {3/3,6/0,0/0,0/6}
        {\node[shape=circle, draw=red, fill=white, text=red,font={\Huge}] at (\i,\j) {2};
        }
        
        \node[font={\Huge}, text=black] at (3,-5) {Collar 4};
        }}
       \bigskip\\

        
        \resizebox{140pt}{!}{
        \tikz{
        \pic[rotate=0] at (0,0) {chair1a};
        \pic[rotate=90] at (3,-3) {chair1a};
         \pic[rotate=270] at (-3,3) {chair1a};
        \pic[rotate=270] at (3,9) {chair1a};
        \pic[rotate=0] at (6,6) {chair1a};
        \pic[rotate=90] at (9,3) {chair1a};

        \pic[rotate=0,draw=red] at (3,3) {chair1a};
        
        \foreach \i/\j in {0/6,6/0}
          {\node[shape=circle, draw=red, fill=white, text=red,font={\Huge}] at (\i,\j) {4};
        }
        \foreach \i/\j in {0/3,3/6,6/3,3/0}
        {\node[shape=circle, draw=red, fill=white, text=red,font={\Huge}] at (\i,\j) {3};
        }
        \foreach \i/\j in {0/0,3/3}
          {\node[shape=circle, draw=red, fill=white, text=red,font={\Huge}] at (\i,\j) {2};
        }

        \node[font={\Huge}, text=black] at (3,-8) {Collar 5};
        }}

& 
        
        \resizebox{140pt}{!}{
        \tikz{
        \pic[rotate=0] at (0,0) {chair1a};
        \pic[rotate=90] at (3,-3) {chair1a};
         \pic[rotate=270] at (-3,3) {chair1a};
        \pic[rotate=90] at (3,9) {chair1a};
        \pic[rotate=90] at (6,6) {chair1a};
        \pic[rotate=90] at (9,3) {chair1a};
        \pic[rotate=180] at (9,9) {chair1a};

        \pic[rotate=0,draw=red] at (3,3) {chair1a};
        
        \foreach \i/\j in {0/6,6/0}
       {\node[shape=circle, draw=red, fill=white, text=red,font={\Huge}] at (\i,\j) {4};
        }
        \foreach \i/\j in {0/3,3/6,6/3,3/0}
         {\node[shape=circle, draw=red, fill=white, text=red,font={\Huge}] at (\i,\j) {3};
        }
        \foreach \i/\j in {0/0,3/3}
          {\node[shape=circle, draw=red, fill=white, text=red,font={\Huge}] at (\i,\j) {2};
        }
        
        \node[font={\Huge}, text=black] at (3,-8) {Collar 6};
        }}

& 
        
        \resizebox{140pt}{!}{
        \tikz{
        \pic[rotate=0] at (0,0) {chair1a};
        \pic[rotate=90] at (3,-3) {chair1a};
         \pic[rotate=270] at (-3,3) {chair1a};
        \pic[rotate=270] at (3,9) {chair1a};
        \pic[rotate=270] at (6,6) {chair1a};
        \pic[rotate=270] at (9,3) {chair1a};
        \pic[rotate=180] at (9,9) {chair1a};

        \pic[rotate=0,draw=red] at (3,3) {chair1a};
        
        \foreach \i/\j in {0/6,6/0}
        {\node[shape=circle, draw=red, fill=white, text=red,font={\Huge}] at (\i,\j) {4};
        }
        \foreach \i/\j in {0/3,3/6,6/3,3/0}
          {\node[shape=circle, draw=red, fill=white, text=red,font={\Huge}] at (\i,\j) {3};
        }
        \foreach \i/\j in {0/0,3/3}
          {\node[shape=circle, draw=red, fill=white, text=red,font={\Huge}] at (\i,\j) {2};
        }

        \node[font={\Huge}, text=black] at (3,-8) {Collar 7};
        }}\\

    \end{longtable}
    \hspace*{-3cm}
\end{center}
\end{proof}

\subsubsection{Vertex Placement 2}
The next vertex placement we will look at is where we place vertices where edges meet at a 90$^{\circ}$ angle. By inspecting, we see that each tile has exactly six or eight vertices:

\begin{lemma} There are two prototiles in this tiling:
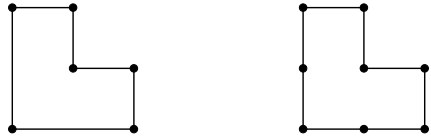
\tikzset{chair2aa/.pic=
{
        \draw[line width=0.7mm] (-3,3)--(0,3)--(0,0)--(3,0)--(3,-3)--(-3,-3)--(-3,3);

            
       
        \foreach \i/\j in {-3/-3,-3/3,0/3,0/0,3/0,3/-3}
            \node[shape=circle, draw=black, fill=black,font={\Huge}] at (\i,\j){$ $};
    }}

    \tikzset{chair2bb/.pic=
{
        \draw[line width=0.7mm] (-3,3)--(0,3)--(0,0)--(3,0)--(3,-3)--(-3,-3)--(-3,3);

       
        \foreach \i/\j in {-3/-3,-3/3,0/3,0/0,3/0,3/-3,-3/0,0/-3}
            \node[shape=circle, draw=black, fill=black,font={\Huge}] at (\i,\j){$ $};
    }}
\begin{center}
    \setlength{\tabcolsep}{30pt}
    \begin{tabular}{c c}
    \resizebox{50pt}{!}{
    \tikz[]{
    \pic{chair2aa}}}
         &  
    \resizebox{50pt}{!}{
    \tikz[]{
    \pic{chair2bb}}}\\ 
    \end{tabular}
\end{center}

The first prototile is at the center of Collar 1, Collar 2 and Collar 3, while the second prototile is at the center of Collar 4, Collar 5, Collar 6 and Collar 7.

\end{lemma}
\begin{proof}
We can check this for all collared tiles:

    \tikzset{chair1a/.pic=
{
        \draw[line width=0.7mm] (-3,3)--(0,3)--(0,0)--(3,0)--(3,-3)--(-3,-3)--(-3,3);
       
    }}

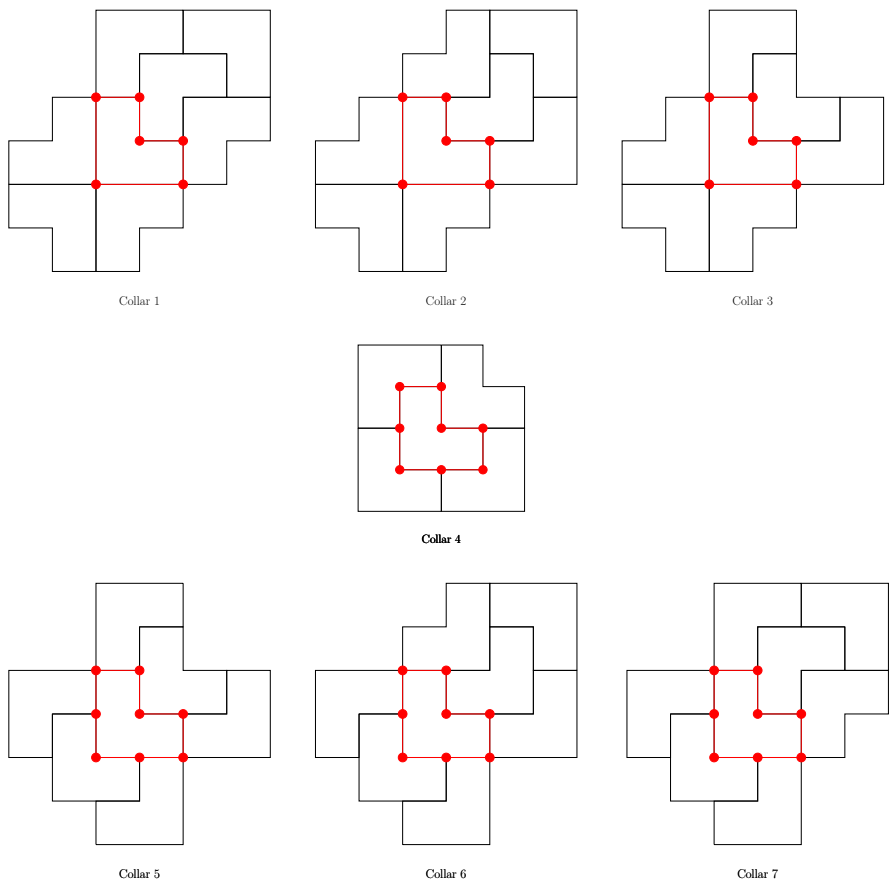
\begin{center}
    \hspace*{-3cm}
    \setlength{\tabcolsep}{8pt}
    \begin{longtable}[c]{c c c}
        
        \resizebox{100pt}{!}{
        \tikz{
        \pic[rotate=180] at (0,0) {chair1a};
        \pic[rotate=270] at (6,0) {chair1a};
         \pic[rotate=90] at (0,6) {chair1a};
        \pic[rotate=270] at (6,12) {chair1a};
        \pic[rotate=270] at (9,9) {chair1a};
        \pic[rotate=270] at (12,6) {chair1a};
        \pic[rotate=180] at (12,12) {chair1a};
        \pic[rotate=0,draw=red] at (6,6) {chair1a};
        
        \foreach \i/\j in {3/3,3/9,9/3}
        {\node[shape=circle, draw=red, fill=red, text=red,font={\Huge},minimum size=6mm] at (\i,\j) {};
        }
        \foreach \i/\j in {9/6,6/9}
        {\node[shape=circle, draw=red, fill=red, text=red,font={\Huge},minimum size=6mm] at (\i,\j) {};
        }
        \foreach \i/\j in {6/6}
        {\node[shape=circle, draw=red, fill=red, text=red,font={\Huge},minimum size=6mm] at (\i,\j) {};

        \node[font={\Huge}, text=black] at (6,-5) {Collar 1};
        }}}

&
        
        \resizebox{100pt}{!}{
        \tikz{
        \pic[rotate=180] at (0,0) {chair1a};
        \pic[rotate=270] at (6,0) {chair1a};
         \pic[rotate=90] at (0,6) {chair1a};
        \pic[rotate=90] at (6,12) {chair1a};
        \pic[rotate=90] at (9,9) {chair1a};
        \pic[rotate=90] at (12,6) {chair1a};
        \pic[rotate=180] at (12,12) {chair1a};
        \pic[rotate=0,draw=red] at (6,6) {chair1a};
        
        \foreach \i/\j in {3/3,3/9,9/3}
        {\node[shape=circle, draw=red, fill=red, text=red,font={\Huge},minimum size=6mm] at (\i,\j) {};
        }
        \foreach \i/\j in {9/6,6/9}
        {\node[shape=circle, draw=red, fill=red, text=red,font={\Huge},minimum size=6mm] at (\i,\j) {};
        }
        \foreach \i/\j in {6/6}
        {\node[shape=circle, draw=red, fill=red, text=red,font={\Huge},minimum size=6mm] at (\i,\j) {};

        \node[font={\Huge}, text=black] at (6,-5) {Collar 2};
        }}}

&
        
        \resizebox{100pt}{!}{
        \tikz{
        \pic[rotate=180] at (0,0) {chair1a};
        \pic[rotate=270] at (6,0) {chair1a};
         \pic[rotate=90] at (0,6) {chair1a};
        \pic[rotate=270] at (6,12) {chair1a};
        \pic[rotate=0] at (9,9) {chair1a};
        \pic[rotate=90] at (12,6) {chair1a};
        \pic[rotate=0,draw=red] at (6,6) {chair1a};
        
        \foreach \i/\j in {3/3,3/9,9/3}
        {\node[shape=circle, draw=red, fill=red, text=red,font={\Huge},minimum size=6mm] at (\i,\j) {};
        }
        \foreach \i/\j in {9/6,6/9}
        {\node[shape=circle, draw=red, fill=red, text=red,font={\Huge},minimum size=6mm] at (\i,\j) {};
        }
        \foreach \i/\j in {6/6}
        {\node[shape=circle, draw=red, fill=red, text=red,font={\Huge},minimum size=6mm] at (\i,\j) {};

        \node[font={\Huge}, text=black] at (6,-5) {Collar 3};
        }}}
       \bigskip \\

&
        
        \resizebox{64pt}{!}{
        \tikz{
        \pic[rotate=0] at (0,0) {chair1a};
        \pic[rotate=0] at (6,6) {chair1a};
         \pic[rotate=270] at (0,6) {chair1a};
        \pic[rotate=90] at (6,0) {chair1a};
        
        \pic[rotate=0,draw=red] at (3,3) {chair1a};
        
        \foreach \i/\j in {}
        {\node[shape=circle, draw=red, fill=red, text=red,font={\Huge},minimum size=6mm] at (\i,\j) {};
        }
        \foreach \i/\j in {3/0,0/3,3/6,6/3}
        {\node[shape=circle, draw=red, fill=red, text=red,font={\Huge},minimum size=6mm] at (\i,\j) {};
        }
        \foreach \i/\j in {3/3,6/0,0/0,0/6}
        {\node[shape=circle, draw=red, fill=red, text=red,font={\Huge},minimum size=6mm] at (\i,\j) {};

        \node[font={\Huge}, text=black] at (3,-5) {Collar 4};
        }}}
       \bigskip\\

        
        \resizebox{100pt}{!}{
        \tikz{
        \pic[rotate=0] at (0,0) {chair1a};
        \pic[rotate=90] at (3,-3) {chair1a};
         \pic[rotate=270] at (-3,3) {chair1a};
        \pic[rotate=270] at (3,9) {chair1a};
        \pic[rotate=0] at (6,6) {chair1a};
        \pic[rotate=90] at (9,3) {chair1a};

        \pic[rotate=0,draw=red] at (3,3) {chair1a};
        
        \foreach \i/\j in {0/6,6/0}
        {\node[shape=circle, draw=red, fill=red, text=red,font={\Huge},minimum size=6mm] at (\i,\j) {};
        }
        \foreach \i/\j in {0/3,3/6,6/3,3/0}
        {\node[shape=circle, draw=red, fill=red, text=red,font={\Huge},minimum size=6mm] at (\i,\j) {};
        }
        \foreach \i/\j in {0/0,3/3}
        {\node[shape=circle, draw=red, fill=red, text=red,font={\Huge},minimum size=6mm] at (\i,\j) {};

        \node[font={\Huge}, text=black] at (3,-8) {Collar 5};
        }}}

& 
        
        \resizebox{100pt}{!}{
        \tikz{
        \pic[rotate=0] at (0,0) {chair1a};
        \pic[rotate=90] at (3,-3) {chair1a};
         \pic[rotate=270] at (-3,3) {chair1a};
        \pic[rotate=90] at (3,9) {chair1a};
        \pic[rotate=90] at (6,6) {chair1a};
        \pic[rotate=90] at (9,3) {chair1a};
        \pic[rotate=180] at (9,9) {chair1a};

        \pic[rotate=0,draw=red] at (3,3) {chair1a};
        
        \foreach \i/\j in {0/6,6/0}
        {\node[shape=circle, draw=red, fill=red, text=red,font={\Huge},minimum size=6mm] at (\i,\j) {};
        }
        \foreach \i/\j in {0/3,3/6,6/3,3/0}
        {\node[shape=circle, draw=red, fill=red, text=red,font={\Huge},minimum size=6mm] at (\i,\j) {};
        }
        \foreach \i/\j in {0/0,3/3}
        {\node[shape=circle, draw=red, fill=red, text=red,font={\Huge},minimum size=6mm] at (\i,\j) {};

        \node[font={\Huge}, text=black] at (3,-8) {Collar 6};
        }}}

& 
        
        \resizebox{100pt}{!}{
        \tikz{
        \pic[rotate=0] at (0,0) {chair1a};
        \pic[rotate=90] at (3,-3) {chair1a};
         \pic[rotate=270] at (-3,3) {chair1a};
        \pic[rotate=270] at (3,9) {chair1a};
        \pic[rotate=270] at (6,6) {chair1a};
        \pic[rotate=270] at (9,3) {chair1a};
        \pic[rotate=180] at (9,9) {chair1a};

        \pic[rotate=0,draw=red] at (3,3) {chair1a};
        
        \foreach \i/\j in {0/6,6/0}
        {\node[shape=circle, draw=red, fill=red, text=red,font={\Huge},minimum size=6mm] at (\i,\j) {};
        }
        \foreach \i/\j in {0/3,3/6,6/3,3/0}
        {\node[shape=circle, draw=red, fill=red, text=red,font={\Huge},minimum size=6mm] at (\i,\j) {};
        }
        \foreach \i/\j in {0/0,3/3}
        {\node[shape=circle, draw=red, fill=red, text=red,font={\Huge},minimum size=6mm] at (\i,\j) {};

        \node[font={\Huge}, text=black] at (3,-8) {Collar 7};
        }}}\\

    \end{longtable}
    \hspace*{-3cm}
\end{center}
\end{proof}

\begin{proposition}
    Assigning the ``long" edges a weight of 2 and the ``short" edges a weight of 1 gives us a 1-2-3 solution.
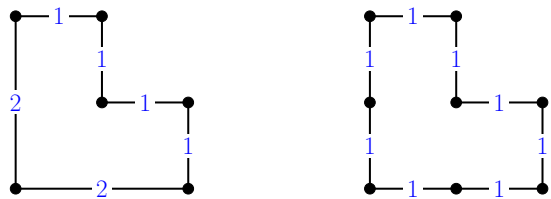
\tikzset{chair2a/.pic=
{
        \draw[line width=0.7mm] (-3,3)--(0,3)--(0,0)--(3,0)--(3,-3)--(-3,-3)--(-3,3);

        \foreach \i/\j in {-1.5/3, 0/1.5, 1.5/0, 3/-1.5}
            \draw (\i,\j) node[font={\Huge},fill=white] {\color{blue}$1$};
            
        \foreach \i/\j in {-3/0,0/-3}
            \draw (\i,\j) node[font={\Huge},fill=white] {\color{blue}$2$};
       
        \foreach \i/\j in {-3/-3,-3/3,0/3,0/0,3/0,3/-3}
            \node[shape=circle, draw=black, fill=black,font={\Huge}] at (\i,\j){$ $};
    }}

    \tikzset{chair2b/.pic=
{
        \draw[line width=0.7mm] (-3,3)--(0,3)--(0,0)--(3,0)--(3,-3)--(-3,-3)--(-3,3);

        \foreach \i/\j in {-3/-1.5, -3/1.5, -1.5/3, 0/1.5, 1.5/0, 3/-1.5, -1.5/-3, 1.5/-3}
            \draw (\i,\j) node[font={\Huge},fill=white] {\color{blue}$1$};
       
        \foreach \i/\j in {-3/-3,-3/3,0/3,0/0,3/0,3/-3,-3/0,0/-3}
            \node[shape=circle, draw=black, fill=black,font={\Huge}] at (\i,\j){$ $};
    }}

\tikzset{chair2c/.pic=
{
        \draw[line width=0.7mm] (-3,3)--(0,3)--(0,0)--(3,0)--(3,-3)--(-3,-3)--(-3,3);

        \foreach \i/\j in {-1.5/3, 0/1.5, 1.5/0, 3/-1.5}
            \draw (\i,\j) node[font={\Huge},fill=white] {\color{blue}$1$};
       
        \foreach \i/\j in {-3/-3,-3/3,0/3,0/0,3/0,3/-3}
            \node[shape=circle, draw=black, fill=black,font={\Huge}] at (\i,\j){$ $};
    }}
    
\begin{center}
    \setlength{\tabcolsep}{30pt}
    \begin{tabular}{c c}
    \resizebox{74pt}{!}{
    \tikz[]{
    \pic{chair2a}}}
         &  
    \resizebox{74pt}{!}{
    \tikz[]{
    \pic{chair2b}}}\\ 
    \end{tabular}
\end{center}

\end{proposition}

\begin{proof}
    We will use the 7 possible collared tiles to show that no conflicts exist.

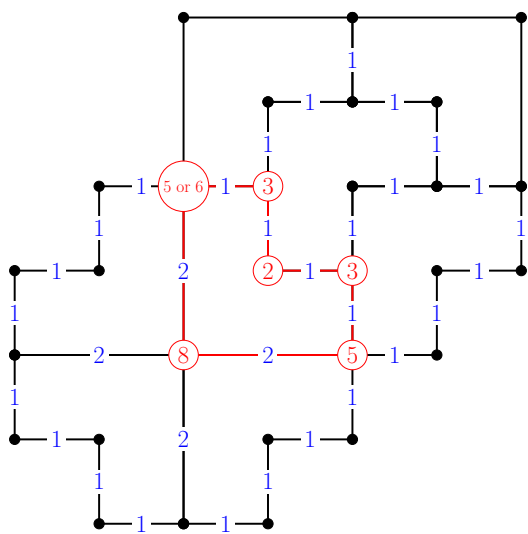
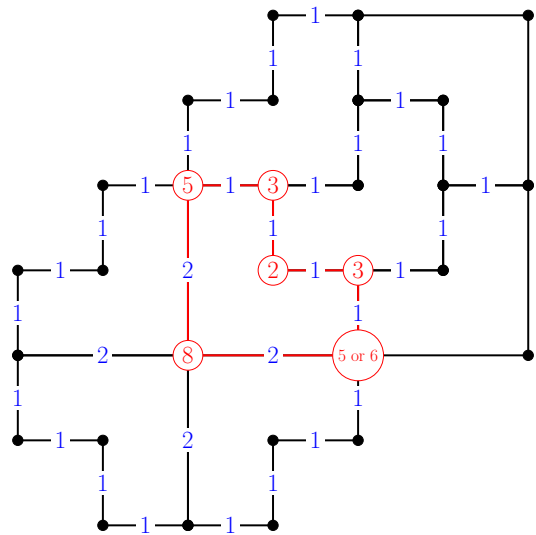
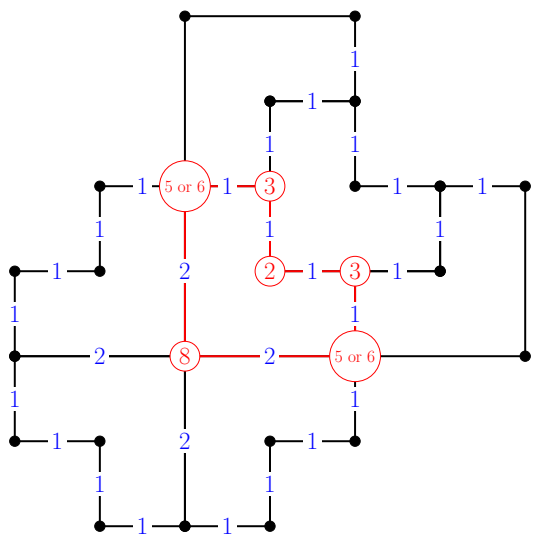
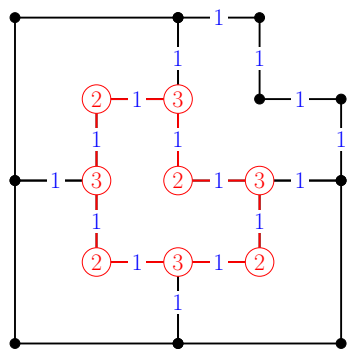
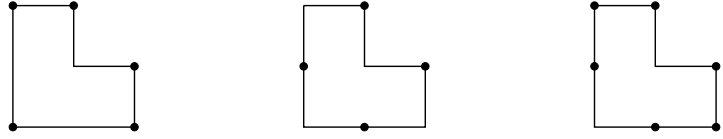
\tikzset{chair2a/.pic=
{
        \draw[line width=0.7mm] (-3,3)--(0,3)--(0,0)--(3,0)--(3,-3)--(-3,-3)--(-3,3);

        \foreach \i/\j in {-1.5/3, 0/1.5, 1.5/0, 3/-1.5}
            \draw (\i,\j) node[font={\Huge},fill=white] {\color{blue}$1$};
            
        \foreach \i/\j in {-3/0,0/-3}
            \draw (\i,\j) node[font={\Huge},fill=white] {\color{blue}$2$};
       
        \foreach \i/\j in {-3/-3,-3/3,0/3,0/0,3/0,3/-3}
            \node[shape=circle, draw=black, fill=black,font={\Huge}] at (\i,\j){$ $};
    }}

    \tikzset{chair2b/.pic=
{
        \draw[line width=0.7mm] (-3,3)--(0,3)--(0,0)--(3,0)--(3,-3)--(-3,-3)--(-3,3);

        \foreach \i/\j in {-3/-1.5, -3/1.5, -1.5/3, 0/1.5, 1.5/0, 3/-1.5, -1.5/-3, 1.5/-3}
            \draw (\i,\j) node[font={\Huge},fill=white] {\color{blue}$1$};
       
        \foreach \i/\j in {-3/-3,-3/3,0/3,0/0,3/0,3/-3,-3/0,0/-3}
            \node[shape=circle, draw=black, fill=black,font={\Huge}] at (\i,\j){$ $};
    }}

\tikzset{chair2c/.pic=
{
        \draw[line width=0.7mm] (-3,3)--(0,3)--(0,0)--(3,0)--(3,-3)--(-3,-3)--(-3,3);

        \foreach \i/\j in {-1.5/3, 0/1.5, 1.5/0, 3/-1.5}
            \draw (\i,\j) node[font={\Huge},fill=white] {\color{blue}$1$};
       
        \foreach \i/\j in {-3/-3,-3/3,0/3,0/0,3/0,3/-3}
            \node[shape=circle, draw=black, fill=black,font={\Huge}] at (\i,\j){$ $};
    }}

\begin{center}
    \setlength{\tabcolsep}{8pt}
    \begin{longtable}[c]{c c}
    \resizebox{200pt}{!}{
        \tikz{
        \pic[rotate=180] at (0,0) {chair2a};
        \pic[rotate=270] at (6,0) {chair2a};
         \pic[rotate=90] at (0,6) {chair2a};
        \pic[rotate=270] at (6,12) {chair2c};
        \pic[rotate=270] at (9,9) {chair2b};
        \pic[rotate=270] at (12,6) {chair2b};
        \pic[rotate=180] at (12,12) {chair2c};
        \pic[rotate=0,draw=red] at (6,6) {chair2a};
        
        \foreach \i/\j in {3/3}
        {\node[shape=circle, draw=red, fill=white, text=red,font={\Huge}] at (\i,\j) {8};
        }
        \foreach \i/\j in {9/6,6/9}
        {\node[shape=circle, draw=red, fill=white, text=red,font={\Huge}] at (\i,\j) {3};
        }
        \foreach \i/\j in {6/6}
        {\node[shape=circle, draw=red, fill=white, text=red,font={\Huge}] at (\i,\j) {2};
        \node[shape=circle,draw=red,fill=white,text=red,font={\Huge}] at (9,3) {5};
        \node[shape=circle,draw=red,fill=white,text=red,font={\LARGE}] at (3,9) {5 or 6};

        \node[font={\Huge}, text=black] at (6,-4.5) {Collar 1};
        }}}

         &  
    \resizebox{200pt}{!}{
        \tikz{
        \pic[rotate=180] at (0,0) {chair2a};
        \pic[rotate=270] at (6,0) {chair2a};
         \pic[rotate=90] at (0,6) {chair2a};
        \pic[rotate=90] at (6,12) {chair2b};
        \pic[rotate=90] at (9,9) {chair2b};
        \pic[rotate=90] at (12,6) {chair2c};
        \pic[rotate=180] at (12,12) {chair2c};
        \pic[rotate=0,draw=red] at (6,6) {chair2a};
        
        \foreach \i/\j in {3/3}
        {\node[shape=circle, draw=red, fill=white, text=red,font={\Huge}] at (\i,\j) {8};
        }
        \foreach \i/\j in {9/6,6/9}
        {\node[shape=circle, draw=red, fill=white, text=red,font={\Huge}] at (\i,\j) {3};
        }
        \foreach \i/\j in {6/6}
        {\node[shape=circle, draw=red, fill=white, text=red,font={\Huge}] at (\i,\j) {2};
        \node[shape=circle,draw=red,fill=white,text=red,font={\Huge}] at (3,9) {5};
        \node[shape=circle,draw=red,fill=white,text=red,font={\LARGE}] at (9,3) {5 or 6};

        \node[font={\Huge}, text=black] at (6,-4.5) {Collar 2};
        }}}
         \bigskip
         \\
    \resizebox{200pt}{!}{
        \tikz{
        \pic[rotate=180] at (0,0) {chair2a};
        \pic[rotate=270] at (6,0) {chair2a};
         \pic[rotate=90] at (0,6) {chair2a};
        \pic[rotate=270] at (6,12) {chair2c};
        \pic[rotate=0] at (9,9) {chair2b};
        \pic[rotate=90] at (12,6) {chair2c};
        \pic[rotate=0,draw=red] at (6,6) {chair2a};
        
        \foreach \i/\j in {3/3}
        {\node[shape=circle, draw=red, fill=white, text=red,font={\Huge}] at (\i,\j) {8};
        }
        \foreach \i/\j in {9/6,6/9}
        {\node[shape=circle, draw=red, fill=white, text=red,font={\Huge}] at (\i,\j) {3};
        }
        \foreach \i/\j in {6/6}
        {\node[shape=circle, draw=red, fill=white, text=red,font={\Huge}] at (\i,\j) {2};
        \node[shape=circle,draw=red,fill=white,text=red,font={\LARGE}] at (9,3) {5 or 6};
        \node[shape=circle,draw=red,fill=white,text=red,font={\LARGE}] at (3,9) {5 or 6};

        \node[font={\Huge}, text=black] at (6,-4.5) {Collar 3};
        }}}
         & 

    \resizebox{130pt}{!}{
        \tikz{
        \pic[rotate=0] at (0,0) {chair2c};
        \pic[rotate=0] at (6,6) {chair2b};
         \pic[rotate=270] at (0,6) {chair2c};
        \pic[rotate=90] at (6,0) {chair2c};
        
        \pic[rotate=0,draw=red] at (3,3) {chair2b};
        
        \foreach \i/\j in {}
        {\node[shape=circle, draw=red, fill=white, text=red,font={\Huge}] at (\i,\j) {4};
        }
        \foreach \i/\j in {3/0,0/3,3/6,6/3}
        {\node[shape=circle, draw=red, fill=white, text=red,font={\Huge}] at (\i,\j) {3};
        }
        \foreach \i/\j in {3/3,6/0,0/0,0/6}
        {\node[shape=circle, draw=red, fill=white, text=red,font={\Huge}] at (\i,\j) {2};

        \node[font={\Huge}, text=black] at (3,-4.5) {Collar 4};
        }}}
        \bigskip
        \\

    \resizebox{200pt}{!}{
        \tikz{
        \pic[rotate=0] at (0,0) {chair2b};
        \pic[rotate=90] at (3,-3) {chair2c};
        \pic[rotate=270] at (-3,3) {chair2c};
        \pic[rotate=270] at (3,9) {chair2c};
        \pic[rotate=0] at (6,6) {chair2b};
        \pic[rotate=90] at (9,3) {chair2c};

        \pic[rotate=0,draw=red] at (3,3) {chair2b};
        
        \foreach \i/\j in {0/6,6/0}
        {\node[shape=circle, draw=red, fill=white, text=red,font={\huge}] at (\i,\j) {$\geq$ 4};
        }
        \foreach \i/\j in {0/3,3/6,6/3,3/0}
        {\node[shape=circle, draw=red, fill=white, text=red,font={\Huge}] at (\i,\j) {3};
        }
        \foreach \i/\j in {0/0,3/3}
        {\node[shape=circle, draw=red, fill=white, text=red,font={\Huge}] at (\i,\j) {2};

        \node[font={\Huge}, text=black] at (3,-7.5) {Collar 5};
        }}}

        &
    \resizebox{200pt}{!}{
        \tikz{
        \pic[rotate=0] at (0,0) {chair2b};
        \pic[rotate=90] at (3,-3) {chair2c};
        \pic[rotate=270] at (-3,3) {chair2a};
        \pic[rotate=90] at (3,9) {chair2b};
        \pic[rotate=90] at (6,6) {chair2b};
        \pic[rotate=90] at (9,3) {chair2c};
        \pic[rotate=180] at (9,9) {chair2c};

        \pic[rotate=0,draw=red] at (3,3) {chair2b};
        
        \foreach \i/\j in {0/6}
        {\node[shape=circle, draw=red, fill=white, text=red,font={\Huge}] at (\i,\j) {5};
        }
        \foreach \i/\j in {0/3,3/6,6/3,3/0}
        {\node[shape=circle, draw=red, fill=white, text=red,font={\Huge}] at (\i,\j) {3};
        }
        \foreach \i/\j in {0/0,3/3}
        {\node[shape=circle, draw=red, fill=white, text=red,font={\Huge}] at (\i,\j) {2};
        \node[shape=circle,draw=red,fill=white,text=red,font={\huge}] at (6,0) {$\geq$ 4};

        \node[font={\Huge}, text=black] at (3,-7.5) {Collar 6};
        }}}\\

    \end{longtable}
\end{center}

    \begin{center}
    \resizebox{200pt}{!}{
       \tikz{
        \pic[rotate=0] at (0,0) {chair2b};
        \pic[rotate=90] at (3,-3) {chair2a};
        \pic[rotate=270] at (-3,3) {chair2c};
        \pic[rotate=270] at (3,9) {chair2c};
        \pic[rotate=270] at (6,6) {chair2b};
        \pic[rotate=270] at (9,3) {chair2b};
        \pic[rotate=180] at (9,9) {chair2c};

        \pic[rotate=0,draw=red] at (3,3) {chair2b};
        
        \foreach \i/\j in {6/0}
        {\node[shape=circle, draw=red, fill=white, text=red,font={\Huge}] at (\i,\j) {5};
        }
        \foreach \i/\j in {0/3,3/6,6/3,3/0}
        {\node[shape=circle, draw=red, fill=white, text=red,font={\Huge}] at (\i,\j) {3};
        }
        \foreach \i/\j in {0/0,3/3}
        {\node[shape=circle, draw=red, fill=white, text=red,font={\Huge}] at (\i,\j) {2};
        \node[shape=circle,draw=red,fill=white,text=red,font={\huge}] at (0,6) {$\geq$ 4};

        \node[font={\Huge}, text=black] at (3,-7.5) {Collar 7};
        }}}
        \end{center}
\end{proof}

\subsubsection{Vertex Placement 3}
The last vertex placement we will look at is where we place vertices where they will have a degree of 3 or higher. A short inspection shows that each prototile will have 4, 5, or 6 vertices, depending on the collar.

\begin{lemma} There are three prototiles in this tiling:
\tikzset{chair2aa/.pic=
{
        \draw[line width=0.7mm] (-3,3)--(0,3)--(0,0)--(3,0)--(3,-3)--(-3,-3)--(-3,3);

            
       
        \foreach \i/\j in {-3/-3,-3/3,0/3,3/0,3/-3}
            \node[shape=circle, draw=black, fill=black,font={\Huge}] at (\i,\j){$ $};
    }}

    \tikzset{chair2bb/.pic=
{
        \draw[line width=0.7mm] (-3,3)--(0,3)--(0,0)--(3,0)--(3,-3)--(-3,-3)--(-3,3);

       
        \foreach \i/\j in {0/3,3/0,-3/0,0/-3}
            \node[shape=circle, draw=black, fill=black,font={\Huge}] at (\i,\j){$ $};
    }}
    \tikzset{chair2cc/.pic=
{
        \draw[line width=0.7mm] (-3,3)--(0,3)--(0,0)--(3,0)--(3,-3)--(-3,-3)--(-3,3);

       
        \foreach \i/\j in {-3/3,0/3,3/0,3/-3,-3/0,0/-3}
            \node[shape=circle, draw=black, fill=black,font={\Huge}] at (\i,\j){$ $};
    }}
\begin{center}
    \setlength{\tabcolsep}{30pt}
    \begin{tabular}{c c c}
    \resizebox{50pt}{!}{
    \tikz[]{
    \pic{chair2aa}}}
         &  
    \resizebox{50pt}{!}{
    \tikz[]{
    \pic{chair2bb}}}         &  
    \resizebox{50pt}{!}{
    \tikz[]{
    \pic{chair2cc}}}\\ 
    \end{tabular}
\end{center}

The first prototile is at the center of Collar 1, Collar 2 and Collar 3, the second prototile is at the center of Collar 4 and the last is at the the center of Collar 5, Collar 6 and Collar 7.
\end{lemma}
\begin{proof}
We can check this for all collared tiles:

    \tikzset{chair1a/.pic=
{
        \draw[line width=0.7mm] (-3,3)--(0,3)--(0,0)--(3,0)--(3,-3)--(-3,-3)--(-3,3);
       
    }}

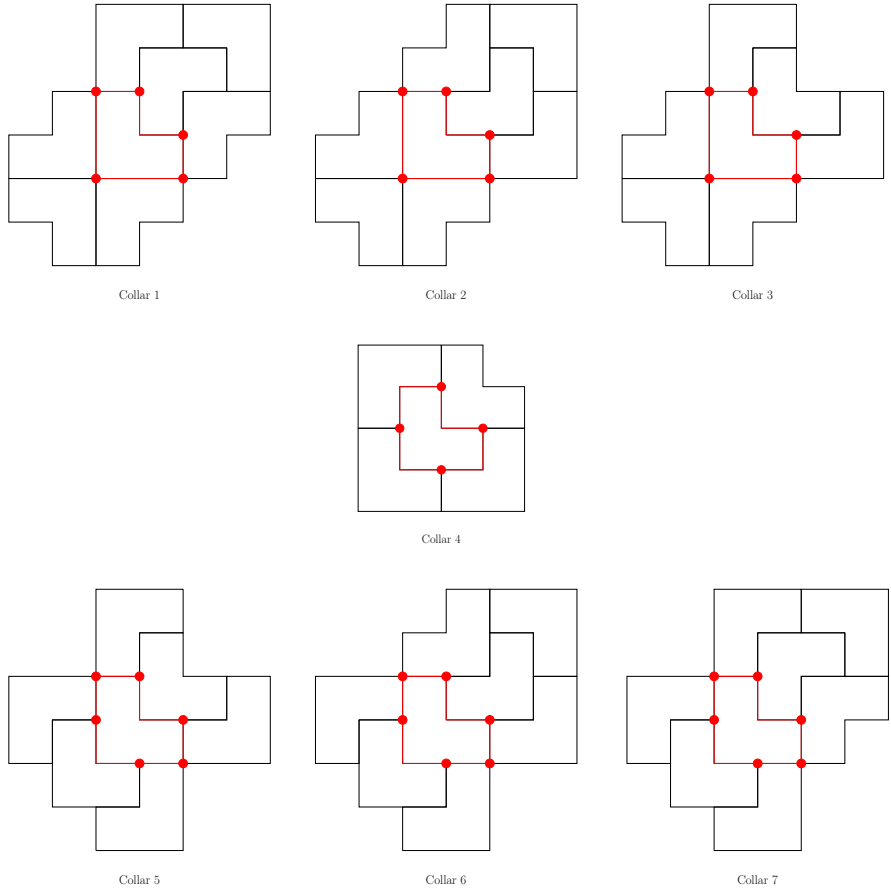
\begin{center}
    \hspace*{-3cm}
    \setlength{\tabcolsep}{8pt}
    \begin{longtable}[c]{c c c}
        
        \resizebox{100pt}{!}{
        \tikz{
        \pic[rotate=180] at (0,0) {chair1a};
        \pic[rotate=270] at (6,0) {chair1a};
         \pic[rotate=90] at (0,6) {chair1a};
        \pic[rotate=270] at (6,12) {chair1a};
        \pic[rotate=270] at (9,9) {chair1a};
        \pic[rotate=270] at (12,6) {chair1a};
        \pic[rotate=180] at (12,12) {chair1a};
        \pic[rotate=0,draw=red] at (6,6) {chair1a};
        
        \foreach \i/\j in {3/3,3/9,9/3}
        {\node[shape=circle, draw=red, fill=red, text=red,font={\Huge},minimum size=6mm] at (\i,\j) {};
        }
        \foreach \i/\j in {9/6,6/9}
        {\node[shape=circle, draw=red, fill=red, text=red,font={\Huge},minimum size=6mm] at (\i,\j) {};
        }

        \node[font={\Huge}, text=black] at (6,-5) {Collar 1};
        }}

&
        
        \resizebox{100pt}{!}{
        \tikz{
        \pic[rotate=180] at (0,0) {chair1a};
        \pic[rotate=270] at (6,0) {chair1a};
         \pic[rotate=90] at (0,6) {chair1a};
        \pic[rotate=90] at (6,12) {chair1a};
        \pic[rotate=90] at (9,9) {chair1a};
        \pic[rotate=90] at (12,6) {chair1a};
        \pic[rotate=180] at (12,12) {chair1a};
        \pic[rotate=0,draw=red] at (6,6) {chair1a};
        
        \foreach \i/\j in {3/3,3/9,9/3}
        {\node[shape=circle, draw=red, fill=red, text=red,font={\Huge},minimum size=6mm] at (\i,\j) {};
        }
        \foreach \i/\j in {9/6,6/9}
        {\node[shape=circle, draw=red, fill=red, text=red,font={\Huge},minimum size=6mm] at (\i,\j) {};
        }

        \node[font={\Huge}, text=black] at (6,-5) {Collar 2};
        }}

&
        
        \resizebox{100pt}{!}{
        \tikz{
        \pic[rotate=180] at (0,0) {chair1a};
        \pic[rotate=270] at (6,0) {chair1a};
         \pic[rotate=90] at (0,6) {chair1a};
        \pic[rotate=270] at (6,12) {chair1a};
        \pic[rotate=0] at (9,9) {chair1a};
        \pic[rotate=90] at (12,6) {chair1a};
        \pic[rotate=0,draw=red] at (6,6) {chair1a};
        
        \foreach \i/\j in {3/3,3/9,9/3}
        {\node[shape=circle, draw=red, fill=red, text=red,font={\Huge},minimum size=6mm] at (\i,\j) {};
        }
        \foreach \i/\j in {9/6,6/9}
        {\node[shape=circle, draw=red, fill=red, text=red,font={\Huge},minimum size=6mm] at (\i,\j) {};
        }

        \node[font={\Huge}, text=black] at (6,-5) {Collar 3};
        }}
       \bigskip \\

&
        
        \resizebox{64pt}{!}{
        \tikz{
        \pic[rotate=0] at (0,0) {chair1a};
        \pic[rotate=0] at (6,6) {chair1a};
         \pic[rotate=270] at (0,6) {chair1a};
        \pic[rotate=90] at (6,0) {chair1a};
        
        \pic[rotate=0,draw=red] at (3,3) {chair1a};
        
        \foreach \i/\j in {}
        {\node[shape=circle, draw=red, fill=red, text=red,font={\Huge},minimum size=6mm] at (\i,\j) {};
        }
        \foreach \i/\j in {3/0,0/3,3/6,6/3}
        {\node[shape=circle, draw=red, fill=red, text=red,font={\Huge},minimum size=6mm] at (\i,\j) {};
        }

        \node[font={\Huge}, text=black] at (3,-5) {Collar 4};
        }}
       \bigskip\\

        
        \resizebox{100pt}{!}{
        \tikz{
        \pic[rotate=0] at (0,0) {chair1a};
        \pic[rotate=90] at (3,-3) {chair1a};
         \pic[rotate=270] at (-3,3) {chair1a};
        \pic[rotate=270] at (3,9) {chair1a};
        \pic[rotate=0] at (6,6) {chair1a};
        \pic[rotate=90] at (9,3) {chair1a};

        \pic[rotate=0,draw=red] at (3,3) {chair1a};
        
        \foreach \i/\j in {0/6,6/0}
        {\node[shape=circle, draw=red, fill=red, text=red,font={\Huge},minimum size=6mm] at (\i,\j) {};
        }
        \foreach \i/\j in {0/3,3/6,6/3,3/0}
        {\node[shape=circle, draw=red, fill=red, text=red,font={\Huge},minimum size=6mm] at (\i,\j) {};
        }

        \node[font={\Huge}, text=black] at (3,-8) {Collar 5};
        }}

& 
        
        \resizebox{100pt}{!}{
        \tikz{
        \pic[rotate=0] at (0,0) {chair1a};
        \pic[rotate=90] at (3,-3) {chair1a};
         \pic[rotate=270] at (-3,3) {chair1a};
        \pic[rotate=90] at (3,9) {chair1a};
        \pic[rotate=90] at (6,6) {chair1a};
        \pic[rotate=90] at (9,3) {chair1a};
        \pic[rotate=180] at (9,9) {chair1a};

        \pic[rotate=0,draw=red] at (3,3) {chair1a};
        
        \foreach \i/\j in {0/6,6/0}
        {\node[shape=circle, draw=red, fill=red, text=red,font={\Huge},minimum size=6mm] at (\i,\j) {};
        }
        \foreach \i/\j in {0/3,3/6,6/3,3/0}
        {\node[shape=circle, draw=red, fill=red, text=red,font={\Huge},minimum size=6mm] at (\i,\j) {};
        }

        \node[font={\Huge}, text=black] at (3,-8) {Collar 6};
        }}

& 
        
        \resizebox{100pt}{!}{
        \tikz{
        \pic[rotate=0] at (0,0) {chair1a};
        \pic[rotate=90] at (3,-3) {chair1a};
         \pic[rotate=270] at (-3,3) {chair1a};
        \pic[rotate=270] at (3,9) {chair1a};
        \pic[rotate=270] at (6,6) {chair1a};
        \pic[rotate=270] at (9,3) {chair1a};
        \pic[rotate=180] at (9,9) {chair1a};

        \pic[rotate=0,draw=red] at (3,3) {chair1a};
        
        \foreach \i/\j in {0/6,6/0}
        {\node[shape=circle, draw=red, fill=red, text=red,font={\Huge},minimum size=6mm] at (\i,\j) {};
        }
        \foreach \i/\j in {0/3,3/6,6/3,3/0}
        {\node[shape=circle, draw=red, fill=red, text=red,font={\Huge},minimum size=6mm] at (\i,\j) {};
        }

        \node[font={\Huge}, text=black] at (3,-8) {Collar 7};
        }}\\

    \end{longtable}
    \hspace*{-3cm}
\end{center}
\end{proof}

\begin{proposition}
    For this vertex placement we will move up to the three possible level 1 supertiles. Assigning weights of 3 to all external edges and assigning weights to the internal edges as below gives us a 1-2-3 solution.

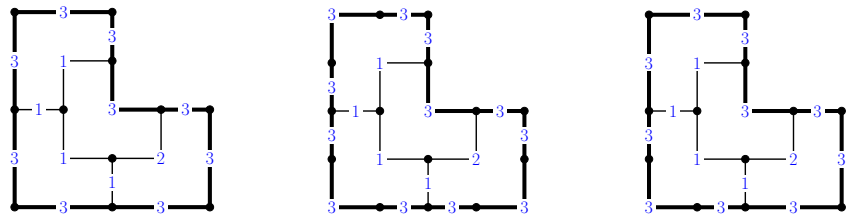
\begin{center}
    \setlength{\tabcolsep}{20pt}
    \begin{tabular}{c c c}
    \resizebox{80pt}{!}{
    \begin{tikzpicture}[scale=0.8]
        \draw[line width=2mm] (-6,-6)--(-6,6)--(0,6)--(0,0)--(6,0)--(6,-6)--(-6,-6);

        \draw[line width=0.7mm] (0,3)--(-3,3)--(-3,-3)--(3,-3)--(3,0);

        \draw[line width=0.7mm]
        (-6,0)--(-3,0);

        \draw[line width=0.7mm]
        (0,-3)--(0,-6);

        \foreach \i/\j in {-3/3,-3/-3,-4.5/0,0/-4.5}
            \draw (\i,\j) node[font={\Huge},fill=white] {\color{blue}$1$};
            
        \draw (3,-3) node[font={\Huge},fill=white]{\color{blue}$2$};

        \foreach \i/\j in {-3/6,-6/3,-6/-3,-3/-6,3/-6,6/-3,4.5/0,0/0,0/4.5}
            \draw (\i,\j) node[font={\Huge},fill=white] {\color{blue}$3$};
        
        \foreach \i/\j in {-6/-6,-6/0,-6/6,0/6,0/3,3/0,6/0,6/-6,0/-6,-3/0,0/-3}
            \node[shape=circle, draw=black, fill=black,font={\Huge}] at (\i,\j){$ $};
    \end{tikzpicture}}
         &  
    \resizebox{80pt}{!}{
    \begin{tikzpicture}[scale=0.8]
        \draw[line width=2mm] (-6,-6)--(-6,6)--(0,6)--(0,0)--(6,0)--(6,-6)--(-6,-6);

        \draw[line width=0.7mm] (0,3)--(-3,3)--(-3,-3)--(3,-3)--(3,0);

        \draw[line width=0.7mm]
        (-6,0)--(-3,0);

        \draw[line width=0.7mm]
        (0,-3)--(0,-6);

        \foreach \i/\j in {-3/3,-3/-3,-4.5/0,0/-4.5}
            \draw (\i,\j) node[font={\Huge},fill=white] {\color{blue}$1$};
            
        \draw (3,-3) node[font={\Huge},fill=white]{\color{blue}$2$};

        \foreach \i/\j in {-6/6,-1.5/6,-6/1.5,-6/-1.5,-6/-6,-1.5/-6,1.5/-6,6/-1.5,6/-6,4.5/0,0/0,0/4.5}
            \draw (\i,\j) node[font={\Huge},fill=white] {\color{blue}$3$};
        
        \foreach \i/\j in {-6/0,0/6,0/3,3/0,6/0,0/-6,-3/0,0/-3,-6/-3,-6/3,-3/6,6/-3,-3/-6,3/-6}
            \node[shape=circle, draw=black, fill=black,font={\Huge}] at (\i,\j){$ $};
            \end{tikzpicture}
    }

    &
    \resizebox{80pt}{!}{
    \begin{tikzpicture}[scale=0.8]
        \draw[line width=2mm] (-6,-6)--(-6,6)--(0,6)--(0,0)--(6,0)--(6,-6)--(-6,-6);

        \draw[line width=0.7mm] (0,3)--(-3,3)--(-3,-3)--(3,-3)--(3,0);

        \draw[line width=0.7mm]
        (-6,0)--(-3,0);

        \draw[line width=0.7mm]
        (0,-3)--(0,-6);

        \foreach \i/\j in {-3/3,-3/-3,-4.5/0,0/-4.5}
            \draw (\i,\j) node[font={\Huge},fill=white] {\color{blue}$1$};
            
        \draw (3,-3) node[font={\Huge},fill=white]{\color{blue}$2$};

        \foreach \i/\j in {-6/3,-3/6,-6/-1.5,-6/-6,-1.5/-6,3/-6,6/-3,4.5/0,0/0,0/4.5}
            \draw (\i,\j) node[font={\Huge},fill=white] {\color{blue}$3$};
        
        \foreach \i/\j in {-6/-3,-6/0,-6/6,0/6,0/3,3/0,6/0,6/-6,0/-6,-3/-6,-3/0,0/-3}
            \node[shape=circle, draw=black, fill=black,font={\Huge}] at (\i,\j){$ $};
            \end{tikzpicture}
    }
    \end{tabular}
\end{center}
\end{proposition}

\begin{proof}
    We will use the 7 collared tiles, with level 1 supertiles this time, to show that no conflicts exist.

\tikzset{chair3a/.pic=
{
        \draw[line width=2mm] (-6,-6)--(-6,6)--(0,6)--(0,0)--(6,0)--(6,-6)--(-6,-6);

        \draw[line width=0.7mm] (0,3)--(-3,3)--(-3,-3)--(3,-3)--(3,0);

        \draw[line width=0.7mm]
        (-6,0)--(-3,0);

        \draw[line width=0.7mm]
        (0,-3)--(0,-6);

        \foreach \i/\j in {-3/3,-3/-3,-4.5/0,0/-4.5}
            \draw (\i,\j) node[font={\Huge},fill=white] {\color{blue}$1$};
            
        \draw (3,-3) node[font={\Huge},fill=white]{\color{blue}$2$};

        \foreach \i/\j in {-3/6,-6/3,-6/-3,-3/-6,3/-6,6/-3,4.5/0,0/0,0/4.5}
            \draw (\i,\j) node[font={\Huge},fill=white] {\color{blue}$3$};
        
        \foreach \i/\j in {-6/-6,-6/0,-6/6,0/6,0/3,3/0,6/0,6/-6,0/-6,-3/0,0/-3}
            \node[shape=circle, draw=black, fill=black,font={\Huge}] at (\i,\j){$ $};
    }}

\tikzset{chair3b/.pic=
{
        \draw[line width=2mm] (-6,-6)--(-6,6)--(0,6)--(0,0)--(6,0)--(6,-6)--(-6,-6);

        \draw[line width=0.7mm] (0,3)--(-3,3)--(-3,-3)--(3,-3)--(3,0);

        \draw[line width=0.7mm]
        (-6,0)--(-3,0);

        \draw[line width=0.7mm]
        (0,-3)--(0,-6);

        \foreach \i/\j in {-3/3,-3/-3,-4.5/0,0/-4.5}
            \draw (\i,\j) node[font={\Huge},fill=white] {\color{blue}$1$};
            
        \draw (3,-3) node[font={\Huge},fill=white]{\color{blue}$2$};

        \foreach \i/\j in {-6/6,-1.5/6,-6/1.5,-6/-1.5,-6/-6,-1.5/-6,1.5/-6,6/-1.5,6/-6,4.5/0,0/0,0/4.5}
            \draw (\i,\j) node[font={\Huge},fill=white] {\color{blue}$3$};
        
        \foreach \i/\j in {-6/0,0/6,0/3,3/0,6/0,0/-6,-3/0,0/-3,-6/-3,-6/3,-3/6,6/-3,-3/-6,3/-6}
            \node[shape=circle, draw=black, fill=black,font={\Huge}] at (\i,\j){$ $};
    }}

\tikzset{chair3c/.pic=
{
        \draw[line width=2mm] (-6,-6)--(-6,6)--(0,6)--(0,0)--(6,0)--(6,-6)--(-6,-6);

        \draw[line width=0.7mm] (0,3)--(-3,3)--(-3,-3)--(3,-3)--(3,0);

        \draw[line width=0.7mm]
        (-6,0)--(-3,0);

        \draw[line width=0.7mm]
        (0,-3)--(0,-6);

        \foreach \i/\j in {-3/3,-3/-3,-4.5/0,0/-4.5}
            \draw (\i,\j) node[font={\Huge},fill=white] {\color{blue}$1$};
            
        \draw (3,-3) node[font={\Huge},fill=white]{\color{blue}$2$};

        \foreach \i/\j in {-6/3,-3/6,-6/-1.5,-6/-6,-1.5/-6,3/-6,6/-3,4.5/0,0/0,0/4.5}
            \draw (\i,\j) node[font={\Huge},fill=white] {\color{blue}$3$};
        
        \foreach \i/\j in {-6/-3,-6/0,-6/6,0/6,0/3,3/0,6/0,6/-6,0/-6,-3/-6,-3/0,0/-3}
            \node[shape=circle, draw=black, fill=black,font={\Huge}] at (\i,\j){$ $};
    }}

\tikzset{chair3d/.pic=
{
        \draw[line width=2mm] (-6,-6)--(-6,6)--(0,6)--(0,0)--(6,0)--(6,-6)--(-6,-6);

        \draw[line width=0.7mm] (0,3)--(-3,3)--(-3,-3)--(3,-3)--(3,0);

        \draw[line width=0.7mm]
        (-6,0)--(-3,0);

        \draw[line width=0.7mm]
        (0,-3)--(0,-6);

        \foreach \i/\j in {-3/3,-3/-3,-4.5/0,0/-4.5}
            \draw (\i,\j) node[font={\Huge},fill=white] {\color{blue}$1$};
            
        \draw (3,-3) node[font={\Huge},fill=white]{\color{blue}$2$};

        \foreach \i/\j in {-6/3,-3/6,0/4.5,0/0,4.5/0,3/-6,6/-3}
            \draw (\i,\j) node[font={\Huge},fill=white] {\color{blue}$3$};
        
        \foreach \i/\j in {-6/0,-6/6,0/6,0/3,3/0,6/0,6/-6,0/-6,-3/0,0/-3}
            \node[shape=circle, draw=black, fill=black,font={\Huge}] at (\i,\j){$ $};
    }}

\begin{center}
\hspace*{-5cm}
    \setlength{\tabcolsep}{5pt}
    \begin{longtable}[c]{cc}
    \resizebox{225pt}{!}{
        \tikz{
        \pic[rotate=180,scale=0.7] at (0,0) {chair3a};
        \pic[rotate=270,scale=0.7] at (8.4,0) {chair3a};
        \pic[rotate=90,scale=0.7] at (0,8.4) {chair3a};
        \pic[rotate=270,scale=0.7] at (8.4,16.8) {chair3d};
        \pic[rotate=270,scale=0.7] at (12.6,12.6) {chair3b};
        \pic[rotate=270,scale=0.7] at (16.8,8.4) {chair3c};
        \pic[rotate=180,scale=0.7] at (16.8,16.8) {chair3d};
        \pic[rotate=0,draw=red,scale=0.7] at (8.4,8.4) {chair3a};
        
        \foreach \i/\j in {4.2/4.2,4.2/12.6,12.6/4.2}
        {\node[shape=circle, draw=red, fill=white, text=red,font={\Huge}] at (\i,\j) {12};
        }
        \foreach \i/\j in {8.4/12.6,12.6/8.4}
        {\node[shape=circle, draw=red, fill=white, text=red,font={\Huge}] at (\i,\j) {10};
        }
        \foreach \i/\j in {4.2/8.4,8.4/4.2,10.5/8.4}
        {\node[shape=circle, draw=red, fill=white, text=red,font={\Huge}] at (\i,\j) {8};
        \node[shape=circle,draw=red,fill=white,text=red,font={\Huge}] at (8.4,10.5) {7};

        \node[shape=circle,draw=red,fill=white,text=red,font={\Huge}] at (6.3,8.4) {3};
        \node[shape=circle,draw=red,fill=white,text=red,font={\Huge}] at (8.4,6.3) {4};

        \node[font={\Huge}, text=black] at (8.4,-6) {Collar 1};
        }}}

        &
    \resizebox{225pt}{!}{
        \tikz{
        \pic[rotate=180,scale=0.7] at (0,0) {chair3a};
        \pic[rotate=270,scale=0.7] at (8.4,0) {chair3a};
        \pic[rotate=90,scale=0.7] at (0,8.4) {chair3a};
        \pic[rotate=90,scale=0.7] at (8.4,16.8) {chair3c};
        \pic[rotate=90,scale=0.7] at (12.6,12.6) {chair3b};
        \pic[rotate=90,scale=0.7] at (16.8,8.4) {chair3d};
        \pic[rotate=180,scale=0.7] at (16.8,16.8) {chair3d};
        \pic[rotate=0,draw=red,scale=0.7] at (8.4,8.4) {chair3a};
        
        \foreach \i/\j in {4.2/4.2,4.2/12.6,12.6/4.2}
        {\node[shape=circle, draw=red, fill=white, text=red,font={\Huge}] at (\i,\j) {12};
        }
        \foreach \i/\j in {8.4/12.6,12.6/8.4}
        {\node[shape=circle, draw=red, fill=white, text=red,font={\Huge}] at (\i,\j) {10};
        }
        \foreach \i/\j in {4.2/8.4,8.4/4.2,10.5/8.4}
        {\node[shape=circle, draw=red, fill=white, text=red,font={\Huge}] at (\i,\j) {8};
        \node[shape=circle,draw=red,fill=white,text=red,font={\Huge}] at (8.4,10.5) {7};

        \node[shape=circle,draw=red,fill=white,text=red,font={\Huge}] at (6.3,8.4) {3};
        \node[shape=circle,draw=red,fill=white,text=red,font={\Huge}] at (8.4,6.3) {4};

        \node[font={\Huge}, text=black] at (8.4,-6) {Collar 2};
        }}}
        \bigskip
        \\
    \resizebox{225pt}{!}{
        \tikz{
        \pic[rotate=180,scale=0.7] at (0,0) {chair3a};
        \pic[rotate=270,scale=0.7] at (8.4,0) {chair3a};
        \pic[rotate=90,scale=0.7] at (0,8.4) {chair3a};
        \pic[rotate=270,scale=0.7] at (8.4,16.8) {chair3d};
        \pic[rotate=0,scale=0.7] at (12.6,12.6) {chair3b};
        \pic[rotate=90,scale=0.7] at (16.8,8.4) {chair3d};
        \pic[rotate=0,draw=red,scale=0.7] at (8.4,8.4) {chair3a};
        
        \foreach \i/\j in {4.2/4.2,4.2/12.6,12.6/4.2}
        {\node[shape=circle, draw=red, fill=white, text=red,font={\Huge}] at (\i,\j) {12};
        }
        \foreach \i/\j in {8.4/12.6,12.6/8.4}
        {\node[shape=circle, draw=red, fill=white, text=red,font={\Huge}] at (\i,\j) {10};
        }
        \foreach \i/\j in {4.2/8.4,8.4/4.2,10.5/8.4}
        {\node[shape=circle, draw=red, fill=white, text=red,font={\Huge}] at (\i,\j) {8};
        \node[shape=circle,draw=red,fill=white,text=red,font={\Huge}] at (8.4,10.5) {7};

        \node[shape=circle,draw=red,fill=white,text=red,font={\Huge}] at (6.3,8.4) {3};
        \node[shape=circle,draw=red,fill=white,text=red,font={\Huge}] at (8.4,6.3) {4};

        \node[font={\Huge}, text=black] at (8.4,-6) {Collar 3};
        }}}

&
    \resizebox{155pt}{!}{
        \tikz{
        \pic[rotate=0,scale=0.7] at (0,0) {chair3d};
        \pic[rotate=270,scale=0.7] at (0,8.4) {chair3d};
        \pic[rotate=90,scale=0.7] at (8.4,0) {chair3d};
        \pic[rotate=0,scale=0.7] at (8.4,8.4) {chair3c};
        \pic[rotate=0,draw=red,scale=0.7] at (4.2,4.2) {chair3b};
        
        \foreach \i/\j in {0/4.2,8.4/4.2,4.2/8.4,4.2/0}
        {\node[shape=circle, draw=red, fill=white, text=red,font={\Huge}] at (\i,\j) {10};
        }
        \foreach \i/\j in {0/6.3,2.1/0,8.4/2.4,6.3/4.2}
        {\node[shape=circle, draw=red, fill=white, text=red,font={\Huge}] at (\i,\j) {8};
        }
        \foreach \i/\j in {2.4/8.4,4.2/6.3,6.3/0,0/2.1}
        {\node[shape=circle, draw=red, fill=white, text=red,font={\Huge}] at (\i,\j) {7};
        }
        \node[shape=circle,draw=red,fill=white,text=red,font={\Huge}] at (2.1,4.2) {3};
        \node[shape=circle,draw=red,fill=white,text=red,font={\Huge}] at (4.2,2.1) {4};

        \node[font={\Huge}, text=black] at (4.2,-6) {Collar 4};
        }}
        \bigskip
        \\

    \resizebox{225pt}{!}{
        \tikz{
        \pic[rotate=0,scale=0.7] at (0,0) {chair3b};
        \pic[rotate=270,scale=0.7] at (-4.2,4.2) {chair3d};
        \pic[rotate=90,scale=0.7] at (4.2,-4.2) {chair3d};
        \pic[rotate=0,scale=0.7] at (8.4,8.4) {chair3b};
        \pic[rotate=270,scale=0.7] at (4.2,12.6) {chair3d};
        \pic[rotate=90,scale=0.7] at (12.6,4.2) {chair3d};
        \pic[rotate=0,draw=red,scale=0.7] at (4.2,4.2) {chair3c};
        
        \foreach \i/\j in {0/8.4,8.4/0}
        {\node[shape=circle, draw=red, fill=white, text=red,font={\Huge}] at (\i,\j) {12};
        }
        \foreach \i/\j in {4.2/8.4,8.4/4.2,4.2/0,0/4.2}
        {\node[shape=circle, draw=red, fill=white, text=red,font={\Huge}] at (\i,\j) {10};
        }
        \foreach \i/\j in {6.3/4.2,2.1/0}
        {\node[shape=circle, draw=red, fill=white, text=red,font={\Huge}] at (\i,\j) {8};
        }
        \foreach \i/\j in {4.2/6.3,0/2.1}
        {\node[shape=circle, draw=red, fill=white, text=red,font={\Huge}] at (\i,\j) {7};
        }
        \node[shape=circle,draw=red,fill=white,text=red,font={\Huge}] at (2.1,4.2) {3};
        \node[shape=circle,draw=red,fill=white,text=red,font={\Huge}] at (4.2,2.1) {4};

        \node[font={\Huge}, text=black] at (4.2,-10.2) {Collar 5};
        }}

&
    \resizebox{225pt}{!}{
        \tikz{
        \pic[rotate=0,scale=0.7] at (0,0) {chair3b};
        \pic[rotate=270,scale=0.7] at (-4.2,4.2) {chair3d};
        \pic[rotate=90,scale=0.7] at (4.2,-4.2) {chair3d};
        \pic[rotate=90,scale=0.7] at (8.4,8.4) {chair3b};
        \pic[rotate=90,scale=0.7] at (4.2,12.6) {chair3c};
        \pic[rotate=90,scale=0.7] at (12.6,4.2) {chair3d};
        \pic[rotate=180,scale=0.7] at (12.6,12.6) {chair3d};
        \pic[rotate=0,draw=red,scale=0.7] at (4.2,4.2) {chair3c};
        
        \foreach \i/\j in {0/8.4,8.4/0}
        {\node[shape=circle, draw=red, fill=white, text=red,font={\Huge}] at (\i,\j) {12};
        }
        \foreach \i/\j in {4.2/8.4,8.4/4.2,4.2/0,0/4.2}
        {\node[shape=circle, draw=red, fill=white, text=red,font={\Huge}] at (\i,\j) {10};
        }
        \foreach \i/\j in {6.3/4.2,2.1/0}
        {\node[shape=circle, draw=red, fill=white, text=red,font={\Huge}] at (\i,\j) {8};
        }
        \foreach \i/\j in {4.2/6.3,0/2.1}
        {\node[shape=circle, draw=red, fill=white, text=red,font={\Huge}] at (\i,\j) {7};
        }
        \node[shape=circle,draw=red,fill=white,text=red,font={\Huge}] at (2.1,4.2) {3};
        \node[shape=circle,draw=red,fill=white,text=red,font={\Huge}] at (4.2,2.1) {4};

        \node[font={\Huge}, text=black] at (4.2,-10.2) {Collar 6};
        }}\\
    \end{longtable}
    \hspace*{-5cm}
\end{center}

\begin{center}
    \resizebox{225pt}{!}{
        \tikz{
        \pic[rotate=0,scale=0.7] at (0,0) {chair3b};
        \pic[rotate=270,scale=0.7] at (-4.2,4.2) {chair3d};
        \pic[rotate=90,scale=0.7] at (4.2,-4.2) {chair3d};
        \pic[rotate=270,scale=0.7] at (8.4,8.4) {chair3b};
        \pic[rotate=270,scale=0.7] at (4.2,12.6) {chair3d};
        \pic[rotate=270,scale=0.7] at (12.6,4.2) {chair3c};
        \pic[rotate=180,scale=0.7] at (12.6,12.6) {chair3d};
        \pic[rotate=0,draw=red,scale=0.7] at (4.2,4.2) {chair3c};
        
        \foreach \i/\j in {0/8.4,8.4/0}
        {\node[shape=circle, draw=red, fill=white, text=red,font={\Huge}] at (\i,\j) {12};
        }
        \foreach \i/\j in {4.2/8.4,8.4/4.2,4.2/0,0/4.2}
        {\node[shape=circle, draw=red, fill=white, text=red,font={\Huge}] at (\i,\j) {10};
        }
        \foreach \i/\j in {6.3/4.2,2.1/0}
        {\node[shape=circle, draw=red, fill=white, text=red,font={\Huge}] at (\i,\j) {8};
        }
        \foreach \i/\j in {4.2/6.3,0/2.1}
        {\node[shape=circle, draw=red, fill=white, text=red,font={\Huge}] at (\i,\j) {7};
        }
        \node[shape=circle,draw=red,fill=white,text=red,font={\Huge}] at (2.1,4.2) {3};
        \node[shape=circle,draw=red,fill=white,text=red,font={\Huge}] at (4.2,2.1) {4};

        \node[font={\Huge}, text=black] at (4.2,-10.2) {Collar 7};
        }}
\end{center}
\end{proof}

\subsection{Non-Pinwheel}

\begin{proposition}
    Consider the non-periodic tiling of the plane with the \href{https://tilings.math.uni-bielefeld.de/substitution/non-pinwheel/}{non-pinwheel} in \cite{TilingEncyclopedia}. We identify a rectangle supertile consisting of 10 right triangles with side lengths of 1, 2, and $\sqrt{5}$. As the supertiles are joined, the colour of the 4 corners uniquely identifies the rotation of the supertile.

    \tikzset{nonpin/.pic=
{
\draw[line width=3mm] (0,0)--(0,4.48)--(2.24,4.48)--(2.24,0)--(0,0);
        \draw[line width=0.8mm] (0,4.48)--(2.24,0);
        \draw[line width=0.8mm] (0.45,3.59)--(2.24,4.48);
        \draw[line width=0.8mm] (0,0)--(1.79,0.89);
        \draw[line width=0.8mm] (0.89,0.45)--(0,2.24);
        \draw[line width=0.8mm] (1.35,4.03)--(2.24,2.24);
        \draw[line width=0.8mm] (0,2.24)--(0.89,2.69);
        \draw[line width=0.8mm] (0.89,2.69)--(0.89,0.45);
        \draw[line width=0.8mm] (1.35,1.79)--(2.24,2.24);
        \draw[line width=0.8mm] (1.35,4.03)--(1.35,1.79);

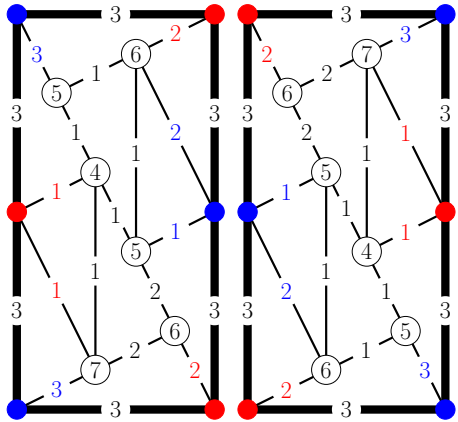
        \foreach \i/\j in {0/1.12,0/3.36,1.12/4.48,2.24/3.36,2.24/1.12,1.12/0}
        {
        \draw (\i,\j) node[shape=circle,fill=white,font={\Huge}] {\color{black}$3$};
        }
        \draw (0.45,0.23) node[shape=circle,fill=white,font={\Huge}] {\color{blue}$3$};
        \draw (0.23,4.03) node[shape=circle,fill=white,font={\Huge}] {\color{blue}$3$};
        \draw (1.8,3.14) node[shape=circle,fill=white,font={\Huge}] {\color{blue}$2$};
        \draw (1.8,2.02) node[shape=circle,fill=white,font={\Huge}] {\color{blue}$1$};
        \draw (1.8,4.26) node[shape=circle,fill=white,font={\Huge}] {\color{red}$2$};
        \draw (2.02,0.45) node[shape=circle,fill=white,font={\Huge}] {\color{red}$2$};
        \draw (0.45,2.47) node[shape=circle,fill=white,font={\Huge}] {\color{red}$1$};
        \draw (0.45,1.35) node[shape=circle,fill=white,font={\Huge}] {\color{red}$1$};
        \draw (1.34,0.67) node[shape=circle,fill=white,font={\Huge}] {\color{black}$2$};
        \draw (1.57,1.34) node[shape=circle,fill=white,font={\Huge}]  {\color{black}$2$};
        \foreach \i/\j in {1.12/2.2, 0.67/3.14, 0.9/3.81, 1.35/2.91, 0.89/1.57}
        {
        \draw (\i,\j) node[shape=circle,fill=white,font={\Huge}]  {\color{black}$1$};
        }

        \node[shape=circle, draw=blue, fill=blue,minimum size=7mm] at (0,4.48){};
        \node[shape=circle, draw=blue, fill=blue,minimum size=7mm] at (2.24,2.24){};
        \node[shape=circle, draw=blue, fill=blue,minimum size=7mm] at (0,0){};
        \node[shape=circle, draw=red, fill=red,minimum size=7mm] at (0,2.24){};
        \node[shape=circle, draw=red, fill=red,minimum size=7mm] at (2.24,4.48){};
        \node[shape=circle, draw=red, fill=red,minimum size=7mm] at (2.24,0){};

        \node[shape=circle, draw=black, fill=white, text=black,font={\Huge}] at (0.89,0.45) {$7$};
        \node[shape=circle, draw=black, fill=white, text=black,font={\Huge}] at (1.35,1.79) {$5$};
        \node[shape=circle, draw=black, fill=white, text=black,font={\Huge}] at (1.79,0.89) {$6$};
        \node[shape=circle, draw=black, fill=white, text=black,font={\Huge}] at (0.89,2.69) {$4$};
        \node[shape=circle, draw=black, fill=white, text=black,font={\Huge}] at (0.45,3.59) {$5$};
        \node[shape=circle, draw=black, fill=white, text=black,font={\Huge}] at (1.35,4.03) {$6$};
}}

\begin{center}
    \resizebox{176pt}{!}{
    \tikz{
    \pic[scale=3.2] {nonpin}
    }
    \tikz{
   \pic[scale=3.2,rotate=180]{nonpin}
    }
    }
\end{center}

Assign all external supertile edges a weight of 3. There are 6 vertices along the external supertile and they are thus 2-colourable. The bipartitle rectangle supertiles are investigated in \cite{ColouringsofAperiodicTilings}. Colour the vertices blue and red alternating. External vertices are incident to either 1 or 2 internal supertile edges. Assign weights to the internal edges such that the sum of internal edges incident to a red vertex is 2 and 3 to a blue vertex. Weight the remaining internal edges as above.
This provides a solution to the 1-2-3 problem. 
\end{proposition}

\begin{proof}
    When combined, each external vertex will have 2, 3, or 4 external supertile edges and an equivalent number of groups of internal edges. Let $d$ represent that number (i.e. 2, 3, or 4). All external edges are weighted 3 and each group of internal edges incident to a red vertex sum up to 2. Thus for all the red vertices $3d + 2d = 5d$, so all red vertices will have a sum that is a multiple of 5, specifically 10, 15, or 20. While the groups of internal edges incident to the blue vertices sum up to 3. Therefore, $3d + 3d = 6d$, so all blue vertices will have a sum that is a multiple of 6, specifically 12, 18, or 24. Since they are alternating and the internal vertices have sums 7 or less the result follows.
\end{proof}

\tikzset{nonpinblue/.pic=
{
\draw[line width=3mm] (0,0)--(0,4.48)--(2.24,4.48)--(2.24,0)--(0,0);
        \draw[line width=0.8mm] (0,4.48)--(2.24,0);
        \draw[line width=0.8mm] (0.45,3.59)--(2.24,4.48);
        \draw[line width=0.8mm] (0,0)--(1.79,0.89);
        \draw[line width=0.8mm] (0.89,0.45)--(0,2.24);
        \draw[line width=0.8mm] (1.35,4.03)--(2.24,2.24);
        \draw[line width=0.8mm] (0,2.24)--(0.89,2.69);
        \draw[line width=0.8mm] (0.89,2.69)--(0.89,0.45);
        \draw[line width=0.8mm] (1.35,1.79)--(2.24,2.24);
        \draw[line width=0.8mm] (1.35,4.03)--(1.35,1.79);

        \foreach \i/\j in {0/1.12,0/3.36,1.12/4.48,2.24/3.36,2.24/1.12,1.12/0}
        {
        \draw (\i,\j) node[shape=circle,fill=white,font={\Huge}] {\color{black}$3$};
        }
        \draw (0.45,0.23) node[shape=circle,fill=white,font={\Huge}] {\color{blue}$3$};
        \draw (0.23,4.03) node[shape=circle,fill=white,font={\Huge}] {\color{blue}$3$};
        \draw (1.8,3.14) node[shape=circle,fill=white,font={\Huge}] {\color{blue}$2$};
        \draw (1.8,2.02) node[shape=circle,fill=white,font={\Huge}] {\color{blue}$1$};
        \draw (1.8,4.26) node[shape=circle,fill=white,font={\Huge}] {\color{red}$2$};
        \draw (2.02,0.45) node[shape=circle,fill=white,font={\Huge}] {\color{red}$2$};
        \draw (0.45,2.47) node[shape=circle,fill=white,font={\Huge}] {\color{red}$1$};
        \draw (0.45,1.35) node[shape=circle,fill=white,font={\Huge}] {\color{red}$1$};
        \draw (1.34,0.67) node[shape=circle,fill=white,font={\Huge}] {\color{black}$2$};
        \draw (1.57,1.34) node[shape=circle,fill=white,font={\Huge}]  {\color{black}$2$};
        \foreach \i/\j in {1.12/2.2, 0.67/3.14, 0.9/3.81, 1.35/2.91, 0.89/1.57}
        {
        \draw (\i,\j) node[shape=circle,fill=white,font={\Huge}]  {\color{black}$1$};
        }

        \node[shape=circle, draw=blue, fill=blue,minimum size=7mm] at (0,4.48){};
        \node[shape=circle, draw=blue, fill=blue,minimum size=7mm] at (2.24,2.24){};
        \node[shape=circle, draw=blue, fill=blue,minimum size=7mm] at (0,0){};
        \node[shape=circle, draw=red, fill=red,minimum size=7mm] at (0,2.24){};
        \node[shape=circle, draw=red, fill=red,minimum size=7mm] at (2.24,4.48){};
        \node[shape=circle, draw=red, fill=red,minimum size=7mm] at (2.24,0){};

        \node[shape=circle, draw=black, fill=white, text=black,font={\Huge}] at (0.89,0.45) {$7$};
        \node[shape=circle, draw=black, fill=white, text=black,font={\Huge}] at (1.35,1.79) {$5$};
        \node[shape=circle, draw=black, fill=white, text=black,font={\Huge}] at (1.79,0.89) {$6$};
        \node[shape=circle, draw=black, fill=white, text=black,font={\Huge}] at (0.89,2.69) {$4$};
        \node[shape=circle, draw=black, fill=white, text=black,font={\Huge}] at (0.45,3.59) {$5$};
        \node[shape=circle, draw=black, fill=white, text=black,font={\Huge}] at (1.35,4.03) {$6$};

    \coordinate (-A) at (0,0);
    \coordinate (-B) at (0,4.48);
    \coordinate (-C) at (2.24,4.48);
    \coordinate (-D) at (2.24,0);
    \coordinate (-E) at (0,2.24);
    \coordinate (-F) at (2.24,2.24);
}}

 \tikzset{nonpinred/.pic=
{
\draw[line width=3mm] (0,0)--(0,4.48)--(2.24,4.48)--(2.24,0)--(0,0);
        \draw[line width=0.8mm] (0,4.48)--(2.24,0);
        \draw[line width=0.8mm] (0.45,3.59)--(2.24,4.48);
        \draw[line width=0.8mm] (0,0)--(1.79,0.89);
        \draw[line width=0.8mm] (0.89,0.45)--(0,2.24);
        \draw[line width=0.8mm] (1.35,4.03)--(2.24,2.24);
        \draw[line width=0.8mm] (0,2.24)--(0.89,2.69);
        \draw[line width=0.8mm] (0.89,2.69)--(0.89,0.45);
        \draw[line width=0.8mm] (1.35,1.79)--(2.24,2.24);
        \draw[line width=0.8mm] (1.35,4.03)--(1.35,1.79);

        \foreach \i/\j in {0/1.12,0/3.36,1.12/4.48,2.24/3.36,2.24/1.12,1.12/0}
        {
        \draw (\i,\j) node[shape=circle,fill=white,font={\Huge}] {\color{black}$3$};
        }
        \draw (0.45,0.23) node[shape=circle,fill=white,font={\Huge}] {\color{red}$2$};
        \draw (0.23,4.03) node[shape=circle,fill=white,font={\Huge}] {\color{red}$2$};
        \draw (1.8,3.14) node[shape=circle,fill=white,font={\Huge}] {\color{red}$1$};
        \draw (1.8,2.02) node[shape=circle,fill=white,font={\Huge}] {\color{red}$1$};
        \draw (1.8,4.26) node[shape=circle,fill=white,font={\Huge}] {\color{blue}$3$};
        \draw (2.02,0.45) node[shape=circle,fill=white,font={\Huge}] {\color{blue}$3$};
        \draw (0.45,2.47) node[shape=circle,fill=white,font={\Huge}] {\color{blue}$1$};
        \draw (0.45,1.35) node[shape=circle,fill=white,font={\Huge}] {\color{blue}$2$};
        \draw (1.34,0.67) node[shape=circle,fill=white,font={\Huge}] {\color{black}$1$};
         \draw (0.67,3.14) node[shape=circle,fill=white,font={\Huge}] {\color{black}$2$};
         \draw (0.9,3.81) node[shape=circle,fill=white,font={\Huge}] {\color{black}$2$};
        \draw (1.57,1.34) node[shape=circle,fill=white,font={\Huge}]  {\color{black}$1$};
        \foreach \i/\j in {1.12/2.2, 1.35/2.91, 0.89/1.57}
        {
        \draw (\i,\j) node[shape=circle,fill=white,font={\Huge}]  {\color{black}$1$};
        }

        \node[shape=circle, draw=red, fill=red,minimum size=7mm] at (0,4.48){};
        \node[shape=circle, draw=red, fill=red,minimum size=7mm] at (2.24,2.24){};
        \node[shape=circle, draw=red, fill=red,minimum size=7mm] at (0,0){};
        \node[shape=circle, draw=blue, fill=blue,minimum size=7mm] at (0,2.24){};
        \node[shape=circle, draw=blue, fill=blue,minimum size=7mm] at (2.24,4.48){};
        \node[shape=circle, draw=blue, fill=blue,minimum size=7mm] at (2.24,0){};

        \node[shape=circle, draw=black, fill=white, text=black,font={\Huge}] at (0.89,0.45) {$6$};
        \node[shape=circle, draw=black, fill=white, text=black,font={\Huge}] at (1.35,1.79) {$4$};
        \node[shape=circle, draw=black, fill=white, text=black,font={\Huge}] at (1.79,0.89) {$5$};
        \node[shape=circle, draw=black, fill=white, text=black,font={\Huge}] at (0.89,2.69) {$5$};
        \node[shape=circle, draw=black, fill=white, text=black,font={\Huge}] at (0.45,3.59) {$6$};
        \node[shape=circle, draw=black, fill=white, text=black,font={\Huge}] at (1.35,4.03) {$7$};

    \coordinate (-A) at (0,0);
    \coordinate (-B) at (0,4.48);
    \coordinate (-C) at (2.24,4.48);
    \coordinate (-D) at (2.24,0);
    \coordinate (-E) at (0,2.24);
    \coordinate (-F) at (2.24,2.24);
}}

\begin{center}

\resizebox{400pt}{!}{
    \tikz{
    \clip (-25,-1) rectangle (14,28);
    \pic (T1) [scale=3,rotate=-26.565] at (0,0) {nonpinblue};
    \pic (T2) [scale=3,rotate=63.43] at (T1-A) {nonpinblue};
    \pic (T3) [scale=3,rotate=243.43] at (T1-F) {nonpinblue};
    \pic (T4) [scale=3,rotate=63.43] at (T3-B) {nonpinblue};
    \pic (T5) [scale=3,rotate=63.43] at (T4-F) {nonpinblue};
    \pic (T6) [scale=3,rotate=63.43] at (T2-B) {nonpinblue};
    \pic (T7) [scale=3,rotate=243.43] at (T2-E) {nonpinred};
    \pic (T8) [scale=3,rotate=243.43] at (T6-E) {nonpinred};
    \pic (T9) [scale=3,rotate=-206.565] at (T8-A) {nonpinred};
    \pic (T10) [scale=3,rotate=243.43] at (T8-D) {nonpinblue};
    \pic (T11) [scale=3,rotate=243.43] at (T4-D) {nonpinred};
    \pic (T12) [scale=3,rotate=63.43] at (T11-E) {nonpinblue};
    \pic (T13) [scale=3,rotate=63.43] at (T1-E) {nonpinred};
    \pic (T14) [scale=3,rotate=-26.565] at (T13-F) {nonpinred};
    \pic (T15) [scale=3,rotate=63.43] at (T13-F) {nonpinred};
    \pic (T16) [scale=3,rotate=63.43] at (T2-C) {nonpinred};
    \pic (T17) [scale=3,rotate=-206.565] at (T16-C) {nonpinblue};
    \pic (T18) [scale=3,rotate=63.43] at (T16-F) {nonpinred};
    \pic (T19) [scale=3,rotate=-26.565] at (T18-C) {nonpinblue};
    \pic (T20) [scale=3,rotate=-206.565] at (T19-E) {nonpinred};
    \pic (T21) [scale=3,rotate=-26.565] at (T20-D) {nonpinblue};
    \pic (T22) [scale=3,rotate=-206.565] at (T21-B) {nonpinblue};
    \pic (T23) [scale=3,rotate=63.43] at (T21-C) {nonpinred};
    \pic (T24) [scale=3,rotate=-26.565] at (T21-F) {nonpinblue};
    \pic (T25) [scale=3,rotate=-26.565] at (T18-F) {nonpinred};
    \pic (T26) [scale=3,rotate=63.43] at (T15-D) {nonpinblue};
    \pic (T27) [scale=3,rotate=63.43] at (T26-D) {nonpinred};
    \pic (T28) [scale=3,rotate=63.43] at (T27-F) {nonpinred};
    \pic (T29) [scale=3,rotate=63.43] at (T28-D) {nonpinblue};
    \pic (T30) [scale=3,rotate=63.43] at (T29-F) {nonpinblue};
    \pic (T31) [scale=3,rotate=63.43] at (T5-D) {nonpinred};
    \pic (T32) [scale=3,rotate=63.43] at (T12-D) {nonpinred};
    \pic (T33) [scale=3,rotate=63.43] at (T31-F) {nonpinred};
    \pic (T34) [scale=3,rotate=63.43] at (T32-F) {nonpinred};
    \pic (T35) [scale=3,rotate=63.43] at (T34-D) {nonpinblue};
    \pic (T36) [scale=3,rotate=-26.565] at (T33-F) {nonpinred};
    \pic (T37) [scale=3,rotate=63.43] at (T33-F) {nonpinred};
    \pic (T38) [scale=3,rotate=63.43] at (T37-D) {nonpinblue};
    \pic (T39) [scale=3,rotate=63.43] at (T38-F) {nonpinblue};
    \pic (T40) [scale=3,rotate=63.43] at (T39-D) {nonpinred};
    \pic (T41) [scale=3,rotate=-26.565] at (T38-F) {nonpinblue};
    \pic (T42) [scale=3,rotate=63.43] at (T35-D) {nonpinred};
    \pic (T43) [scale=3,rotate=63.43] at (T42-F) {nonpinred};
    \pic (T44) [scale=3,rotate=-26.565] at (T41-F) {nonpinblue};

    \foreach \i/\j in {T9/-F,T10/-E,T3/-E,T15/-F,T21/-E,T29/-E,T34/-F,T37/-F}
    {
    \node[shape=circle, draw=red, fill=red!60, text=black,font={\Huge}] at (\i\j) {10};
    }

    \foreach \i/\j in {T9/-A,T8/-B,T11/-A,T17/-E,T2/-D,T1/-C,T32/-F,T5/-C,T13/-F,T16/-F,T31/-F,T18/-F,T20/-A,T25/-F,T21/-C,T24/-C,T29/-D,T28/-A,T33/-F,T36/-F,T39/-D,T41/-C,T44/-D,T43/-A}
    {
    \node[shape=circle, draw=red, fill=red!60, text=black,font={\Huge}] at (\i\j) {15};
    }

    \foreach \i/\j in {T7/-B,T17/-D,T6/-D,T12/-C,T14/-B,T19/-C,T36/-B,T39/-C}
    {
    \node[shape=circle, draw=red, fill=red!60, text=black,font={\Huge}] at (\i\j) {20};
    }

    \foreach \i/\j in {T6/-F,T2/-F,T12/-F,T5/-F,T19/-F,T26/-F,T39/-F,T42/-E}
    {
    \node[shape=circle, draw=blue, fill=blue!30, text=black,font={\Huge}] at (\i\j) {12};
    }

    \foreach \i/\j in {T9/-E,T8/-E,T7/-E,T3/-A,T11/-E,T17/-A,T4/-F,T32/-C,T13/-C,T14/-E,T14/-C,T18/-C,T19/-B,T25/-C,T21/-B,T24/-F,T30/-A,T27/-D,T37/-D,T36/-C,T38/-F,T41/-F}
    {
    \node[shape=circle, draw=blue, fill=blue!30, text=black,font={\Huge}] at (\i\j) {18};
    }

    \foreach \i/\j in {T9/-D,T10/-B,T3/-B,T1/-B,T15/-C,T20/-D,T29/-A,T33/-D}
    {
    \node[shape=circle, draw=blue, fill=blue!30, text=black,font={\Huge}] at (\i\j) {24};
    }
    }}
    \end{center}

\subsection{Pinwheel}

\begin{proposition}
    Consider the non-periodic tiling of the plane with the \href{https://tilings.math.uni-bielefeld.de/substitution/pinwheel/}{pinwheel} in \cite{TilingEncyclopedia}. We identify the following four supertiles whose external vertices are 2-colourable. The bipartite rectangular supertiles are investigated in \cite{ColouringsofAperiodicTilings}.

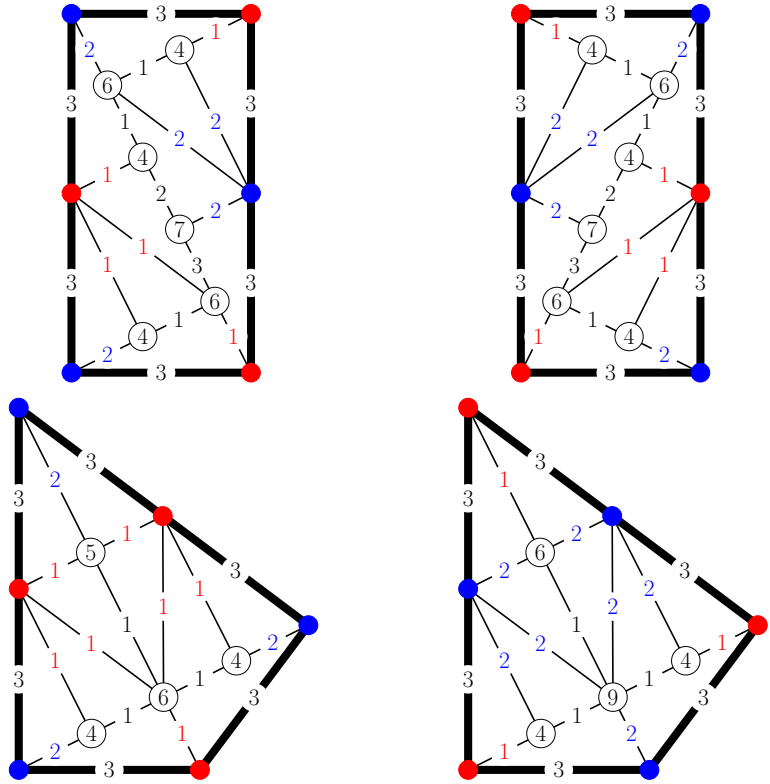
    \tikzset{rectangleblue/.pic=
{
        \draw[line width=3mm] (0,0)--(0,4.48)--(2.24,4.48)--(2.24,0)--(0,0);
        \draw[line width=0.6mm] (0,4.48)--(2.24,0);
        \draw[line width=0.6mm] (0.45,3.59)--(2.24,4.48);
        \draw[line width=0.6mm] (0,0)--(1.79,0.89);
        \draw[line width=0.6mm] (0.89,0.45)--(0,2.24);
        \draw[line width=0.6mm] (1.35,4.03)--(2.24,2.24);
        \draw[line width=0.6mm] (0,2.24)--(0.89,2.69);
        \draw[line width=0.6mm] (0,2.24)--(1.79,0.89);
        \draw[line width=0.6mm] (1.35,1.79)--(2.24,2.24);
        \draw[line width=0.6mm] (0.45,3.59)--(2.24,2.24);

        \draw(0,1.12) node[shape=circle,fill=white,font={\Huge}]{\color{black}$3$};
        \draw(0,3.36) node[shape=circle,fill=white,font={\Huge}]{\color{black}$3$};
        \draw(1.12,4.48) node[shape=circle,fill=white,font={\Huge}]{\color{black}$3$};
        \draw(2.24,3.36) node[shape=circle,fill=white,font={\Huge}]{\color{black}$3$};
        \draw(2.24,1.12) node[shape=circle,fill=white,font={\Huge}]{\color{black}$3$};
        \draw(1.12,0) node[shape=circle,fill=white,font={\Huge}]{\color{black}$3$};
        \draw (0.225,4.04) node[shape=circle,fill=white,font={\Huge}] {\color{blue}$2$};
        \draw (0.45,0.225) node[shape=circle,fill=white,font={\Huge}] {\color{blue}$2$};
        \draw (1.8,3.14) node[shape=circle,fill=white,font={\Huge}] {\color{blue}$2$};
        \draw (1.35,2.92) node[shape=circle,fill=white,font={\Huge}] {\color{blue}$2$};
        \draw (1.8,2.02) node[shape=circle,fill=white,font={\Huge}] {\color{blue}$2$};
        \draw (1.8,4.26) node[shape=circle,fill=white,font={\Huge}] {\color{red}$1$};
        \draw (2.02,0.45) node[shape=circle,fill=white,font={\Huge}] {\color{red}$1$};
        \draw (0.45,2.47) node[shape=circle,fill=white,font={\Huge}] {\color{red}$1$};
        \draw (0.9,1.57) node[shape=circle,fill=white,font={\Huge}] {\color{red}$1$};
        \draw (0.45,1.35) node[shape=circle,fill=white,font={\Huge}] {\color{red}$1$};
        \draw (1.34,0.67) node[shape=circle,fill=white,font={\Huge}] {\color{black}$1$};
        \draw (1.57,1.34) node[shape=circle,fill=white,font={\Huge}] {\color{black}$3$};
        \draw (1.12,2.24) node[shape=circle,fill=white,font={\Huge}] {\color{black}$2$};
        \draw (0.67,3.14) node[shape=circle,fill=white,font={\Huge}] {\color{black}$1$};
        \draw (0.9,3.81) node[shape=circle,fill=white,font={\Huge}] {\color{black}$1$};

        \node[shape=circle, draw=blue, fill=blue,minimum size=7mm] at (0,4.48){};
        \node[shape=circle, draw=blue, fill=blue,minimum size=7mm] at (2.24,2.24){};
        \node[shape=circle, draw=blue, fill=blue,minimum size=7mm] at (0,0){};
        \node[shape=circle, draw=red, fill=red,minimum size=7mm] at (0,2.24){};
        \node[shape=circle, draw=red, fill=red,minimum size=7mm] at (2.24,4.48){};
        \node[shape=circle, draw=red, fill=red,minimum size=7mm] at (2.24,0){};

        \node[shape=circle, draw=black, fill=white, text=black, font={\Huge}] at (0.89,0.45) {$4$};
        \node[shape=circle, draw=black, fill=white, text=black, font={\Huge}] at (1.35,1.79) {$7$};
        \node[shape=circle, draw=black, fill=white, text=black, font={\Huge}] at (1.79,0.89) {$6$};
        \node[shape=circle, draw=black, fill=white, text=black, font={\Huge}] at (0.89,2.69) {$4$};
        \node[shape=circle, draw=black, fill=white, text=black, font={\Huge}] at (0.45,3.59) {$6$};
        \node[shape=circle, draw=black, fill=white, text=black, font={\Huge}] at (1.35,4.03) {$4$};
}}

\tikzset{kiteblue/.pic={
        \draw[line width=3mm] (0,0)--(0,4.48)--(3.58,1.79)--(2.24,0)--(0,0);
        \draw[line width=0.6mm] (0,4.48)--(2.24,0);
        \draw[line width=0.6mm] (0,2.24)--(1.78,3.14);
        \draw[line width=0.6mm] (0,0)--(3.58,1.79);
        \draw[line width=0.6mm] (0,2.24)--(0.9,0.45);
        \draw[line width=0.6mm] (1.78,3.14)--(2.69,1.35);
        \draw[line width=0.6mm] (0,2.24)--(1.79,0.9);
        \draw[line width=0.6mm] (1.78,3.14)--(1.79,0.9);

        \draw (1.34,1.8) node[shape=circle,fill=white,font={\Huge}] {\color{black}$1$};
        \draw (1.35,0.675) node[shape=circle,fill=white,font={\Huge}] {\color{black}$1$};
        \draw (2.24,1.13) node[shape=circle,fill=white,font={\Huge}] {\color{black}$1$};
        
        \draw(0,1.12) 
        node[shape=circle,fill=white,font={\Huge}]{\color{black}$3$};
        \draw(0,3.36) node[shape=circle,fill=white,font={\Huge}]{\color{black}$3$};
        \draw(1.12,0) node[shape=circle,fill=white,font={\Huge}]{\color{black}$3$};
        \draw(2.91,0.895) node[shape=circle,fill=white,font={\Huge}]{\color{black}$3$};
        \draw(0.895,3.81) node[shape=circle,fill=white,font={\Huge}]{\color{black}$3$};
        \draw(2.68,2.465) node[shape=circle,fill=white,font={\Huge}]{\color{black}$3$};
        \draw (0.45,3.59) node[shape=circle,fill=white,font={\Huge}] {\color{blue}$2$};
        \draw (0.45,0.225) node[shape=circle,fill=white,font={\Huge}] {\color{blue}$2$};
        \draw (3.14,1.58) node[shape=circle,fill=white,font={\Huge}] {\color{blue}$2$};
        \draw (2.02,0.45) node[shape=circle,fill=white,font={\Huge}] {\color{red}$1$};
        \draw (0.45,2.47) node[shape=circle,fill=white,font={\Huge}] {\color{red}$1$};
        \draw (1.34,2.92) node[shape=circle,fill=white,font={\Huge}] {\color{red}$1$};
        \draw (0.45,1.35) node[shape=circle,fill=white,font={\Huge}] {\color{red}$1$};
        \draw (0.895,1.57) node[shape=circle,fill=white,font={\Huge}] {\color{red}$1$};
        \draw (1.79,2.02) node[shape=circle,fill=white,font={\Huge}] {\color{red}$1$};
        \draw (2.24,2.25) node[shape=circle,fill=white,font={\Huge}] {\color{red}$1$};

        \node[shape=circle, draw=black, fill=white, text=black, font={\Huge}] at (0.9,0.45) {$4$};
        \node[shape=circle, draw=black, fill=white, text=black, font={\Huge}] at (2.69,1.35) {$4$};
        \node[shape=circle, draw=blue, fill=blue,minimum size=7mm] at (0,4.48){};
        \node[shape=circle, draw=blue, fill=blue,minimum size=7mm] at (0,0){};
        \node[shape=circle, draw=blue, fill=blue,minimum size=7mm] at (3.58,1.79){};
        \node[shape=circle, draw=red, fill=red,minimum size=7mm] at (0,2.24){};
        \node[shape=circle, draw=red, fill=red,minimum size=7mm] at (1.78,3.14){};
        \node[shape=circle, draw=red, fill=red,minimum size=7mm] at (2.24,0){};
        \node[shape=circle, draw=black, fill=white, text=black, font={\Huge}] at (0.89,2.69) {$5$};
        \node[shape=circle, draw=black, fill=white, text=black, font={\Huge}] at (1.79,0.9) {$6$};
}
}

\tikzset{kitered/.pic={
        \draw[line width=3mm] (0,0)--(0,4.48)--(3.58,1.79)--(2.24,0)--(0,0);
        \draw[line width=0.6mm] (0,4.48)--(2.24,0);
        \draw[line width=0.6mm] (0,2.24)--(1.78,3.14);
        \draw[line width=0.6mm] (0,0)--(3.58,1.79);
        \draw[line width=0.6mm] (0,2.24)--(0.9,0.45);
        \draw[line width=0.6mm] (1.78,3.14)--(2.69,1.35);
        \draw[line width=0.6mm] (0,2.24)--(1.79,0.9);
        \draw[line width=0.6mm] (1.78,3.14)--(1.79,0.9);

        \draw (1.34,1.8) node[shape=circle,fill=white,font={\Huge}] {\color{black}$1$};
        \draw (1.35,0.675) node[shape=circle,fill=white,font={\Huge}] {\color{black}$1$};
        \draw (2.24,1.13) node[shape=circle,fill=white,font={\Huge}] {\color{black}$1$};
        
        \draw(0,1.12) 
        node[shape=circle,fill=white,font={\Huge}]{\color{black}$3$};
        \draw(0,3.36) node[shape=circle,fill=white,font={\Huge}]{\color{black}$3$};
        \draw(1.12,0) node[shape=circle,fill=white,font={\Huge}]{\color{black}$3$};
        \draw(2.91,0.895) node[shape=circle,fill=white,font={\Huge}]{\color{black}$3$};
        \draw(0.895,3.81) node[shape=circle,fill=white,font={\Huge}]{\color{black}$3$};
        \draw(2.68,2.465) node[shape=circle,fill=white,font={\Huge}]{\color{black}$3$};
        \draw (0.45,3.59) node[shape=circle,fill=white,font={\Huge}] {\color{red}$1$};
        \draw (0.45,0.225) node[shape=circle,fill=white,font={\Huge}] {\color{red}$1$};
        \draw (3.14,1.58) node[shape=circle,fill=white,font={\Huge}] {\color{red}$1$};
        \draw (2.02,0.45) node[shape=circle,fill=white,font={\Huge}] {\color{blue}$2$};
        \draw (0.45,2.47) node[shape=circle,fill=white,font={\Huge}] {\color{blue}$2$};
        \draw (1.34,2.92) node[shape=circle,fill=white,font={\Huge}] {\color{blue}$2$};
        \draw (0.45,1.35) node[shape=circle,fill=white,font={\Huge}] {\color{blue}$2$};
        \draw (0.895,1.57) node[shape=circle,fill=white,font={\Huge}] {\color{blue}$2$};
        \draw (1.79,2.02) node[shape=circle,fill=white,font={\Huge}] {\color{blue}$2$};
        \draw (2.24,2.25) node[shape=circle,fill=white,font={\Huge}] {\color{blue}$2$};

        \node[shape=circle, draw=black, fill=white, text=black, font={\Huge}] at (0.9,0.45) {$4$};
        \node[shape=circle, draw=black, fill=white, text=black, font={\Huge}] at (2.69,1.35) {$4$};
        \node[shape=circle, draw=red, fill=red,minimum size=7mm] at (0,4.48){};
        \node[shape=circle, draw=red, fill=red,minimum size=7mm] at (0,0){};
        \node[shape=circle, draw=red, fill=red,minimum size=7mm] at (3.58,1.79){};
        \node[shape=circle, draw=blue, fill=blue,minimum size=7mm] at (0,2.24){};
        \node[shape=circle, draw=blue, fill=blue,minimum size=7mm] at (1.78,3.14){};
        \node[shape=circle, draw=blue, fill=blue,minimum size=7mm] at (2.24,0){};
        \node[shape=circle, draw=black, fill=white, text=black, font={\Huge}] at (0.89,2.69) {$6$};
        \node[shape=circle, draw=black, fill=white, text=black, font={\Huge}] at (1.79,0.9) {$9$};
}
}

\begin{center}
    \setlength{\tabcolsep}{25pt}
    \begin{tabular}{c c}
    \resizebox{80pt}{!}{
    \tikz[]{
    \pic[scale=3] {rectangleblue};
    }}
         &  
     \resizebox{80pt}{!}{
    \tikz[]{
    \pic[xscale=-3,yscale=3] {rectangleblue};
    }}\\
     \resizebox{120pt}{!}{
    \tikz[]{
    \pic[scale=3] {kiteblue};
    }}
         & 
   \resizebox{120pt}{!}{
    \tikz[]{
    \pic[scale=3] {kitered};
    }}      
    \end{tabular}
\end{center}

 Assign all external supertile edges a weight of 3. Internal edges that are incident to a red vertex are weighted 1 and internal edges that are incident to a blue vertex are weighted 2. Assign the remaining internal edge weights as above. This provides a solution to the 1-2-3 problem.
\end{proposition}

\begin{proof}
When combined, each external supertile vertex has a degree of 8.

At each red vertex, the weight on external supertile edges is 3 and the weight on the internal edges is 1. Since there are at most 4 external supertile edges, at each red vertex the sum is at most 
\[
4 \cdot 3+4\cdot 1=16 \,.
\]

At each blue vertex, the weight on external supertile edges is 3 and the weight on the internal edges is 2. Since there are at least 2 external supertile edges, at each blue vertex the sum is at least
\[
2 \cdot 3+6\cdot 2=18 \,.
\]
It follows that red vertices and blue vertices will never have the same sum.
\end{proof}

\begin{remark} Considering the cases of 2,3, or 4 supertile edges at each point, a simple case by case analysis shows that the sum at each red vertex is 12 or 14 or 16, while the sum at each blue vertex is 18 or 19 or 20.
\end{remark}

\tikzset{rectangleblue/.pic=
{
        \draw[line width=3mm] (-1.12,-2.24)--(-1.12,2.24)--(1.12,2.24)--(1.12,-2.24)--(-1.12,-2.24);
        \draw[line width=0.6mm] (-1.12,2.24)--(1.12,-2.24);
        \draw[line width=0.6mm] (-0.67,1.35)--(1.12,2.24);
        \draw[line width=0.6mm] (0.23,1.79)--(1.12,0)--(0.23,-0.45);
        \draw[line width=0.6mm] (-0.67,1.35)--(1.12,0);
        \draw[line width=0.6mm] (-1.12,-2.24)--(0.67,-1.35);
        \draw[line width=0.6mm] (-0.23,-1.79)--(-1.12,0)--(-0.23,0.45);
        \draw[line width=0.6mm] (-1.12,0)--(0.67,-1.35);

        \foreach \i/\j in {-1.12/-1.12,-1.12/1.12,0/2.24,1.12/1.12,1.12/-1.12,0/-2.24}
        {
        \draw(\i,\j) node[shape=circle,fill=white,font={\Huge}]{\color{black}$3$};
        }
        
        \draw (-0.9,1.8) node[shape=circle,fill=white,font={\Huge}] {\color{blue}$2$};
        \draw (-0.675,-2.02) node[shape=circle,fill=white,font={\Huge}] {\color{blue}$2$};
        \draw (0.675,0.9) node[shape=circle,fill=white,font={\Huge}] {\color{blue}$2$};
        \draw (0.675,-0.225) node[shape=circle,fill=white,font={\Huge}] {\color{blue}$2$};
        \draw (0.675,2.02) node[shape=circle,fill=white,font={\Huge}] {\color{red}$1$};
        \draw (0.9,-1.8) node[shape=circle,fill=white,font={\Huge}] {\color{red}$1$};
        \draw (-0.675,0.225) node[shape=circle,fill=white,font={\Huge}] {\color{red}$1$};
        \draw (-0.675,-0.9) node[shape=circle,fill=white,font={\Huge}] {\color{red}$1$};
        \draw (0.22,-1.57) node[shape=circle,fill=white,font={\Huge}] {\color{black}$1$};
        \draw (0.45,-0.9) node[shape=circle,fill=white,font={\Huge}]  {\color{black}$3$};
        \draw (-0.22,1.57) node[shape=circle,fill=white,font={\Huge}]  {\color{black}$1$};
        \draw (0.23,0.67) node[shape=circle,fill=white,font={\Huge}]  {\color{blue}$2$};
        \draw (-0.45,0.9) node[shape=circle,fill=white,font={\Huge}]  {\color{black}$1$};
        \draw (-0.23,-0.67) node[shape=circle,fill=white,font={\Huge}]  {\color{red}$1$};
        \draw (0,0) node[shape=circle,fill=white,font={\Huge}]  {\color{black}$2$};

        \node[shape=circle, draw=blue, fill=blue,minimum size=7mm] at (-1.12,2.24){};
        \node[shape=circle, draw=blue, fill=blue,minimum size=7mm] at (1.12,0){};
        \node[shape=circle, draw=blue, fill=blue,minimum size=7mm] at (-1.12,-2.24){};
        \node[shape=circle, draw=red, fill=red,minimum size=7mm] at (-1.12,0){};
        \node[shape=circle, draw=red, fill=red,minimum size=7mm] at (1.12,2.24){};
        \node[shape=circle, draw=red, fill=red,minimum size=7mm] at (1.12,-2.24){};

        \node[shape=circle, draw=black, fill=white, text=black, font={\Huge}] at (-0.23,-1.79) {4};
        \node[shape=circle, draw=black, fill=white, text=black, font={\Huge}] at (0.23,-0.45) {7};
        \node[shape=circle, draw=black, fill=white, text=black, font={\Huge}] at (0.67,-1.35) {6};
        \node[shape=circle, draw=black, fill=white, text=black, font={\Huge}] at (0.23,1.79) {4};
        \node[shape=circle, draw=black, fill=white, text=black, font={\Huge}] at (-0.67,1.35) {6};
        \node[shape=circle, draw=black, fill=white, text=black, font={\Huge}] at (-0.23,0.45) {4};
}}

\tikzset{kiteblue/.pic={
        \draw[line width=3mm] (-1.12,-2.24)--(-1.12,2.24)--(2.46,-0.45)--(1.12,-2.24)--(-1.12,-2.24);
        \draw[line width=0.6mm] (-1.12,2.24)--(1.12,-2.24);
        \draw[line width=0.6mm] (-1.12,0)--(0.67,0.9);
        \draw[line width=0.6mm] (-1.12,-2.24)--(2.46,-0.45);
        \draw[line width=0.6mm] (-0.22,-1.79)--(-1.12,0)--(0.67,-1.34);
        \draw[line width=0.6mm] (1.57,-0.89)--(0.67,0.9)--(0.67,-1.34);

        \foreach \i/\j in {0.225/-1.565,1.12/-1.115,0.22/-0.445}
        {
        \draw (\i,\j) node[shape=circle,fill=white,font={\Huge}] {\color{black}$1$};
        }

        \foreach \i/\j in {-1.12/-1.12,-1.12/1.12,0/-2.24,
        -0.225/1.57,1.57/0.225,1.79/-1.35}
        {
        \draw(\i,\j) 
        node[shape=circle,fill=white,font={\Huge}]{\color{black}$3$};
        }
        
        \foreach \i/\j in {-0.675/1.345,-0.67/-2.015,2.015/-0.67}
        {
        \draw (\i,\j) node[shape=circle,fill=white,font={\Huge}] {\color{blue}$2$};
        }
        \foreach \i/\j in {-0.67/-0.895,-0.225/-0.67,-0.675/0.225,
        0/0.675,0.67/-0.22,1.12/0,0.895/-1.79}
        {
        \draw (\i,\j) node[shape=circle,fill=white,font={\Huge}] {\color{red}$1$};
        }

        \node[shape=circle, draw=black, fill=white, text=black, font={\Huge}] at (-0.22,-1.79) {4};
        \node[shape=circle, draw=black, fill=white, text=black, font={\Huge}] at (1.57,-0.89) {4};
        \node[shape=circle, draw=blue, fill=blue,minimum size=7mm] at (-1.12,2.24){};
        \node[shape=circle, draw=blue, fill=blue,minimum size=7mm] at (2.46,-0.45){};
        \node[shape=circle, draw=blue, fill=blue,minimum size=7mm] at (-1.12,-2.24){};
        \node[shape=circle, draw=red, fill=red,minimum size=7mm] at (0.67,0.9){};
        \node[shape=circle, draw=red, fill=red,minimum size=7mm] at (1.12,-2.24){};
        \node[shape=circle, draw=red, fill=red,minimum size=7mm] at (-1.12,0){};
        \node[shape=circle, draw=black, fill=white, text=black, font={\Huge}] at (-0.23,0.45) {5};
        \node[shape=circle, draw=black, fill=white, text=black, font={\Huge}] at (0.67,-1.34) {6};
}
}

\tikzset{kitered/.pic={
        \draw[line width=3mm] (-1.12,-2.24)--(-1.12,2.24)--(2.46,-0.45)--(1.12,-2.24)--(-1.12,-2.24);
        \draw[line width=0.6mm] (-1.12,2.24)--(1.12,-2.24);
        \draw[line width=0.6mm] (-1.12,0)--(0.67,0.9);
        \draw[line width=0.6mm] (-1.12,-2.24)--(2.46,-0.45);
        \draw[line width=0.6mm] (-0.22,-1.79)--(-1.12,0)--(0.67,-1.34);
        \draw[line width=0.6mm] (1.57,-0.89)--(0.67,0.9)--(0.67,-1.34);

        \foreach \i/\j in {0.225/-1.565,1.12/-1.115,0.22/-0.445}
        {
        \draw (\i,\j) node[shape=circle,fill=white,font={\Huge}] {\color{black}$1$};
        }

        \foreach \i/\j in {-1.12/-1.12,-1.12/1.12,0/-2.24,
        -0.225/1.57,1.57/0.225,1.79/-1.35}
        {
        \draw(\i,\j) 
        node[shape=circle,fill=white,font={\Huge}]{\color{black}$3$};
        }
        
        \foreach \i/\j in {-0.675/1.345,-0.67/-2.015,2.015/-0.67}
        {
        \draw (\i,\j) node[shape=circle,fill=white,font={\Huge}] {\color{red}$1$};
        }
        \foreach \i/\j in {-0.67/-0.895,-0.225/-0.67,-0.675/0.225,
        0/0.675,0.67/-0.22,1.12/0,0.895/-1.79}
        {
        \draw (\i,\j) node[shape=circle,fill=white,font={\Huge}] {\color{blue}$2$};
        }

        \node[shape=circle, draw=black, fill=white, text=black, font={\Huge}] at (-0.22,-1.79) {4};
        \node[shape=circle, draw=black, fill=white, text=black, font={\Huge}] at (1.57,-0.89) {4};
        \node[shape=circle, draw=red, fill=red,minimum size=7mm] at (-1.12,2.24){};
        \node[shape=circle, draw=red, fill=red,minimum size=7mm] at (2.46,-0.45){};
        \node[shape=circle, draw=red, fill=red,minimum size=7mm] at (-1.12,-2.24){};
        \node[shape=circle, draw=blue, fill=blue,minimum size=7mm] at (0.67,0.9){};
        \node[shape=circle, draw=blue, fill=blue,minimum size=7mm] at (1.12,-2.24){};
        \node[shape=circle, draw=blue, fill=blue,minimum size=7mm] at (-1.12,0){};
        \node[shape=circle, draw=black, fill=white, text=black, font={\Huge}] at (-0.23,0.45) {6};
        \node[shape=circle, draw=black, fill=white, text=black, font={\Huge}] at (0.67,-1.34) {9};
}
}


\begin{center}
\resizebox{435pt}{!}{
\tikz{
\clip (-3.5,-1) rectangle(44,39);

\pic[rotate=270,xscale=-3,yscale=3] at (0,0) {rectangleblue};
\pic[rotate=270,scale=3] at (13.44,0) {kitered};
\pic[rotate=180,scale=3] at (23.52,-3.36) {kiteblue};
\pic[rotate=90,xscale=-3,yscale=3] at (33.6,0) {rectangleblue};
\pic[rotate=270,scale=3] at (47.04,0) {rectangleblue};
\pic[rotate=270,scale=3] at (0,6.72) {rectangleblue};
\pic[rotate=90,xscale=-3,yscale=3] at (13.44,6.72) {rectangleblue};
\pic[rotate=270,scale=3] at (26.88,6.72) {rectangleblue};
\pic[rotate=90, xscale=-3,yscale=3] at (40.32,6.72) {rectangleblue};
\pic[scale=3] at (50.4,10.08) {kitered};
\pic[rotate=216.86,scale=3] at (0,13.44) {kiteblue};
\pic[rotate=-53.14,scale=3] at (10.08,16.8) {kitered};
\pic[rotate=180,scale=3] at (16.8,16.8) {rectangleblue};
\pic [rotate=180, xscale=-3,yscale=3]at (23.52,16.8) {rectangleblue};
\pic[scale=3] at (30.24,16.8) {kitered};
\pic[rotate=90,scale=3] at (40.32,13.44) {kiteblue};
\pic[rotate=126.86,scale=3] at (-3.36,23.52) {kitered};
\pic[rotate=126.86,xscale=-3,yscale=3] at (6.1,22.15) {rectangleblue};
\pic[rotate=-53.14,scale=3] at (16.8,30.24) {kiteblue};
\pic[scale=3] at (23.52,30.24) {kiteblue};
\pic[rotate=-126.86,scale=3] at (34.25,22.15) {rectangleblue};
\pic[rotate=180,scale=3] at (43.68,23.52) {kitered};
\pic[rotate=270,scale=3] at (53.76,20.16) {kiteblue};
\pic[rotate=216.86,scale=3] at (-6.72,33.6) {kiteblue};
\pic[rotate=126.86,xscale=-3,yscale=3] at (7.4,31.56) {rectangleblue};
\pic[rotate=126.86,scale=3] at (3.36,36.96) {rectangleblue};
\pic[rotate=126.86, xscale=-3,yscale=3] at (-0.65,42.4) {rectangleblue};
\pic[rotate=37,scale=3] at (13.44,40.32) {kitered};
\pic[rotate=270,scale=3] at (26.88,40.32) {kitered};
\pic[rotate=233,scale=3] at (32.95,31.65) {rectangleblue};
\pic[rotate=180,scale=3] at (36.96,36.96) {kiteblue};
\pic[rotate=180,xscale=-3,yscale=3] at (43.68,36.96) {rectangleblue};
\pic[rotate=180,xscale=-3,yscale=3] at (50.4,30.24) {rectangleblue};
\pic[scale=3] at (50.4,43.68) {kiteblue};

\node[shape=circle, draw=blue, fill=blue!30, text=black, font={\Huge}] at (0,3.36){$18$};
\node[shape=circle, draw=red, fill=red!60, text=black, font={\Huge}] at (6.72,3.36){$16$};
\node[shape=circle, draw=blue, fill=blue!30, text=black, font={\Huge}] at (13.44,3.36){$18$};
\node[shape=circle, draw=red, fill=red!60, text=black, font={\Huge}] at (20.16,3.36){$16$};
\node[shape=circle, draw=blue, fill=blue!30, text=black, font={\Huge}] at (26.88,3.36){$19$};
\node[shape=circle, draw=red, fill=red!60, text=black, font={\Huge}] at (33.6,3.36){$14$};
\node[shape=circle, draw=blue, fill=blue!30, text=black, font={\Huge}] at (40.32,3.36){$19$};
\node[shape=circle, draw=red, fill=red!60, text=black, font={\Huge}] at (47.04,3.36){$14$};

\node[shape=circle, draw=red, fill=red!60, text=black, font={\Huge}] at (0,10.08){$12$};
\node[shape=circle, draw=blue, fill=blue!30, text=black, font={\Huge}] at (6.72,10.08){$20$};
\node[shape=circle, draw=red, fill=red!60, text=black, font={\Huge}] at (13.44,10.08){$14$};
\node[shape=circle, draw=blue, fill=blue!30, text=black, font={\Huge}] at (20.16,10.08){$20$};
\node[shape=circle, draw=red, fill=red!60, text=black, font={\Huge}] at (26.88,10.08){$14$};
\node[shape=circle, draw=blue, fill=blue!30, text=black, font={\Huge}] at (33.6,10.08){$20$};
\node[shape=circle, draw=red, fill=red!60, text=black, font={\Huge}] at (40.32,10.08){$12$};
\node[shape=circle, draw=blue, fill=blue!30, text=black, font={\Huge}] at (47.04,10.08){$19$};

\node[shape=circle, draw=red, fill=red!60, text=black, font={\Huge}] at (20.16,16.8){$12$};
\node[shape=circle, draw=blue, fill=blue!30, text=black, font={\Huge}] at (20.16,23.52){$20$};
\node[shape=circle, draw=blue, fill=blue!30, text=black, font={\Huge}] at (20.16,36.96){$20$};
\node[shape=circle, draw=red, fill=red!60, text=black, font={\Huge}] at (20.16,30.24){$12$};

\node[shape=circle, draw=red, fill=red!60, text=black, font={\Huge}] at (47.04,16.8){$16$};
\node[shape=circle, draw=blue, fill=blue!30, text=black, font={\Huge}] at (47.04,23.52){$19$};
\node[shape=circle, draw=red, fill=red!60, text=black, font={\Huge}] at (47.04,30.24){$14$};
\node[shape=circle, draw=red, fill=red!60, text=black, font={\Huge}] at (40.32,36.96){$12$};
\node[shape=circle, draw=blue, fill=blue!30, text=black, font={\Huge}] at (47.04,36.96){$19$};
\node[shape=circle, draw=red, fill=red!60, text=black, font={\Huge}] at (13.44,23.52){$16$};
\node[shape=circle, draw=blue, fill=blue!30, text=black, font={\Huge}] at (40.32,30.24){$20$};
\node[shape=circle, draw=red, fill=red!60, text=black, font={\Huge}] at (26.88,23.52){$16$};
\node[shape=circle, draw=red, fill=red!60, text=black, font={\Huge}] at (37.62,15.45){$14$};

\node[shape=circle, draw=blue, fill=blue!30, text=black, font={\Huge}] at (26.88,16.8){$18$};
\node[shape=circle, draw=red, fill=red!60, text=black, font={\Huge}] at (2.7,15.45){$14$};
\node[shape=circle, draw=blue, fill=blue!30, text=black, font={\Huge}] at (13.44,16.8){$18$};

\node[shape=circle, draw=blue, fill=blue!30, text=black, font={\Huge}] at (32.25,19.5){$18$};
\node[shape=circle, draw=blue, fill=blue!30, text=black, font={\Huge}] at (8.07,19.5){$18$};

\node[shape=circle, draw=blue, fill=blue!30, text=black, font={\Huge}] at (41.66,20.83){$19$};
\node[shape=circle, draw=blue, fill=blue!30, text=black, font={\Huge}] at (-1.34,20.83){$19$};
\node[shape=circle, draw=red, fill=red!60, text=black, font={\Huge}] at (25.5,32.94){$14$};
\node[shape=circle, draw=red, fill=red!60, text=black, font={\Huge}] at (14.82,32.94){$14$};
\node[shape=circle, draw=blue, fill=blue!30, text=black, font={\Huge}] at (30.9,28.89){$19$};
\node[shape=circle, draw=red, fill=red!60, text=black, font={\Huge}] at (36.28,24.86){$14$};
\node[shape=circle, draw=blue, fill=blue!30, text=black, font={\Huge}] at (9.42,28.89){$19$};
\node[shape=circle, draw=blue, fill=blue!30, text=black, font={\Huge}] at (29.58,38.34){$19$};
\node[shape=circle, draw=blue, fill=blue!30, text=black, font={\Huge}] at (10.74,38.34){$19$};
\node[shape=circle, draw=red, fill=red!60, text=black, font={\Huge}] at (34.95,34.29){$12$};
\node[shape=circle, draw=red, fill=red!60, text=black, font={\Huge}] at (4.04,24.86){$14$};
\node[shape=circle, draw=blue, fill=blue!30, text=black, font={\Huge}] at (0,30.24){$20$};
\node[shape=circle, draw=red, fill=red!60, text=black, font={\Huge}] at (5.37,34.29){$12$};
}}
\end{center}

\subsection{Half-Hex}

\begin{proposition}
    Consider the non-periodic tiling of the plane with the \href{https://tilings.math.uni-bielefeld.de/substitution/half-hex/}{half-hex} in \cite{TilingEncyclopedia}. We identify a hexagon supertile, as shown in \cite{HalfHex}, comprised of 8 half hexagons. The external supertile vertices are 2 colourable. Supertiles are joined at vertices of the same colour which results in 3 possible rotations of the supertile.

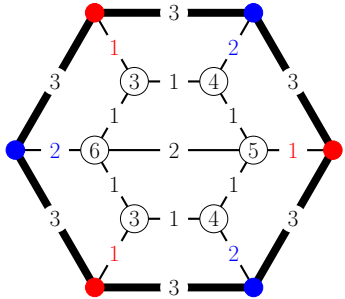
\tikzset{halfhex/.pic={

        \draw[line width=3mm] (-3,-5.2)--(-6,0)--(-3,5.2)--(3,5.2)--(6,0)--(3,-5.2)--(-3,-5.2);
        \draw[line width=0.8mm] (-1.5,-2.6)--(-3,0)--(-1.5,2.6)--(1.5,2.6)--(3,0)--(1.5,-2.6)--(-1.5,-2.6);
        \draw[line width=0.8mm] (-3,5.2)--(3,-5.2);
        \draw[line width=0.8mm] (-6,0)--(-3,0);
        \draw[line width=0.8mm] (-3,-5.2)--(-1.5,-2.6);
        \draw[line width=0.8mm] (3,0)--(6,0);
        \draw[line width=0.8mm] (1.5,2.6)--(3,5.2);

        \foreach \i/\j in {0/-5.2,-4.5/-2.6,-4.5/2.6,0/5.2,4.5/2.6,4.5/-2.6}
        {
        \draw (\i,\j) node[shape=circle,fill=white,font={\Huge}] {\color{black}$3$};
        }
        \foreach\i/\j in {0/2.6,2.25/1.3,2.25/-1.3,0/-2.6,-2.25/-1.3,-2.25/1.3}
        {
        \draw (\i,\j) node[shape=circle,fill=white,font={\Huge}] {\color{black}$1$};
        }
        \foreach \i/\j in {-2.25/3.9,4.5/0,-2.25/-3.9}
        {
        \draw (\i,\j) node[shape=circle,fill=white,font={\Huge}] {\color{red}$1$};
        }
        \foreach \i/\j in {2.25/3.9,2.25/-3.9,-4.5/0}
        {
        \draw (\i,\j) node[shape=circle,fill=white,font={\Huge}] {\color{blue}$2$};
        }
        \draw (0,0) node[shape=circle,fill=white,font={\Huge}] {\color{black}$2$};
        
        \node[shape=circle, draw=red, fill=red,minimum size=7mm] at (-3,5.2){};
        \node[shape=circle, draw=red, fill=red,minimum size=7mm] at (6,0){};
        \node[shape=circle, draw=red, fill=red,minimum size=7mm] at (-3,-5.2){};
        \node[shape=circle, draw=blue, fill=blue,minimum size=7mm] at (3,5.2){};
        \node[shape=circle, draw=blue, fill=blue,minimum size=7mm] at (3,-5.2){};
        \node[shape=circle, draw=blue, fill=blue,minimum size=7mm] at (-6,0){};
        \node[shape=circle, draw=black, fill=white, text=black,font={\Huge}] at (-1.5,2.6){$5$};
        \node[shape=circle, draw=black, fill=white, text=black,font={\Huge}] at (1.5,2.6){$4$};
        \node[shape=circle, draw=black, fill=white, text=black,font={\Huge}] at (3,0){$3$};
        \node[shape=circle, draw=black, fill=white, text=black,font={\Huge}] at (1.5,-2.6){$6$};
        \node[shape=circle, draw=black, fill=white, text=black,font={\Huge}] at (-1.5,-2.6){$3$};
        \node[shape=circle, draw=black, fill=white, text=black,font={\Huge}] at (-3,0){$4$};
    }}

\begin{center}
   \resizebox{130pt}{!}{
   \tikz{
   \pic[scale=1,rotate=240] {halfhex};
   }
   }
\end{center}

    The external supertile vertices are 2-colourable. Assign a weight of 3 to each external edge of the supertile. For internal edges incident to a red vertex, assign a weight of 1. For internal edges incident to a blue vertex assign a weight of 2. Assign the weights of the remaining internal edges as above. This provides a solution to the 1-2-3 problem.
\end{proposition}

\begin{proof}
    When combined, each external supertile vertex connects exactly 3 supertiles and each external vertex is incident to 1 internal edge. Thus, each vertex will be incident to 3 external edges, weighted 3, and 3 internal edges, weighted 2 for blue or 1 for red. Therefore, all blue vertices will have a sum of 15 and all red vertices will have a sum of 12. Also notice that the greatest internal sum is 6. 
\end{proof}

\tikzset{halfhex/.pic={

        \draw[line width=3mm] (-3,-5.2)--(-6,0)--(-3,5.2)--(3,5.2)--(6,0)--(3,-5.2)--(-3,-5.2);
        \draw[line width=0.8mm] (-1.5,-2.6)--(-3,0)--(-1.5,2.6)--(1.5,2.6)--(3,0)--(1.5,-2.6)--(-1.5,-2.6);
        \draw[line width=0.8mm] (-3,5.2)--(3,-5.2);
        \draw[line width=0.8mm] (-6,0)--(-3,0);
        \draw[line width=0.8mm] (-3,-5.2)--(-1.5,-2.6);
        \draw[line width=0.8mm] (3,0)--(6,0);
        \draw[line width=0.8mm] (1.5,2.6)--(3,5.2);

        \foreach \i/\j in {0/-5.2,-4.5/-2.6,-4.5/2.6,0/5.2,4.5/2.6,4.5/-2.6}
        {
        \draw (\i,\j) node[shape=circle,fill=white,font={\Huge}] {\color{black}$3$};
        }
        \foreach\i/\j in {0/2.6,2.25/1.3,2.25/-1.3,0/-2.6,-2.25/-1.3,-2.25/1.3}
        {
        \draw (\i,\j) node[shape=circle,fill=white,font={\Huge}] {\color{black}$1$};
        }
        \foreach \i/\j in {-2.25/3.9,4.5/0,-2.25/-3.9}
        {
        \draw (\i,\j) node[shape=circle,fill=white,font={\Huge}] {\color{red}$1$};
        }
        \foreach \i/\j in {2.25/3.9,2.25/-3.9,-4.5/0}
        {
        \draw (\i,\j) node[shape=circle,fill=white,font={\Huge}] {\color{blue}$2$};
        }
        \draw (0,0) node[shape=circle,fill=white,font={\Huge}] {\color{black}$2$};
        
        \node[shape=circle, draw=red, fill=red!60,minimum size=7mm,font={\Huge}] at (-3,5.2){12};
        \node[shape=circle, draw=red, fill=red!60,minimum size=7mm,font={\Huge}] at (6,0){12};
        \node[shape=circle, draw=red, fill=red!60,minimum size=7mm,font={\Huge}] at (-3,-5.2){12};
        \node[shape=circle, draw=blue, fill=blue!30,minimum size=7mm,font={\Huge}] at (3,5.2){15};
        \node[shape=circle, draw=blue, fill=blue!30,minimum size=7mm,font={\Huge}] at (3,-5.2){15};
        \node[shape=circle, draw=blue, fill=blue!30,minimum size=7mm,font={\Huge}] at (-6,0){15};
        \node[shape=circle, draw=black, fill=white, text=black,font={\Huge}] at (-1.5,2.6){$5$};
        \node[shape=circle, draw=black, fill=white, text=black,font={\Huge}] at (1.5,2.6){$4$};
        \node[shape=circle, draw=black, fill=white, text=black,font={\Huge}] at (3,0){$3$};
        \node[shape=circle, draw=black, fill=white, text=black,font={\Huge}] at (1.5,-2.6){$6$};
        \node[shape=circle, draw=black, fill=white, text=black,font={\Huge}] at (-1.5,-2.6){$3$};
        \node[shape=circle, draw=black, fill=white, text=black,font={\Huge}] at (-3,0){$4$};
    }}

\begin{center}
\resizebox{400pt}{!}{
    \tikz{
    \clip (-26.1,3.2) rectangle (26.1,32);
    \foreach \i/\j in {0/0, 0/10.4, 0/20.8, 0/31.2, 0/41.6,0/52,
    36/10.4, -36/10.4,
    -18/20.8, 18/20.8, -18/41.6, 18/41.6,
    -36/31.2, -36/41.6, -36/52,
    36/31.2,36/41.6,36/52}
    {
    \pic[scale=1,rotate=240] at (\i,\j) {halfhex};
    }
    \foreach \i/\j in {-9/5.2,-18/10.4,-27/15.6,-36/20.8,
    27/5.2, 9/15.6, -9/26, -9/46.8,
    9/36.4, 9/57.2,
    -18/52, 18/31.2,
    27/26, 27/46.8,
    -27/36.4,-27/57.2}
    {
    \pic[scale=1,rotate=120] at (\i,\j) {halfhex};
    }
    \foreach \i/\j in {9/5.2,18/10.4,27/15.6,36/20.8,
    -9/15.6, 9/26,
    -27/26, -9/36.4, -9/57.2,
    9/46.8,
    -18/31.2, 18/52,
    27/36.4, 27/57.2,
    -27/46.8}
    {
    \pic[scale=1,rotate=0] at (\i,\j) {halfhex};
    }
    \foreach \i/\j in {-18/0,18/0,-36/0,36/0}
    {
    \pic[scale=1,rotate=240] at (\i,\j) {halfhex};
    }
    \foreach \i/\j in {-27/5.2}
    {
    \pic[scale=1,rotate=0] at (\i,\j) {halfhex};
    }
    }}
\end{center}

\subsection{Ammann-Beenker Rhomb-Triangle tiling}

\begin{proposition}
    Consider the non-periodic tiling of the plane with the \href{https://tilings.math.uni-bielefeld.de/substitution/ammann-beenker-rhomb-triangle/}{Ammann-Beenker rhomb triangle} in \cite{TilingEncyclopedia}. There are 2 prototiles, a triangle with side lengths of 1, 1, and $\sqrt{2}$ and a rhombus with 4 sides of length 1. We identify the following 2 supertiles, the diamond consisting of 6 triangles and 4 rhombi, and the rhombus consisting of 4 triangles and 3 rhombi. The external vertices are 2 colourable. The bipartite borders of these supertiles are explored in \cite{ColouringsofAperiodicTilings}. Assign a weight of 1 to each edge with a length of 1 and a weight of 2 to each edge with a length of $\sqrt{2}$. This provides a solution to the 1-2-3 problem.

    \tikzset{diamond/.pic={
        \draw[line width=2mm] (0,0)--(1.70710678,1.70710678)--(3.41421356,0)--(1.70710678,-1.70710678)--(0,0);
        \draw[line width=0.7mm] (1,0)--(1.70710678,0.70710678)--(2.41421356,0)--(1.70710678,-0.70710678)--(1,0);
        \draw[line width=0.7mm] (0,0)--(3.41421356,0);
        \draw[line width=0.7mm] (1,1)--(1,-1);
        \draw[line width=0.7mm] (1.70710678,1.70710678)--(1.70710678,0.70710678)--(2.70710678,0.70710678);
        \draw[line width=0.7mm] (1.70710678,-1.70710678)--(1.70710678,-0.70710678)--(2.70710678,-0.70710678);
        
        \draw[line width=3.5,->] (1,1)--(0.4,0.4);
        \draw[line width=3.5,->] (1,-1)--(0.4,-0.4);
        \draw[line width=3.5,->] (2.41421356,0)--(1.55710678,0);
        \draw[line width=3.5,->] (2.70710678,0.70710678)--(2.10710678,1.30710678);
        \draw[line width=3.5,->] (2.70710678,-0.70710678)--(2.10710678,-1.30710678);

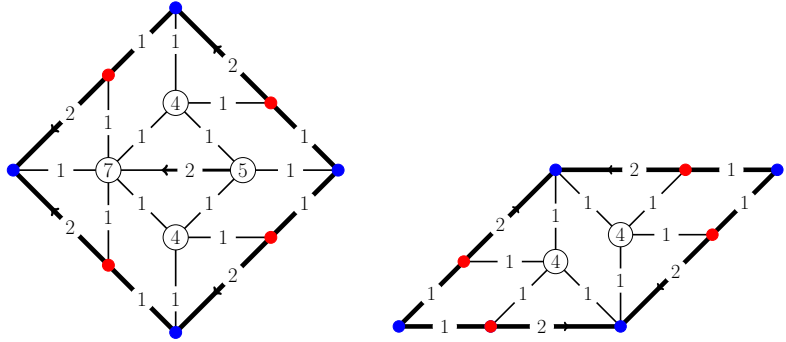
        \foreach \i/\j in {0.6,0.6,0.6/-0.6,1.85710678/0,2.30710678/1.10710678,2.30710678/-1.10710678}
        {
        \draw (\i,\j) node[shape=circle,fill=white,font={\Huge}]{\color{black}$2$};
        }
        \foreach \i/\j in {1.35355339/1.35355339,1/0.5,1.35355339/-1.35355339,1/-0.5,1.35355339/0.35355339,1.35355339/-0.35355339,1.70710678/-1.35355339,1.70710678/1.35355339,2.20710678/-0.70710678,2.20710678/0.70710678,2.06066017/0.35355339,2.06066017/-0.35355339,3.06066017/0.35355339,3.06066017/-0.35355339,2.91421356/0,0.5/0}
        {
        \draw (\i,\j) node[shape=circle,fill=white,font={\Huge}]{\color{black}$1$};
        }
        \node[shape=circle, draw=black, fill=white, text=black,font={\Huge}] at (1,0) {$7$};
        \node[shape=circle, draw=black, fill=white, text=black,font={\Huge}] at (2.41421356,0) {$5$};
        \node[shape=circle, draw=black, fill=white, text=black,font={\Huge}] at (1.70710678,0.70710678) {$4$};
        \node[shape=circle, draw=black, fill=white, text=black,font={\Huge}] at (1.70710678,-0.70710678) {$4$};  
        \foreach \i/\j in {0/0,1.70710678/1.70710678,3.41421356/0,1.70710678/-1.70710678}
        {
        \node[shape=circle, draw=blue, fill=blue,minimum size=5mm] at (\i,\j) {};
        }

         \foreach \i/\j in {1/1,2.70710678/0.70710678,2.70710678/-0.70710678,1/-1}
        {
        \node[shape=circle, draw=red, fill=red,minimum size=5mm] at (\i,\j) {};
        }

        \coordinate (-A) at (0,0);
        \coordinate (-B) at (1.70710678,1.70710678);
        \coordinate (-C) at (3.41421356,0);
        \coordinate (-D) at (1.70710678,-1.70710678);
         
}}

\tikzset{rhomb/.pic={
        \draw[line width=2mm] (0,0)--(1.70710678,1.70710678)--(4.12132034,1.70710678)--(2.41421356,0)--(0,0);
        \draw[line width=0.7mm] (0.70710678,0.70710678)--(1.70710678,0.70710678)--(1,0);
        \draw[line width=0.7mm] (1.70710678,1.70710678)--(1.70710678,0.70710678)--(2.41421356,0)--(2.41421356,1)--(1.70710678,1.70710678);
        \draw[line width=0.7mm] (3.12132034,1.70710678)--(2.41421356,1)--(3.41421356,1);
        
        \draw[line width=3,->] (1,0)--(1.85710678,0);
        \draw[line width=3,->] (0.70710678,0.70710678)--(1.30710678,1.30710678);
        \draw[line width=3,->] (3.41421356,1)--(2.81421356,0.4);
        \draw[line width=3,->] (3.12132034,1.70710678)--(2.26421356,1.70710678);
        \foreach \i/\j in {1.55710678/0,1.10710678/1.10710678,3.01421356/0.6,2.56421356/1.70710678}
        {
        \draw (\i,\j) node[shape=circle,fill=white,font={\Huge}]{\color{black}$2$};
        }
        \foreach \i/\j in {0.35355339/0.35355339,0.5/0,1.20710678/0.70710678,1.35355339/0.35355339,1.70710678/1.20710678,2.06066017/0.35355339,2.41421356/0.5,2.91421356/1,2.06066017/1.35355339,2.76776695/1.35355339,3.62132034/1.70710678,3.76776695/1.35355339}
        {
        \draw (\i,\j) node[shape=circle,fill=white,font={\Huge}]{\color{black}$1$};
        }
        
        \node[shape=circle, draw=black, fill=white, text=black,font={\Huge}] at (2.41421356,1) {$4$};
        \node[shape=circle, draw=black, fill=white, text=black,font={\Huge}] at (1.70710678,0.70710678) {$4$};
        \foreach \i/\j in {0/0,1.70710678/1.70710678,4.12132034/1.70710678,2.41421356/0,1/0}
        {
        \node[shape=circle, draw=blue, fill=blue,minimum size=5mm] at (\i,\j) {};
        }
        \foreach \i/\j in {0.70710678/0.70710678,3.12132034/1.70710678,3.41421356/1,1/0}
        {
        \node[shape=circle, draw=red, fill=red,minimum size=5mm] at (\i,\j) {};
        }
        \coordinate (-A) at (0,0);
        \coordinate (-B) at (1.70710678,1.70710678);
        \coordinate (-C) at (4.12132034,1.70710678);
        \coordinate (-D) at (2.41421356,0);
        
}}

\begin{center}
    \setlength{\tabcolsep}{8pt}
    \begin{tabular}{c c}
    \resizebox{130pt}{!}{
    \tikz{
    \pic[scale=4] {diamond};
    }
    }
         &  
    \resizebox{150pt}{!}{
    \tikz{
    \pic[scale=4] {rhomb};
    }
    }\\
    \end{tabular}
\end{center}
    
\end{proposition}

\begin{proof}
    Since each side of the supertiles has an edge of length $\sqrt{2}$ with a directed arrow and one of length 1 the supertiles can only join corner to corner. Notice that each corner of the supertiles are comprised of $45^\circ$ angles. When combined the corners will have degree 8 ($360^\circ \div 45^\circ = 8$) and will therefore have a sum $\ge8$. Each side of the supertiles has exactly 1 vertex between corners, coloured red and is made up of 2 $135^\circ$ angles and 2 $45^\circ$ angles. When combined, this vertex will have degree 4 and have exactly 1 edge with a weight of 2. Therefore, these middle, external vertices will sum to 5. Notice that there is an internal vertex with a sum of 5; however, this vertex is only ever adjacent to an external corner and will therefore not create a conflict.
\end{proof}

\tikzset{diamond/.pic={
        \draw[line width=2mm] (0,0)--(1.70710678,1.70710678)--(3.41421356,0)--(1.70710678,-1.70710678)--(0,0);
        \draw[line width=0.7mm] (1,0)--(1.70710678,0.70710678)--(2.41421356,0)--(1.70710678,-0.70710678)--(1,0);
        \draw[line width=0.7mm] (0,0)--(3.41421356,0);
        \draw[line width=0.7mm] (1,1)--(1,-1);
        \draw[line width=0.7mm] (1.70710678,1.70710678)--(1.70710678,0.70710678)--(2.70710678,0.70710678);
        \draw[line width=0.7mm] (1.70710678,-1.70710678)--(1.70710678,-0.70710678)--(2.70710678,-0.70710678);
        
        \draw[line width=3.5,->] (1,1)--(0.4,0.4);
        \draw[line width=3.5,->] (1,-1)--(0.4,-0.4);
        \draw[line width=3.5,->] (2.41421356,0)--(1.55710678,0);
        \draw[line width=3.5,->] (2.70710678,0.70710678)--(2.10710678,1.30710678);
        \draw[line width=3.5,->] (2.70710678,-0.70710678)--(2.10710678,-1.30710678);

        \foreach \i/\j in {0.6,0.6,0.6/-0.6,1.85710678/0,2.30710678/1.10710678,2.30710678/-1.10710678}
        {
        \draw (\i,\j) node[shape=circle,fill=white,font={\Huge}]{\color{black}$2$};
        }
        \foreach \i/\j in {1.35355339/1.35355339,1/0.5,1.35355339/-1.35355339,1/-0.5,1.35355339/0.35355339,1.35355339/-0.35355339,1.70710678/-1.35355339,1.70710678/1.35355339,2.20710678/-0.70710678,2.20710678/0.70710678,2.06066017/0.35355339,2.06066017/-0.35355339,3.06066017/0.35355339,3.06066017/-0.35355339,2.91421356/0,0.5/0}
        {
        \draw (\i,\j) node[shape=circle,fill=white,font={\Huge}]{\color{black}$1$};
        }
        \node[shape=circle, draw=black, fill=white, text=black,font={\Huge}] at (1,0) {$7$};
        \node[shape=circle, draw=black, fill=white, text=black,font={\Huge}] at (2.41421356,0) {$5$};
        \node[shape=circle, draw=black, fill=white, text=black,font={\Huge}] at (1.70710678,0.70710678) {$4$};
        \node[shape=circle, draw=black, fill=white, text=black,font={\Huge}] at (1.70710678,-0.70710678) {$4$};  
        \foreach \i/\j in {0/0,1/1,1.70710678/1.70710678,2.70710678/0.70710678,3.41421356/0,2.70710678/-0.70710678,1.70710678/-1.70710678,1/-1}
        {
        \node[shape=circle, draw=black, fill=black,minimum size=5mm] at (\i,\j) {};
        }
        \foreach \i/\j in {1/1,2.70710678/0.70710678,2.70710678/-0.70710678,1/-1}
        {
        \node[shape=circle, draw=red, fill=red!60,text=black,font={\Huge}] at (\i,\j) {5};
        }

        \coordinate (-A) at (0,0);
        \coordinate (-B) at (1.70710678,1.70710678);
        \coordinate (-C) at (3.41421356,0);
        \coordinate (-D) at (1.70710678,-1.70710678);
         
}}

\tikzset{rhomb/.pic={
        \draw[line width=2mm] (0,0)--(1.70710678,1.70710678)--(4.12132034,1.70710678)--(2.41421356,0)--(0,0);
        \draw[line width=0.7mm] (0.70710678,0.70710678)--(1.70710678,0.70710678)--(1,0);
        \draw[line width=0.7mm] (1.70710678,1.70710678)--(1.70710678,0.70710678)--(2.41421356,0)--(2.41421356,1)--(1.70710678,1.70710678);
        \draw[line width=0.7mm] (3.12132034,1.70710678)--(2.41421356,1)--(3.41421356,1);
        
        \draw[line width=3,->] (1,0)--(1.85710678,0);
        \draw[line width=3,->] (0.70710678,0.70710678)--(1.30710678,1.30710678);
        \draw[line width=3,->] (3.41421356,1)--(2.81421356,0.4);
        \draw[line width=3,->] (3.12132034,1.70710678)--(2.26421356,1.70710678);
        \foreach \i/\j in {1.55710678/0,1.10710678/1.10710678,3.01421356/0.6,2.56421356/1.70710678}
        {
        \draw (\i,\j) node[shape=circle,fill=white,font={\Huge}]{\color{black}$2$};
        }
        \foreach \i/\j in {0.35355339/0.35355339,0.5/0,1.20710678/0.70710678,1.35355339/0.35355339,1.70710678/1.20710678,2.06066017/0.35355339,2.41421356/0.5,2.91421356/1,2.06066017/1.35355339,2.76776695/1.35355339,3.62132034/1.70710678,3.76776695/1.35355339}
        {
        \draw (\i,\j) node[shape=circle,fill=white,font={\Huge}]{\color{black}$1$};
        }
        
        \node[shape=circle, draw=black, fill=white, text=black,font={\Huge}] at (2.41421356,1) {4};
        \node[shape=circle, draw=black, fill=white, text=black,font={\Huge}] at (1.70710678,0.70710678) {4};
        \foreach \i/\j in {0/0,0.70710678/0.70710678,1.70710678/1.70710678,3.12132034/1.70710678,4.12132034/1.70710678,3.41421356/1,2.41421356/0,1/0}
        {
        \node[shape=circle, draw=black, fill=black,minimum size=5mm] at (\i,\j) {};
        }
        \foreach \i/\j in {0.70710678/0.70710678,3.12132034/1.70710678,3.41421356/1,1/0}
        {
        \node[shape=circle, draw=red, fill=red!60,text=black,font={\Huge}] at (\i,\j) {5};
        }
        \coordinate (-A) at (0,0);
        \coordinate (-B) at (1.70710678,1.70710678);
        \coordinate (-C) at (4.12132034,1.70710678);
        \coordinate (-D) at (2.41421356,0);
        
}}

\begin{center}
    \resizebox{350pt}{!}{
    \tikz[scale=1]{
    \clip (0.6,10) rectangle (38,38);
    \pic (T1)[scale=4] {rhomb};
    \pic (T2)[xscale=-4,yscale=4,rotate=90] at (T1-A) {rhomb};
    \pic (T3)[scale=4] at (T1-D) {diamond};
    \pic (T4)[scale=4,rotate=45] at (T1-B) {diamond};
    \pic (T5)[scale=4,rotate=90] at (T3-C) {rhomb};
    \pic (T6)[scale=4,rotate=-45] at (T5-D) {diamond};
    \pic (T8)[scale=4,rotate=45] at (T6-C) {rhomb};
    \pic (T9)[scale=4,rotate=90] at (T2-D) {diamond};
    \pic (T10)[scale=4,rotate=90] at (T4-C) {rhomb};
    \pic (T11)[xscale=-4,yscale=4] at (T10-A) {rhomb};
    \pic (T12)[xscale=-4,yscale=4,rotate=90] at (T4-C) {rhomb};
    \pic (T13)[scale=4] at (T12-A) {rhomb};
    \pic (T14)[xscale=-4,yscale=4] at (T8-B) {rhomb};
    \pic (T15)[scale=4] at (T13-D) {diamond};
    \pic (T16)[scale=4,rotate=90] at (T15-C) {rhomb};
    \pic (T17)[xscale=-4,yscale=4,rotate=90] at (T16-A) {rhomb};
    \pic (T18)[scale=4,rotate=0] at (T17-A) {rhomb};
    \pic (T19)[scale=4,rotate=-90] at (T18-D) {rhomb};
    \pic (T20)[xscale=-4,yscale=4,rotate=45] at (T19-D) {diamond};
    \pic (T21)[xscale=-4,yscale=4,rotate=0] at (T19-C) {rhomb};
    \pic (T7)[xscale=-4,yscale=4,rotate=0] at (T21-D) {diamond};
    \pic (T22)[scale=4,rotate=90] at (T19-B) {diamond};
    \pic (T23)[scale=4,rotate=135] at (T10-B) {diamond};
    \pic (T24)[scale=4,rotate=90] at (T10-D) {diamond};
    \pic (T25)[scale=4,rotate=45] at (T13-B) {diamond};
    \pic (T26)[scale=4,rotate=90] at (T16-D) {diamond};
    \pic (T27)[scale=4,rotate=45] at (T18-B) {diamond};
    \pic (T28)[scale=4,rotate=0] at (T23-C) {rhomb};
    \pic (T29)[xscale=-4,yscale=4,rotate=0] at (T25-C) {rhomb};
    \pic (T30)[scale=4,rotate=0] at (T29-B) {diamond};
    \pic (T31)[xscale=-4,yscale=4,rotate=0] at (T27-C) {rhomb};
    \pic (T32)[scale=4,rotate=90] at (T28-B) {rhomb};
    \pic (T33)[scale=4,rotate=-90] at (T32-B) {diamond};
    \pic (T34)[scale=4,rotate=-45] at (T32-D) {diamond};
    \pic (T35)[xscale=-4,yscale=4,rotate=0] at (T34-B) {rhomb};
    \pic (T36)[scale=4,rotate=0] at (T24-C) {rhomb};
    \pic (T37)[xscale=-4,yscale=4,rotate=90] at (T24-C) {rhomb};
    \pic (T38)[scale=4,rotate=0] at (T35-B) {diamond};
    \pic (T39)[scale=4,rotate=45] at (T36-B) {diamond};
    \pic (T40)[scale=4,rotate=90] at (T26-C) {rhomb};
    \pic (T41)[scale=4,rotate=0] at (T40-D) {rhomb};
    \pic (T42)[scale=4,rotate=180] at (T41-B) {diamond};
    \pic (T43)[scale=4,rotate=-135] at (T41-D) {diamond};
    \pic (T44)[xscale=-4,yscale=4,rotate=90] at (T43-B) {rhomb};
    \pic (T45)[scale=4,rotate=-90] at (T44-B) {diamond};

    \foreach \i/\j in {T1/-B,T19/-D,T13/-D,T23/-A,T25/-A,T27/-A,T30/-A,T34/-A,T43/-A,T6/-A,T26/-A,T39/-A,T24/-A}
    {
    \node[shape=circle, draw=blue, fill=blue!30,text=black,font={\Huge}] at (\i\j) {11};
    }
    
     \foreach \i/\j in {T3/-B,T7/-B,T6/-B,T4/-B,T22/-B,T24/-B,T24/-D,T26/-D,T33/-B,T43/-B,T34/-B,T43/-D,T30/-B,T15/-B}
    {
    \node[shape=circle, draw=blue, fill=teal!40,text=black,font={\Huge}] at (\i\j) {10};
    }

     \foreach \i/\j in {T30/-D}
    {
    \node[shape=circle, draw=blue, fill=cyan!50,text=black,font={\Huge}] at (\i\j) {9};
    }

     \foreach \i/\j in {T4/-C,T15/-C,T34/-C,T26/-C}
    {
    \node[shape=circle, draw=blue, fill=cyan!10,text=black,font={\Huge}] at (\i\j) {8};
    }
    }
    }
    
\end{center}

\subsection{Ammann--Beenker Rhomb--Square tiling}

\begin{proposition}
    Consider the non-periodic tiling of the plane with \href{https://tilings.math.uni-bielefeld.de/substitution/ammann-beenker/}{Ammann--Beenker} with rhombi and squares prototiles in \cite{TilingEncyclopedia}. We identify the following 2 supertiles, the diamond consisting of 1 diamond, 4 half diamonds, and 4 rhombi, and the rhombus consisting of 4 half diamonds and 3 rhombi.

    \tikzset{square/.pic={
\draw[line width=2mm] (0,0)--(0,4)--(4,4)--(4,0)--(0,0);
\draw[loosely dashed, line width=2mm] (2,2)--(0,0);
\draw[loosely dashed,line width=2mm,->] (4,4)--(2,2);
\coordinate (-A) at (0,0);
\coordinate (-B) at (0,4);
\coordinate (-C) at (4,4);
\coordinate (-D) at (4,0);
\coordinate (-a) at (0,2);
\coordinate (-b) at (2,4);
\coordinate (-c) at (4,2);
\coordinate (-d) at (2,0);
}}

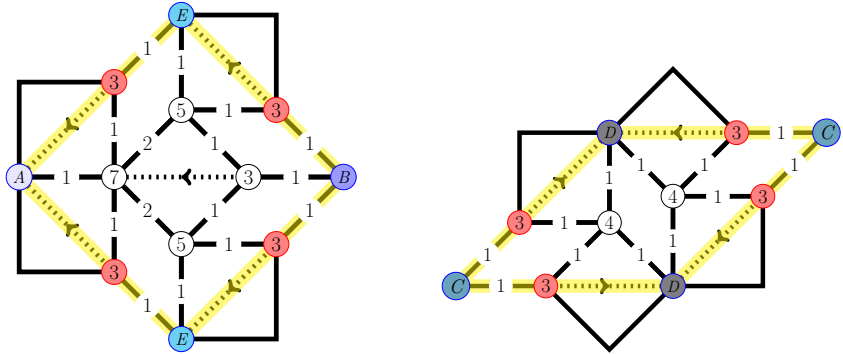
\tikzset{rhomb/.pic={
\draw[line width=2mm] (0,0)--(2.828427,2.828427)--(6.828427,2.828427)--(4,0)--(0,0);
\coordinate (-A) at (0,0);
\coordinate (-B) at (2.828427,2.828427);
\coordinate (-C) at (6.828427,2.828427);
\coordinate (-D) at (4,0);
\coordinate (-a) at (1.4142135,1.4142135);
\coordinate (-b) at (4.828427,2.828427);
\coordinate (-c) at (5.4142135,1.4142135);
\coordinate (-d) at (2,0);
}}

\begin{center}
    \setlength{\tabcolsep}{15pt}
    \begin{tabular}{cc}
    \resizebox{135pt}{!}{
    \tikz{
    \pic (T1) [scale=1,rotate=0] {square};
    \pic (T2) [scale=1,rotate=270] at (T1-A) {square};
    \pic (T3) [scale=1,rotate=45] at (T1-D) {rhomb};
    \pic (T4) [scale=1,rotate=-90] at (T3-C) {square};
    \pic (T5) [scale=1,rotate=-45] at (T1-D) {square};
    \pic (T6) [xscale=1,yscale=1,rotate=0] at (T5-D) {rhomb};
    \pic (T7) [xscale=-1,yscale=1,rotate=0] at (T6-C) {rhomb};
    \pic (T8) [xscale=1,yscale=1,rotate=-90] at (T1-D) {rhomb};
    \pic (T9) [xscale=1,yscale=1,rotate=0] at (T8-C) {square};
    \draw[draw=yellow,opacity=0.5,line width=6mm] (T1-A)--(T3-C)--(T6-C)--(T9-A)--(T1-A);
    \foreach \i/\j in {T1/-d,T1/-c,T2/-b,T4/-d,T4/-c,T5/-b,T5/-c,T7/-d,T9/-a,T9/-b,T3/-b,T7/-a,T6/-c,T8/-c}
    {
    \node[shape=circle,fill=white,text=black,font={\Huge}] at (\i\j) {$1$};
    }
    \foreach \i/\j in {T5/-a,T5/-d}
    {
    \node[shape=circle,fill=white,text=black,font={\Huge}] at (\i\j) {$2$};
    }
    \node[draw=black,shape=circle,fill=white,text=black,font={\Huge}] at (T1-D) {$7$};
    \node[draw=black,shape=circle,fill=white,text=black,font={\Huge}] at (T5-C) {$3$};
    \node[draw=black,shape=circle,fill=white,text=black,font={\Huge}] at (T5-B) {$5$};
    \node[draw=black,shape=circle,fill=white,text=black,font={\Huge}] at (T5-D) {$5$};
    \foreach \i/\j in {T3/-B,T4/-C,T9/-C,T2/-C}
    {
    \node[draw=red,shape=circle,fill=red!50,text=black,font={\Huge}] at (\i\j) {$3$};
    }
    \node[draw=blue,shape=circle,fill=blue!10,text=black,font={\huge}] at (T1-A) {A};
    \node[draw=blue,shape=circle,fill=blue!40,text=black,font={\huge}] at (T7-A) {B};
     \foreach \i/\j in {T4/-A,T9/-A}
    {
    \node[draw=blue,shape=circle,fill=cyan!50,text=black,font={\huge}] at (\i\j) {E};
    }
    }
    }
         &  
    \resizebox{153pt}{!}{
    \tikz{
    \pic (T1) [scale=1,rotate=-90] {rhomb};
    \pic (T2) [scale=1,rotate=180] at (T1-D) {rhomb};
    \pic (T3) [scale=1,rotate=180] at (T1-A) {square};
    \pic (T4) [scale=1,rotate=-45] at (T1-A) {square};
    \pic (T5) [scale=1,rotate=0] at (T1-B) {rhomb};
    \pic (T6) [scale=1,rotate=0] at (T1-C) {square};
    \pic (T7) [scale=1,rotate=135] at (T1-C) {square};
    \draw[draw=yellow,opacity=0.5,line width=6mm] (T2-C)--(T3-A)--(T5-C)--(T6-A)--(T2-C);
    \foreach \i/\j in {T5/-b,T5/-c,T2/-b,T2/-c,T2/-a,T2/-d,T5/-a,T5/-d,T1/-a,T1/-b,T1/-c,T1/-d}
    {
    \node[shape=circle,fill=white,text=black,font={\Huge}] at (\i\j) {$1$};
    }
    \node[draw=black,shape=circle,fill=white,text=black,font={\Huge}] at (T1-D) {$4$};
    \node[draw=black,shape=circle,fill=white,text=black,font={\Huge}] at (T1-B) {$4$};
    \foreach \i/\j in {T4/-C,T6/-C,T7/-C,T2/-D}
    {
    \node[draw=red,shape=circle,fill=red!50,text=black,font={\Huge}] at (\i\j) {$3$};
    }
    \node[draw=blue,shape=circle,fill=cyan!40!gray,text=black,font={\Huge}] at (T2-C) {C};
    \node[draw=blue,shape=circle,fill=cyan!40!gray,text=black,font={\Huge}] at (T5-C) {C};
    \foreach \i/\j in {T1/-A,T1/-C}
    {
    \node[draw=blue,shape=circle,fill=blue!2!gray,text=black,font={\huge}] at (\i\j) {D};
    }
    }
    }
    \end{tabular}
\end{center}
Any edge incident to an external supertile vertex is assigned a weight of 1 and the remaining edges are weighted as above. This provides a solution to the 1-2-3 problem.
\end{proposition}

\begin{proof}

Consider an edge $e$ in the tiling. If the edge $e$ is not incident to one of the corner vertices of the supertile, by inspection we can see that the endpoints of $e$ have different sum.

Therefore, we just need to show that the edges in or on the boundary of supertiles, which are incident to the corner vertices (labeled A,B,C,D,E) create no conflict. We look at these corner vertices.

Since A,B,C, D, E are external patch vertices, and hence each edge incident to those has a weight at 1, the sums assigned to the vertices A,B,C,D,E are the degrees of the corresponding vertices.

Recall that the Ammann-Beenker tiling has the following vertex configurations, see \cite{AmmannBeenker} and \cite{AB7Config}.

    \tikzset{square/.pic={
\draw[line width=2.5mm] (0,0)--(0,4)--(4,4)--(4,0)--(0,0);
\draw[loosely dashed, line width=2.5mm] (2,2)--(0,0);
\draw[loosely dashed,line width=2.5mm,->] (4,4)--(2,2);
\coordinate (-A) at (0,0);
\coordinate (-B) at (0,4);
\coordinate (-C) at (4,4);
\coordinate (-D) at (4,0);
}}

\tikzset{rhomb/.pic={
\draw[line width=2.5mm] (0,0)--(2.828427,2.828427)--(6.828427,2.828427)--(4,0)--(0,0);
\coordinate (-A) at (0,0);
\coordinate (-B) at (2.828427,2.828427);
\coordinate (-C) at (6.828427,2.828427);
\coordinate (-D) at (4,0);
}}

\begin{center}
    \setlength{\tabcolsep}{5pt}
    \begin{tabular}{cccc}
    \resizebox{50pt}{!}{
    \tikz{
    \pic (T1) [scale=1.3,rotate=-135] {square};
    \pic (T2) [scale=1.3,rotate=-90] at (T1-D) {rhomb};
    \pic (T3) [xscale=-1.3,yscale=1.3,rotate=-90] at (T1-B) {rhomb};
    \node[shape=circle,draw=red,fill=red,minimum size=7mm] at (T1-C) {};
    \foreach \i/\j in {T1/-B,T1/-D,T2/-C}
    {
    \node[shape=circle,draw=black,fill=black,minimum size=7mm] at (\i\j) {};
    }
    }
    }
         &  
    \resizebox{80pt}{!}{
    \tikz{
    \pic (T1) [scale=1.3,rotate=-22.5] {rhomb};
    \pic (T2) [scale=1.3,rotate=67.5] at (T1-B) {rhomb};
    \pic (T3) [scale=1.3,rotate=22.5] at (T1-A) {square};
    \pic (T4) [scale=1.3,rotate=67.5] at (T1-C) {square};
    \node[shape=circle,draw=red,fill=red,minimum size=7mm] at (T1-B) {};
    \foreach \i/\j in {T1/-A,T1/-C,T2/-B,T2/-D}
    {
    \node[shape=circle,draw=black,fill=black,minimum size=7mm] at (\i\j) {};
    }
    }
    }
    &
    \resizebox{110pt}{!}{
    \tikz{
    \pic (T1) [scale=1.3,rotate=90] {square};
    \pic (T2) [scale=1.3,rotate=0] at (T1-A) {square};
    \pic (T3) [scale=1.3,rotate=45] at (T1-D) {square};
    \pic (T4) [scale=1.3,rotate=0] at (T1-D) {rhomb};
    \pic (T5) [xscale=-1.3,yscale=1.3,rotate=0] at (T1-D) {rhomb};
    \node[shape=circle,draw=red,fill=red,minimum size=7mm] at (T1-D) {};
    \foreach \i/\j in {T3/-B,T3/-D,T1/-C,T2/-C,T1/-A}
    {
    \node[shape=circle,draw=black,fill=black,minimum size=7mm] at (\i\j) {};
    }
    }
    }
    &
    \resizebox{110pt}{!}{
    \tikz{
    \pic (T1) [scale=1.3,rotate=180] {square};
    \pic (T2) [scale=1.3,rotate=-90] at (T1-A) {square};
    \pic (T3) [scale=1.3,rotate=45] at (T1-A) {square};
    \pic (T4) [scale=1.3,rotate=0] at (T1-A) {rhomb};
    \pic (T5) [xscale=-1.3,yscale=1.3,rotate=0] at (T1-A) {rhomb};
    \node[shape=circle,draw=red,fill=red,minimum size=7mm] at (T1-A) {};
    \foreach \i/\j in {T1/-D,T1/-B,T2/-B,T3/-B,T3/-D}
    {
    \node[shape=circle,draw=black,fill=black,minimum size=7mm] at (\i\j) {};
    }
    }
    }\\
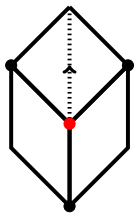
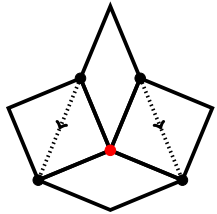
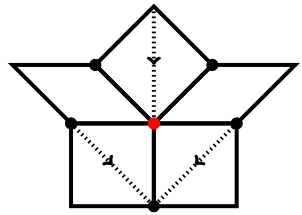
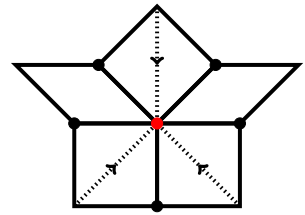
    \footnotesize{Vertex Configuration 1} & \footnotesize{Vertex Configuration 2} & \footnotesize{Vertex Configuration 3a} & \footnotesize{Vertex Configuration 3b}
    \end{tabular}
\end{center}

\begin{center}
    \setlength{\tabcolsep}{15pt}
    \begin{tabular}{ccc}
    \resizebox{80pt}{!}{
    \tikz{
    \pic (T1) [scale=1.3,rotate=67.5] {rhomb};
    \pic (T2) [scale=1.3,rotate=112.5] at (T1-A) {square};
    \pic (T3) [scale=1.3,rotate=-22.5] at (T1-A) {square};
    \pic (T4) [scale=1.3,rotate=-67.5] at (T1-A) {rhomb};
    \pic (T5) [scale=1.3,rotate=-112.5] at (T1-A) {rhomb};
    \pic (T6) [scale=1.3,rotate=-157.5] at (T1-A) {rhomb};
    \node[shape=circle,draw=red,fill=red,minimum size=7mm] at (T1-A) {};
    \foreach \i/\j in {T1/-B,T1/-D,T2/-B,T3/-D,T5/-B,T5/-D}
    {
    \node[shape=circle,draw=black,fill=black,minimum size=7mm] at (\i\j) {};
    }
    }
    }
         &  
    \resizebox{100pt}{!}{
    \tikz{
    \pic (T1) [scale=1.3,rotate=90] {rhomb};
    \pic (T2) [scale=1.3,rotate=45] at (T1-A) {rhomb};
    \pic (T3) [scale=1.3,rotate=0] at (T1-A) {rhomb};
    \pic (T4) [scale=1.3,rotate=135] at (T1-A) {rhomb};
    \pic (T5) [scale=1.3,rotate=-45] at (T1-A) {rhomb};
    \pic (T6) [xscale=-1.3,yscale=1.3,rotate=-45] at (T1-A) {rhomb};
    \pic (T7) [scale=1.3,rotate=-135] at (T1-A) {square};
    
    \node[shape=circle,draw=red,fill=red,minimum size=7mm] at (T1-A) {};
    \foreach \i/\j in {T1/-B,T1/-D,T2/-D,T3/-D,T4/-B,T7/-B,T7/-D}
    {
    \node[shape=circle,draw=black,fill=black,minimum size=7mm] at (\i\j) {};
    }
    }
    }
    &
    \resizebox{100pt}{!}{
    \tikz{
    \pic (T1) [scale=1.3,rotate=90] {rhomb};
    \pic (T2) [scale=1.3,rotate=45] at (T1-A) {rhomb};
    \pic (T3) [scale=1.3,rotate=0] at (T1-A) {rhomb};
    \pic (T4) [scale=1.3,rotate=135] at (T1-A) {rhomb};
    \pic (T5) [scale=1.3,rotate=-45] at (T1-A) {rhomb};
    \pic (T6) [xscale=-1.3,yscale=1.3,rotate=-45] at (T1-A) {rhomb};
    \pic (T7) [xscale=1.3,yscale=1.3,rotate=-90] at (T1-A) {rhomb};
    \pic (T8) [xscale=-1.3,yscale=1.3,rotate=-90] at (T1-A) {rhomb};
    
    \node[shape=circle,draw=red,fill=red,minimum size=7mm] at (T1-A) {};
    \foreach \i/\j in {T1/-B,T1/-D,T2/-D,T3/-D,T4/-B,T5/-D,T6/-D,T8/-D}
    {
    \node[shape=circle,draw=black,fill=black,minimum size=7mm] at (\i\j) {};
    }
    }
    }
    \\
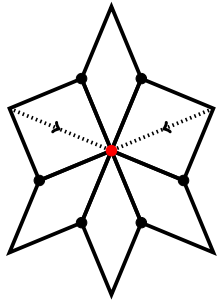
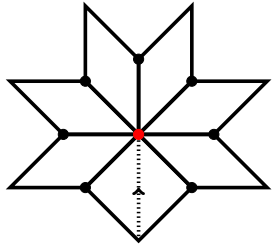
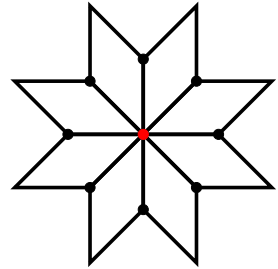
    \footnotesize{Vertex Configuration 4} & \footnotesize{Vertex Configuration 5} & \footnotesize{Vertex Configuration 6}
    \end{tabular}
\end{center}

By inspection we see that
\begin{itemize}
\item{} Since the A vertex has two arrows pointing in with a $90^\circ$ degree in between, any A vertex is the central vertex in Vertex configuration 3b, and hence has a degree of 5. Since A vertex is connected only to vertices of sum weight 7, there is no conflict here.

\item{} Vertex B is only connected with vertices of degree 3. Since there must be at least one more edge besides the three edges we see, there is no conflict. 

\item{} Vertex C is only connected with vertices of degree 3.The $45^circ$ angle shoes that this vertex cannot be in vertex configuration 1, and by inspection, all other have a degree of at least 4. Therefore, there is no conflict.

\item{} Vertex D is only connected with vertices of degree 4. Since there must be at least one more edge besides the four edges we see, there is no conflict. 

\item{} Note that the direction of the arrow $AB$ in the level one supertile is opposite to the direction of the arrow inside its central square (see for example \cite[Figure 6.1]{TAO} or \cite[Figure 4]{AB7Config}, we conclude that the level 1 supertiles at A must be in Vertex configuration 1.
 
Drawing these supertiles at A, we can see that at vertex E we can only get vertex configuration 4 or 5. Therefore, the degree at E must be 6 or 7. Since E is only connected with vertices of degree 3 or 5, there is no conflict.

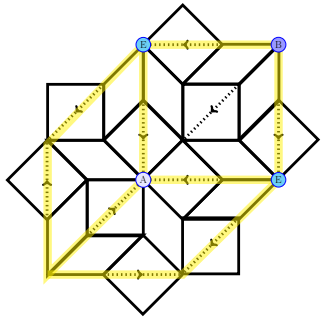
\tikzset{square/.pic={
\draw[line width=2mm] (0,0)--(0,4)--(4,4)--(4,0)--(0,0);
\draw[loosely dashed, line width=2mm] (2,2)--(0,0);
\draw[loosely dashed,line width=2mm,->] (4,4)--(2,2);
\coordinate (-A) at (0,0);
\coordinate (-B) at (0,4);
\coordinate (-C) at (4,4);
\coordinate (-D) at (4,0);
\coordinate (-a) at (0,2);
\coordinate (-b) at (2,4);
\coordinate (-c) at (4,2);
\coordinate (-d) at (2,0);
}}

\tikzset{square2/.pic={
\draw[line width=2mm] (0,0)--(0,4)--(4,4)--(4,0)--(0,0);
\draw[loosely dashed, line width=2mm] (2,2)--(4,4);
\draw[loosely dashed,line width=2mm,->] (0,0)--(2,2);
\coordinate (-A) at (0,0);
\coordinate (-B) at (0,4);
\coordinate (-C) at (4,4);
\coordinate (-D) at (4,0);
\coordinate (-a) at (0,2);
\coordinate (-b) at (2,4);
\coordinate (-c) at (4,2);
\coordinate (-d) at (2,0);
}}
\tikzset{rhomb/.pic={
\draw[line width=2mm] (0,0)--(2.828427,2.828427)--(6.828427,2.828427)--(4,0)--(0,0);
\coordinate (-A) at (0,0);
\coordinate (-B) at (2.828427,2.828427);
\coordinate (-C) at (6.828427,2.828427);
\coordinate (-D) at (4,0);
\coordinate (-a) at (1.4142135,1.4142135);
\coordinate (-b) at (4.828427,2.828427);
\coordinate (-c) at (5.4142135,1.4142135);
\coordinate (-d) at (2,0);
}}
\begin{center}
\resizebox{120pt}{!}{
\tikz{
\pic (B89) [scale=1,rotate=45] at (-12.46,-2.828427) {rhomb};
\pic (B90) [scale=1,rotate=45] at (-9.64,-5.6) {square};
\pic (B90) [scale=1,rotate=90] at (-6.82,-2.828427) {rhomb};
\pic (T92) [scale=1,rotate=0] at (-6.82,1.2) {rhomb};
\pic (T94) [scale=1,rotate=0] at (-6.82,-2.828427) {square};
\pic (T124) [scale=1,rotate=0] at (-16.46,-2.828427)  {square};
\pic (T97) [scale=1,rotate=45] at (0,-5.6) {square};
\pic (T114) [scale=1,rotate=135] at (-4,4){square2};
\pic (T119) [scale=1,rotate=0] at (-6.82,-8.46) {rhomb};
\pic (T121) [scale=1,rotate=45] at (-2.8,-2.85) {rhomb};
\pic (T122) [scale=1,rotate=135] at (0,-5.64) {rhomb};
\pic (T157) [scale=1,rotate=-135] at (-16.5,-2.828427) {square};
\pic (B91) [scale=1,rotate=-225] at (-9.64,-5.64) {rhomb};
\pic (T171) [scale=1,rotate=180] at (-9.64,-5.64){square};
\pic (T172) [scale=1,rotate=225] at (-13.64,-5.64) {rhomb};
\pic (T173) [scale=1,rotate=0] at (-16.46,-12.4) {rhomb};
\pic (T175) [scale=1,rotate=135] at (-6.84,-12.4) {square};
\pic (T176) [scale=1,rotate=0] at (-6.84,-12.35) {square};
\pic (T114) [scale=1,rotate=135] at (-4,-5.64){square2};

\draw[draw=yellow,opacity=0.5,line width=6mm] (-9.64,-5.64)--(-9.64,4)--(0,4)--(0,-5.64)--(-9.64,-5.64);
\draw[draw=yellow,opacity=0.5,line width=6mm] (-9.64,4) --(-16.46,-3.1) --(-16.46,-12.4)--(-9.64,-5.64);
\draw[draw=yellow,opacity=0.5,line width=6mm] (-16.46,-12.4)--(-6.84,-12.4)--(0,-5.64) ;

\node[draw=blue,shape=circle,fill=blue!10,text=black,font={\huge}] at (-9.64,-5.64) {A};
\node[draw=blue,shape=circle,fill=cyan!50,text=black,font={\huge}] at (-9.64,4) {E};
\node[draw=blue,shape=circle,fill=cyan!50,text=black,font={\huge}] at (0,-5.64) {E};
\node[draw=blue,shape=circle,fill=blue!40,text=black,font={\huge}] at (0,4) {B};
}
}
\end{center}
\end{itemize}

Therefore no conflicts exist.
\end{proof}

\tikzset{square/.pic={
\draw[line width=2mm] (0,0)--(0,4)--(4,4)--(4,0)--(0,0);
\draw[loosely dashed, line width=2mm] (2,2)--(0,0);
\draw[loosely dashed,line width=2mm,->] (4,4)--(2,2);
\foreach \i/\j in {2/4,4/2}
{
\node[shape=circle,fill=white,text=black,font={\Huge}] at (\i,\j) {$1$};
}
\coordinate (-A) at (0,0);
\coordinate (-B) at (0,4);
\coordinate (-C) at (4,4);
\coordinate (-D) at (4,0);
\coordinate (-a) at (0,2);
\coordinate (-b) at (2,4);
\coordinate (-c) at (4,2);
\coordinate (-d) at (2,0);
}}

\tikzset{square2/.pic={
\draw[line width=2mm] (0,0)--(0,4)--(4,4)--(4,0)--(0,0);
\draw[loosely dashed, line width=2mm] (2,2)--(4,4);
\draw[loosely dashed,line width=2mm,->] (0,0)--(2,2);
\foreach \i/\j in {2/4,4/2}
{
\node[shape=circle,fill=white,text=black,font={\Huge}] at (\i,\j) {$1$};
}
\coordinate (-A) at (0,0);
\coordinate (-B) at (0,4);
\coordinate (-C) at (4,4);
\coordinate (-D) at (4,0);
\coordinate (-a) at (0,2);
\coordinate (-b) at (2,4);
\coordinate (-c) at (4,2);
\coordinate (-d) at (2,0);
}}

\tikzset{rhomb/.pic={
\draw[line width=2mm] (0,0)--(2.828427,2.828427)--(6.828427,2.828427)--(4,0)--(0,0);
\foreach \i/\j in {1.4142135/1.4142135,4.828427/2.828427,5.4142135/1.4142135,2/0}
{
\node[shape=circle,fill=white,text=black,font={\Huge}] at (\i,\j) {$1$};
}
\coordinate (-A) at (0,0);
\coordinate (-B) at (2.828427,2.828427);
\coordinate (-C) at (6.828427,2.828427);
\coordinate (-D) at (4,0);
\coordinate (-a) at (1.4142135,1.4142135);
\coordinate (-b) at (4.828427,2.828427);
\coordinate (-c) at (5.4142135,1.4142135);
\coordinate (-d) at (2,0);
}}
\begin{center}
\resizebox{420pt}{!}{
\tikz{
\clip (-35,-20) rectangle (13,13);

\pic (T1) [scale=1,rotate=0] {square};
\pic (T2) [scale=1,rotate=90] at (T1-A) {rhomb};
\pic (T3) [scale=1,rotate=-45] at (T2-C) {square};
\pic (T4) [scale=1,rotate=0] at (T1-B) {rhomb};
\pic (T5) [scale=1,rotate=45] at (T1-D) {rhomb};
\pic (T6) [scale=1,rotate=45] at (T3-A) {rhomb};
\pic (T7) [scale=1,rotate=-90] at (T6-C) {square};
\pic (T8) [scale=1,rotate=-45] at (T3-B) {rhomb};
\pic (T9) [scale=1,rotate=-90] at (T7-B) {rhomb};
\pic (T10) [scale=1,rotate=0] at (T4-C) {square};
\pic (T11) [scale=1,rotate=-90] at (T5-C) {rhomb};
\pic (T12) [scale=1,rotate=-45] at (T1-A) {rhomb};
\pic (T13) [scale=1,rotate=45] at (T12-C) {square};
\pic (T14) [scale=1,rotate=-225] at (T7-A) {square};
\pic (T15) [scale=1,rotate=-270] at (T3-A) {square};
\pic (T16) [scale=1,rotate=-135] at (T3-A) {rhomb};
\pic (T17) [scale=1,rotate=-180] at (T3-A) {rhomb};
\pic (T18) [scale=1,rotate=45] at (T17-C) {square};
\pic (T19) [xscale=-1,yscale=1,rotate=0] at (T18-A) {rhomb};
\pic (T20) [scale=1,rotate=0] at (T19-C) {square};
\pic (T21) [xscale=-1,yscale=1,rotate=0] at (T6-B) {rhomb};
\pic (T22) [scale=1,rotate=-45] at (T7-B) {square};
\pic (T23) [scale=1,rotate=0] at (T7-A) {square};
\pic (T24) [scale=1,rotate=45] at (T22-A) {rhomb};
\pic (T25) [scale=1,rotate=-90] at (T24-C) {square};
\pic (T26) [scale=1,rotate=-45] at (T22-B) {rhomb};
\pic (T27) [scale=1,rotate=0] at (T22-D) {rhomb};
\pic (T28) [scale=1,rotate=-45] at (T10-A) {rhomb};
\pic (T29) [scale=1,rotate=45] at (T28-C) {square};
\pic (T30) [scale=1,rotate=45] at (T29-B) {rhomb};
\pic (T31) [scale=1,rotate=-90] at (T30-C) {rhomb};
\pic (T32) [scale=1,rotate=0] at (T29-A) {rhomb};
\pic (T33) [scale=1,rotate=90] at (T32-C) {square};
\pic (T34) [scale=1,rotate=0] at (T13-A) {rhomb};
\pic (T35) [scale=1,rotate=90] at (T34-C) {square};
\pic (T36) [scale=1,rotate=0] at (T35-A) {square};
\pic (T37) [scale=1,rotate=-45] at (T36-A) {rhomb};
\pic (T38) [scale=1,rotate=180] at (T37-C) {square};
\pic (T39) [scale=1,rotate=-90] at (T13-A) {square};
\pic (T40) [scale=1,rotate=-135] at (T35-A) {square};
\pic (T41) [scale=1,rotate=-90] at (T40-D) {rhomb};
\pic (T42) [scale=1,rotate=135] at (T20-A) {square};
\pic (T43) [scale=1,rotate=-90] at (T19-C) {rhomb};
\pic (T44) [xscale=-1,yscale=1,rotate=-90] at (T19-C) {rhomb};
\pic (T45) [scale=1,rotate=90] at (T44-C) {square};
\pic (T46) [scale=1,rotate=-45] at (T39-D) {rhomb};
\pic (T47) [scale=1,rotate=-135] at (T39-A) {rhomb};
\pic (T48) [scale=1,rotate=-45] at (T47-C) {square};
\pic (T49) [scale=1,rotate=-90] at (T48-A) {rhomb};
\pic (T50) [scale=1,rotate=0] at (T48-D) {rhomb};
\pic (T51) [scale=1,rotate=0] at (T49-C) {square};
\pic (T52) [scale=1,rotate=180] at (T13-A) {rhomb};
\pic (T53) [scale=1,rotate=180] at (T45-A) {rhomb};
\pic (T54) [scale=1,rotate=-45] at (T53-C) {rhomb};
\pic (T55) [scale=1,rotate=-225] at (T49-C) {rhomb};
\pic (T56) [scale=1,rotate=-135] at (T51-A) {square};
\pic (T57) [scale=1,rotate=180] at (T56-A) {rhomb};
\pic (T58) [scale=1,rotate=-90] at (T57-C) {square};
\pic (T59) [scale=1,rotate=-90] at (T54-A) {rhomb};
\pic (T60) [scale=1,rotate=-135] at (T58-A) {rhomb};
\pic (T61) [scale=1,rotate=90] at (T60-C) {rhomb};
\pic (T62) [scale=1,rotate=-135] at (T59-A) {rhomb};
\pic (T63) [scale=1,rotate=45] at (T54-A) {square};
\pic (T64) [scale=1,rotate=-225] at (T63-A) {rhomb};
\pic (T65) [scale=1,rotate=0] at (T64-C) {square};
\pic (T66) [scale=1,rotate=90] at (T65-A) {rhomb};
\pic (T67) [scale=1,rotate=135] at (T61-A) {rhomb};
\pic (T68) [scale=1,rotate=-180] at (T17-B) {square2};
\pic (T69) [scale=1,rotate=90] at (T68-C) {square};
\pic (T70) [scale=1,rotate=-45] at (T68-C) {rhomb};
\pic (T71) [scale=1,rotate=45] at (T70-C) {square};
\pic (T72) [scale=1,rotate=-45] at (T71-A) {square};
\pic (T73) [scale=1,rotate=-90] at (T72-A) {rhomb};
\pic (T74) [scale=1,rotate=0] at (T73-C) {square};
\pic (T75) [scale=1,rotate=-90] at (T74-A) {square};
\pic (T76) [scale=1,rotate=-180] at (T71-A) {square};
\pic (T77) [scale=1,rotate=-225] at (T74-A) {square};
\pic (T78) [scale=1,rotate=-180] at (T77-D) {rhomb};
\pic (T79) [scale=1,rotate=-135] at (T68-C) {square};
\pic (T80) [scale=1,rotate=225] at (T76-D) {rhomb};
\pic (T81) [scale=1,rotate=-90] at (T21-C) {rhomb};
\pic (T82) [scale=1,rotate=45] at (T18-B) {rhomb};
\pic (T83) [scale=1,rotate=180] at (T68-C) {rhomb};
\pic (T84) [scale=1,rotate=45] at (T83-C) {square};
\pic (T85) [scale=1,rotate=-90] at (T84-A) {square};
\pic (T86) [scale=1,rotate=-90] at (T83-B) {rhomb};
\pic (T87) [scale=1,rotate=135] at (T84-A) {square};
\pic (T88) [scale=1,rotate=-135] at (T85-A) {rhomb};
\pic (T89) [scale=1,rotate=-45] at (T88-B) {rhomb};
\pic (T90) [scale=1,rotate=-45] at (T88-C) {square};
\pic (T91) [scale=1,rotate=90] at (T90-A) {square};
\pic (T92) [scale=1,rotate=0] at (T90-D) {rhomb};
\pic (T93) [scale=1,rotate=-90] at (T90-A) {rhomb};
\pic (T94) [scale=1,rotate=0] at (T93-C) {square};
\pic (T95) [scale=1,rotate=45] at (T94-D) {rhomb};
\pic (T96) [scale=1,rotate=-45] at (T94-A) {rhomb};
\pic (T97) [scale=1,rotate=45] at (T96-C) {square};
\pic (T98) [scale=1,rotate=-45] at (T80-C) {rhomb};
\pic (T99) [scale=1,rotate=-135] at (T77-A) {rhomb};
\pic (T100) [scale=1,rotate=-45] at (T99-C) {square};
\pic (T101) [scale=1,rotate=90] at (T100-A) {square};
\pic (T102) [scale=1,rotate=0] at (T97-A) {rhomb};
\pic (T103) [scale=1,rotate=90] at (T101-B) {rhomb};
\pic (T104) [scale=1,rotate=-135] at (T100-A) {square};
\pic (T105) [scale=1,rotate=-90] at (T97-A) {square};
\pic (T106) [scale=1,rotate=180] at (T67-A) {rhomb};
\pic (T107) [scale=1,rotate=180] at (T90-A) {rhomb};
\pic (T108) [scale=1,rotate=45] at (T107-C) {square};
\pic (T109) [scale=1,rotate=135] at (T108-A) {rhomb};
\pic (T110) [scale=1,rotate=-90] at (T109-C) {rhomb};
\pic (T111) [scale=1,rotate=-135] at (T105-A) {rhomb};
\pic (T112) [scale=1,rotate=-45] at (T111-C) {square};
\pic (T113) [scale=1,rotate=-90] at (T112-A) {rhomb};
\pic (T114) [scale=1,rotate=135] at (T96-D) {square2};
\pic (T115) [scale=1,rotate=45] at (T114-C) {square};
\pic (T116) [scale=1,rotate=-45] at (T60-C) {square};
\pic (T117) [scale=1,rotate=-45] at (T60-B) {rhomb};
\pic (T118) [scale=1,rotate=-90] at (T58-B) {rhomb};
\pic (T119) [scale=1,rotate=0] at (T114-B) {rhomb};
\pic (T120) [scale=1,rotate=0] at (T109-C) {square};
\pic (T121) [scale=1,rotate=45] at (T115-B) {rhomb};
\pic (T122) [scale=1,rotate=135] at (T115-A) {rhomb};
\pic (T123) [scale=1,rotate=90] at (T122-C) {square};
\pic (T124) [scale=1,rotate=0] at (T122-C) {square};
\pic (T125) [scale=1,rotate=0] at (T65-B) {rhomb};
\pic (T126) [scale=1,rotate=45] at (T65-D) {rhomb};
\pic (T127) [scale=1,rotate=-90] at (T126-C) {rhomb};
\pic (T128) [scale=1,rotate=-45] at (T127-A) {rhomb};
\pic (T129) [scale=1,rotate=0] at (T128-A) {rhomb};
\pic (T130) [scale=1,rotate=0] at (T20-B) {rhomb};
\pic (T131) [scale=1,rotate=90] at (T20-A) {rhomb};
\pic (T132) [scale=1,rotate=180] at (T131-C) {square};
\pic (T133) [scale=1,rotate=-45] at (T131-C) {square};
\pic (T134) [scale=1,rotate=45] at (T129-A) {rhomb};
\pic (T135) [scale=1,rotate=90] at (T129-A) {rhomb};
\pic (T136) [scale=1,rotate=135] at (T132-A) {rhomb};
\pic (T137) [scale=1,rotate=-135] at (T136-C) {square};
\pic (T138) [scale=1,rotate=180] at (T137-A) {rhomb};
\pic (T139) [scale=1,rotate=-90] at (T138-C) {square};
\pic (T140) [scale=1,rotate=45] at (T66-C) {rhomb};
\pic (T141) [scale=1,rotate=-45] at (T140-A) {square};
\pic (T142) [scale=1,rotate=-45] at (T139-D) {rhomb};
\pic (T143) [scale=1,rotate=90] at (T120-A) {rhomb};
\pic (T144) [scale=1,rotate=0] at (T120-B) {rhomb};
\pic (T145) [scale=1,rotate=-45] at (T143-C) {square};
\pic (T146) [scale=1,rotate=45] at (T145-A) {square};
\pic (T147) [scale=1,rotate=90] at (T140-A) {square};
\pic (T148) [scale=1,rotate=135] at (T147-D) {rhomb};
\pic (T149) [scale=1,rotate=135] at (T139-A) {square};
\pic (T150) [scale=1,rotate=180] at (T145-A) {square};
\pic (T151) [scale=1,rotate=135] at (T120-A) {square};
\pic (T152) [scale=1,rotate=180] at (T151-D) {rhomb};
\pic (T153) [scale=1,rotate=180] at (T123-A) {rhomb};
\pic (T154) [scale=1,rotate=45] at (T153-C) {square};
\pic (T155) [scale=1,rotate=45] at (T154-B) {rhomb};
\pic (T156) [scale=1,rotate=-90] at (T154-A) {square};
\pic (T157) [scale=1,rotate=-135] at (T153-A) {square};
\pic (T158) [scale=1,rotate=-90] at (T156-B) {rhomb};
\pic (T159) [scale=1,rotate=135] at (T156-A) {square};
\pic (T160) [scale=1,rotate=-135] at (T156-A) {rhomb};
\pic (T161) [scale=1,rotate=90] at (T160-C) {square};
\pic (T162) [scale=1,rotate=135] at (T159-B) {rhomb};
\pic (T163) [scale=1,rotate=-45] at (T161-A) {square};
\pic (T164) [scale=1,rotate=-45] at (T163-B) {rhomb};
\pic (T165) [scale=1,rotate=0] at (T163-D) {rhomb};
\pic (T166) [scale=1,rotate=-90] at (T163-A) {rhomb};
\pic (T167) [scale=1,rotate=180] at (T163-A) {square};
\pic (T168) [scale=1,rotate=0] at (T166-C) {square};
\pic (T169) [scale=1,rotate=135] at (T168-A) {square};
\pic (T170) [scale=1,rotate=180] at (T169-D) {rhomb};
\pic (T171) [scale=1,rotate=180] at (T115-A) {square};
\pic (T172) [scale=1,rotate=45] at (T158-C) {rhomb};
\pic (T173) [scale=1,rotate=0] at (T172-A) {rhomb};
\pic (T174) [scale=1,rotate=-90] at (T114-C) {rhomb};
\pic (T175) [scale=1,rotate=135] at (T174-C) {square};
\pic (T176) [scale=1,rotate=0] at (T175-A) {square};
\pic (T177) [scale=1,rotate=-90] at (T175-A) {square};
\pic (T178) [scale=1,rotate=-45] at (T164-C) {rhomb};
\pic (T179) [scale=1,rotate=-90] at (T164-C) {rhomb};
\pic (T180) [scale=1,rotate=-135] at (T164-C) {rhomb};
\pic (T181) [scale=1,rotate=-135] at (T176-A) {rhomb};
\pic (T182) [scale=1,rotate=90] at (T181-C) {square};
\pic (T183) [scale=1,rotate=-45] at (T182-A) {square};
\pic (T184) [scale=1,rotate=-45] at (T183-B) {rhomb};
\pic (T185) [scale=1,rotate=0] at (T184-C) {square};
\pic (T186) [scale=1,rotate=45] at (T185-D) {rhomb};
\pic (T187) [scale=1,rotate=-90] at (T186-C) {rhomb};
\pic (T188) [scale=1,rotate=0] at (T116-D) {rhomb};
\pic (T189) [scale=1,rotate=0] at (T187-C) {square};
\pic (T190) [scale=1,rotate=-45] at (T185-A) {rhomb};
\pic (T191) [scale=1,rotate=45] at (T190-C) {square};
\pic (T192) [scale=1,rotate=-45] at (T191-A) {square};
\pic (T193) [scale=1,rotate=-45] at (T189-A) {rhomb};
\pic (T194) [scale=1,rotate=45] at (T193-C) {square};
\pic (T195) [scale=1,rotate=45] at (T194-B) {rhomb};
\pic (T196) [scale=1,rotate=0] at (T194-A) {rhomb};
\pic (T197) [scale=1,rotate=90] at (T196-C) {square};
\pic (T198) [scale=1,rotate=90] at (T197-B) {rhomb};
\pic (T199) [scale=1,rotate=-45] at (T198-C) {rhomb};
\pic (T200) [scale=1,rotate=0] at (T199-A) {rhomb};
\pic (T201) [scale=1,rotate=45] at (T200-A) {rhomb};
\pic (T202) [scale=1,rotate=45] at (T197-A) {rhomb};
\pic (T203) [scale=1,rotate=135] at (T202-C) {square};
\pic (T204) [scale=1,rotate=90] at (T203-A) {rhomb};
\pic (T205) [scale=1,rotate=-45] at (T204-C) {rhomb};
\pic (T206) [scale=1,rotate=180] at (T204-C) {square};
\pic (T207) [scale=1,rotate=-45] at (T56-A) {rhomb};
\pic (T208) [scale=1,rotate=45] at (T206-A) {square};
\pic (T209) [scale=1,rotate=45] at (T51-D) {rhomb};
\pic (T210) [scale=1,rotate=0] at (T202-A) {rhomb};
\pic (T211) [scale=1,rotate=-90] at (T38-A) {rhomb};
\pic (T212) [scale=1,rotate=135] at (T211-C) {square};
\pic (T213) [scale=1,rotate=45] at (T41-C) {rhomb};
\pic (T214) [scale=1,rotate=0] at (T41-C) {rhomb};
\pic (T215) [scale=1,rotate=-45] at (T41-C) {rhomb};
\pic (T216) [scale=1,rotate=-90] at (T41-C) {rhomb};
\pic (T217) [scale=1,rotate=0] at (T208-A) {rhomb};
\pic (T218) [scale=1,rotate=90] at (T217-C) {square};
\pic (T219) [scale=1,rotate=-90] at (T205-D) {square2};
\pic (T220) [scale=1,rotate=180] at (T219-C) {square};

\draw[draw=yellow,opacity=0.5,line width=6mm] (T21-C)--(T14-A)--(T8-C)--(T15-A)--(T21-C);
\draw[draw=yellow,opacity=0.5,line width=6mm] (T23-A)--(T24-C)--(T26-C)--(T10-A);
\draw[draw=yellow,opacity=0.5,line width=6mm] (T31-A)--(T33-A)--(T36-A)--(T10-A);
\draw[draw=yellow,opacity=0.5,line width=6mm] (T34-C)--(T41-C)--(T34-A)--(T11-A);
\draw[draw=yellow,opacity=0.5,line width=6mm] (T13-A)--(T72-A)--(T3-A);
\draw[draw=yellow,opacity=0.5,line width=6mm] (T46-C)--(T51-A)--(T75-A)--(T13-A);
\draw[draw=yellow,opacity=0.5,line width=6mm] (T77-A)--(T78-C)--(T2-A);
\draw[draw=yellow,opacity=0.5,line width=6mm] (T80-C)--(T70-A)--(T16-A);
\draw[draw=yellow,opacity=0.5,line width=6mm] (T81-A)--(T20-A)--(T69-A);
\draw[draw=yellow,opacity=0.5,line width=6mm] (T43-A)--(T84-A)--(T86-C);
\draw[draw=yellow,opacity=0.5,line width=6mm] (T87-A)--(T54-A)--(T88-C)--(T89-C);
\draw[draw=yellow,opacity=0.5,line width=6mm] (T98-A)--(T99-C)--(T55-A);
\draw[draw=yellow,opacity=0.5,line width=6mm] (T103-C)--(T97-A)--(T61-A)--(T104-A);
\draw[draw=yellow,opacity=0.5,line width=6mm] (T90-A)--(T115-A)--(T96-C);
\draw[draw=yellow,opacity=0.5,line width=6mm] (T56-A)--(T117-C)--(T116-A);
\draw[draw=yellow,opacity=0.5,line width=6mm] (T62-A)--(T120-A)--(T123-A)--(T91-A);
\draw[draw=yellow,opacity=0.5,line width=6mm] (T42-A)--(T128-A)--(T63-A);
\draw[draw=yellow,opacity=0.5,line width=6mm] (T129-A)--(T131-C)--(T130-C);
\draw[draw=yellow,opacity=0.5,line width=6mm] (T125-C)--(T141-A)--(T146-A)--(T64-A);
\draw[draw=yellow,opacity=0.5,line width=6mm] (T147-A)--(T148-C)--(T149-A)--(T142-C);
\draw[draw=yellow,opacity=0.5,line width=6mm] (T150-A)--(T152-C)--(T151-A);
\draw[draw=yellow,opacity=0.5,line width=6mm] (T110-A)--(T154-A)--(T158-C)--(T157-A);
\draw[draw=yellow,opacity=0.5,line width=6mm] (T159-A)--(T162-C)--(T161-A)--(T164-C);
\draw[draw=yellow,opacity=0.5,line width=6mm] (T167-A)--(T170-C)--(T169-A)--(T165-C);
\draw[draw=yellow,opacity=0.5,line width=6mm] (T171-A)--(T172-A);
\draw[draw=yellow,opacity=0.5,line width=6mm] (T173-A)--(T175-A)--(T119-C);
\draw[draw=yellow,opacity=0.5,line width=6mm] (T180-A)--(T180-B);
\draw[draw=yellow,opacity=0.5,line width=6mm] (T178-A)--(T183-A)--(T184-C)--(T177-A);
\draw[draw=yellow,opacity=0.5,line width=6mm] (T67-A)--(T185-A);
\draw[draw=yellow,opacity=0.5,line width=6mm] (T186-C)--(T192-A)--(T193-C)--(T195-C);
\draw[draw=yellow,opacity=0.5,line width=6mm] (T209-C)--(T208-A)--(T201-A);
\draw[draw=yellow,opacity=0.5,line width=6mm] (T204-C)--(T219-C)--(T202-A)--(T199-A);

\foreach \i/\j in {T6/-b,T6/-c,T6/-a,T25/-d,T9/-b,T36/-a,T33/-a,T23/-d,T34/-c,T34/-b,T34/-d,T3/-d,T71/-d,T13/-a,T51/-a,T74/-d,T73/-c,T73/-b,T73/-d,T16/-c,T70/-b,T70/-d,T20/-b,T68/-d,T69/-d,T19/-b,T84/-d,T85/-a,T88/-c,T88/-b,T99/-c,T99/-b,T102/-d,T102/-c,T97/-a,T114/-c,T96/-d,T100/-d,T60/-b,T120/-d,T123/-d,T63/-d,T87/-d,T131/-b,T131/-c,T146/-d,T141/-d,T140/-a,T140/-b,T140/-d,T143/-c,T143/-b,T143/-a,T153/-c,T153/-b,T153/-a,T159/-d,T159/-a,T160/-c,T160/-b,T167/-d,T167/-a,T166/-c,T166/-b,T122/-c,T122/-b,T122/-a,T174/-c,T174/-b,T181/-c,T181/-b,T176/-d,T185/-a,T191/-d,T194/-a,T207/-b,T206/-a,T219/-d,T220/-d}
{
\node[shape=circle,fill=white,text=black,font={\Huge}] at (\i\j) {$1$};
}
\foreach \i/\j in {T22/-a,T22/-d,T29/-a,T29/-d,T1/-a,T1/-d,T48/-a,T48/-d,T18/-a,T18/-d,T94/-a,T94/-d,T58/-d,T58/-a,T108/-a,T108/-d,T45/-a,T45/-d,T65/-a,T65/-d,T112/-a,T112/-d,T189/-a,T189/-d,T203/-a,T203/-d}
{
\node[shape=circle,fill=white,text=black,font={\Huge}] at (\i\j) {$2$};
}

\foreach \i/\j in {T26/-D,T29/-C,T4/-D,T48/-C,T18/-C,T94/-C,T58/-C,T108/-C,T45/-C,T65/-C,T112/-C,T189/-C,T203/-C}
{
\node[draw=black,shape=circle,fill=white,text=black,font={\Huge}] at (\i\j) {$3$};
}
\foreach \i/\j in {T6/-B,T6/-D,T34/-B,T34/-D,T78/-A,T73/-B,T70/-B,T70/-D,T83/-D,T83/-B,T88/-D,T88/-B,T99/-D,T99/-B,T102/-B,T102/-D,T131/-B,T131/-D,T140/-B,T140/-D,T143/-B,T143/-D,T153/-D,T153/-B,T160/-D,T160/-B,T166/-D,T166/-B,T122/-D,T122/-B,T174/-B,T174/-D,T181/-B,T181/-D,T207/-D,T207/-B}
{
\node[draw=black,shape=circle,fill=white,text=black,font={\Huge}] at (\i\j) {$4$};
}
\foreach \i/\j in {T25/-D,T22/-D,T10/-D,T33/-B,T1/-B,T1/-D,T48/-B,T48/-D,T18/-B,T18/-D,T94/-B,T94/-D,T58/-B,T58/-D,T108/-B,T108/-D,T45/-B,T45/-D,T65/-B,T65/-D,T112/-B,T112/-D,T189/-B,T189/-D,T203/-D,T203/-B}
{
\node[draw=black,shape=circle,fill=white,text=black,font={\Huge}] at (\i\j) {$5$};
}
\foreach \i/\j in {T22/-A,T29/-A,T1/-A,T48/-A,T18/-A,T94/-A,T58/-A,T108/-A,T45/-A,T65/-A,T112/-A,T189/-A,T203/-A}
{
\node[draw=black,shape=circle,fill=white,text=black,font={\Huge}] at (\i\j) {$7$};
}

\foreach \i/\j in {T21/-B,T7/-C,T3/-C,T15/-C,T23/-C,T25/-C,T10/-C,T11/-D,T11/-B,T39/-C,T40/-C,T2/-B,T72/-C,T52/-B,T75/-C,T51/-C,T76/-C,T77/-C,T68/-A,T79/-C,T20/-C,T69/-C,T84/-C,T85/-C,T87/-C,T91/-C,T90/-C,
T101/-C,T100/-C,T97/-C,T105/-C,T104/-C,T115/-C,T114/-A,T56/-C,T116/-C,T120/-C,T123/-C,T124/-C,T42/-C,T63/-C,T133/-C,T132/-C,T141/-C,T146/-C,T145/-C,T139/-C,T149/-C,T147/-C,T150/-C,T151/-C,T154/-C,T156/-C,T157/-C,T159/-C,T161/-C,T163/-C,T167/-C,T169/-C,T168/-C,T171/-C,T175/-C,T176/-C,T182/-C,T183/-C,T177/-C,T185/-C,T191/-C,T192/-C,T194/-C,T206/-C,T208/-C,T219/-A,T220/-C,T197/-C}
{
\node[draw=red,shape=circle,fill=red!50,text=black,font={\Huge}] at (\i\j) {$3$};
}
\foreach \i/\j in {T14/-A,T36/-A,T73/-A,T73/-C,T68/-C,T88/-A,T101/-A,T123/-A,T146/-A,T156/-A,T161/-A,T166/-C,T122/-A,T175/-A,T183/-A,T191/-A,T219/-C}
{
\node[draw=blue,shape=circle,fill=blue!10,text=black,font={\Huge}] at (\i\j) {$5$};
}
\foreach \i/\j in {T6/-A,T13/-A,T20/-A,T55/-A,T90/-A,T97/-A,T141/-A,T139/-A,T143/-A,T184/-C,T194/-A,T207/-C,T197/-A}
{
\node[draw=blue,shape=circle,fill=cyan!50,text=black,font={\Huge}] at (\i\j) {$6$};
}
\foreach \i/\j in {T9/-C,T54/-A,T106/-A}
{
\node[draw=blue,shape=circle,fill=blue!40,text=black,font={\Huge}] at (\i\j) {$7$};
}
\foreach \i/\j in {T30/-C,T98/-A,T125/-C,T172/-A,T188/-C,T209/-C}
{
\node[draw=blue,shape=circle,fill=cyan!30!gray,text=black,font={\Huge}] at (\i\j) {$8$};
}
}
}
\end{center}

\subsection{Penrose Rhomb}

\begin{proposition}
    Consider the non-periodic tiling of the plane with the \href{https://tilings.math.uni-bielefeld.de/substitution/penrose-rhomb/}{Penrose rhomb} in \cite{TilingEncyclopedia}. Each rhombus prototile has two sides identified with double arrows and two sides with single arrows. Assign the edges with double arrows a weight of 1 and the edges with single arrows a weight of 3. This provides a solution to the 1-2-3 problem.
\end{proposition}

\begin{proof}
    It is known from \cite{Penrose} that there are only eight possible ways for the tiles to meet at a vertex. By assigning a weight of 1 to all double arrows and a weight of 3 to all single arrows we can see from the image above that each vertex configuration has a different sum.

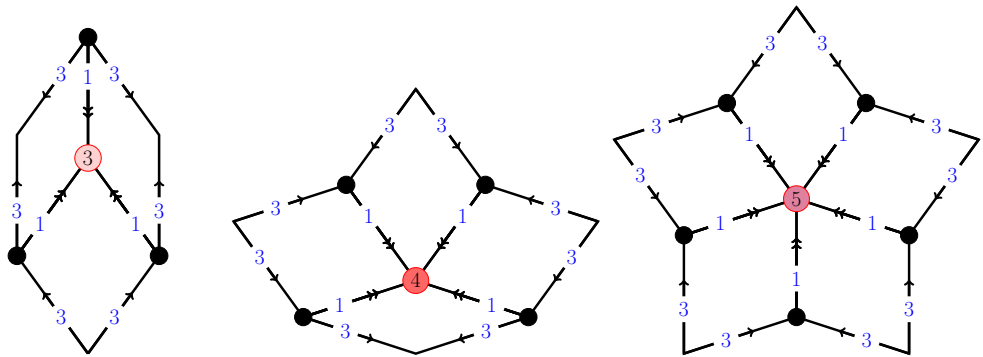
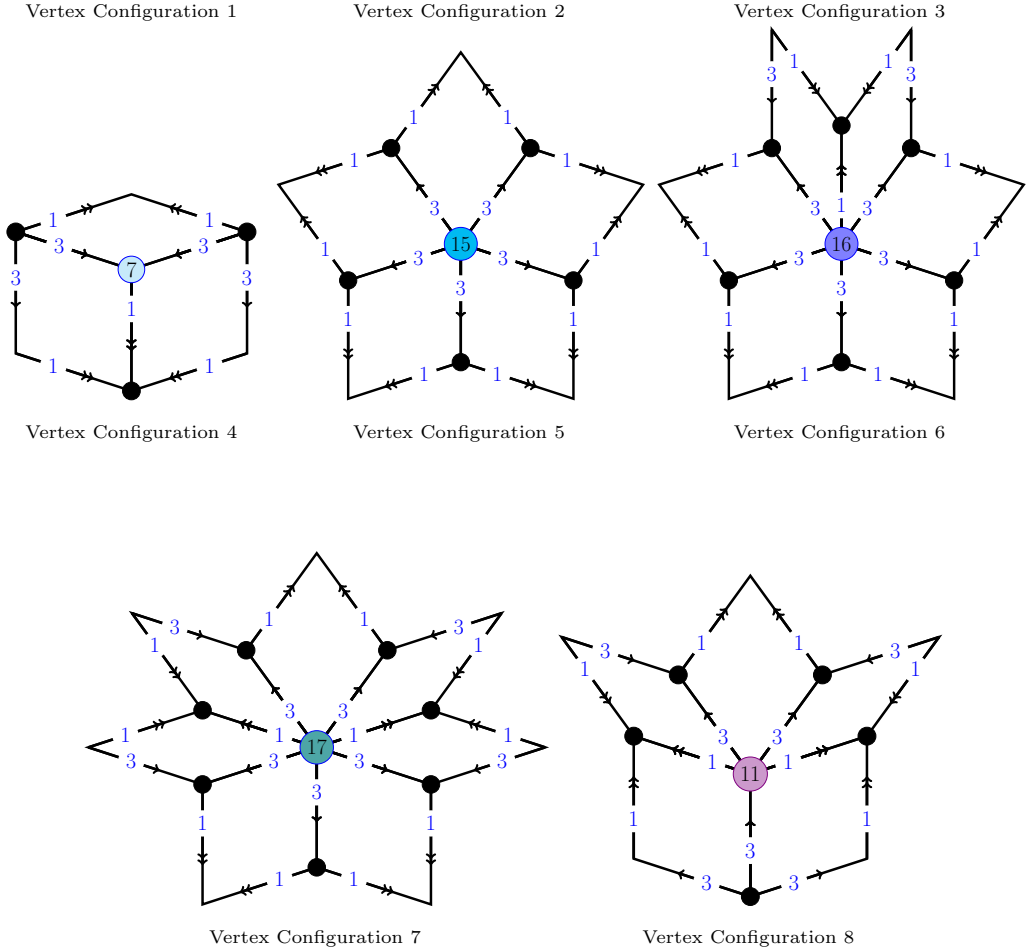
\tikzset{proto1a/.pic={
    \draw[line width=1mm] (0,0)--(-2.351141009,3.2360679775)--(0,6.472133955)--(2.351141009,3.2360679775)--(0,0);
    \draw[line width=3,->] (0,0)--(-1.4694631307,2.0225424859);
    \draw[line width=3,->] (0,0)--(1.4694631307,2.0225424859);
    \draw[line width=3,->>] (-2.351141009,3.2360679775)--(-0.76412082781,5.4204138623);
    \draw[line width=3,->>] (2.351141009,3.2360679775)--(0.76412082781,5.4204138623);
    
    \foreach \i/\j in {-0.881677878/1.21352549,0.881677878/1.21352549}
    {
    \draw (\i,\j) node[shape=circle,fill=white,font={\Huge}]{\color{blue}$3$};
    }
    \foreach \i/\j in {-1.58702018/4.28779007,1.58702018/4.28779007}
    {
    \draw (\i,\j) node[shape=circle,fill=white,font={\Huge}]{\color{blue}$1$};
    }

    \coordinate (-A) at (0,0);
        \coordinate (-B) at (-2.351141009,3.2360679775);
        \coordinate (-C) at (0,6.472133955);
        \coordinate (-D) at (2.351141009,3.2360679775);
    
}
}

\tikzset{proto1b/.pic={
    \draw[line width=1mm] (0,0)--(-2.351141009,3.2360679775)--(0,6.472133955)--(2.351141009,3.2360679775)--(0,0);
    \draw[line width=3,->] (0,6.472135955)--(-1.4694631307,4.4495934691);
    \draw[line width=3,->] (0,6.472135955)--(1.4694631307,4.4495934691);
    \draw[line width=3,->>] (-2.351141009,3.2360679775)--(-0.76412082781,1.0517220927);
    \draw[line width=3,->>] (2.351141009,3.2360679775)--(0.76412082781,1.0517220927);
    
    \foreach \i/\j in {-0.881677878/5.258610465,0.881677878/5.258610465}
    {
    \draw (\i,\j) node[shape=circle,fill=white,font={\Huge}]{\color{blue}$3$};
    }
    \foreach \i/\j in {-1.58702018/2.184345885,1.58702018/2.184345885}
    {
    \draw (\i,\j) node[shape=circle,fill=white,font={\Huge}]{\color{blue}$1$};
    }

    \coordinate (-A) at (0,0);
        \coordinate (-B) at (-2.351141009,3.2360679775);
        \coordinate (-C) at (0,6.472133955);
        \coordinate (-D) at (2.351141009,3.2360679775);
    
}
}

\tikzset{proto2a/.pic={
    \draw[line width=1mm] (0,0)--(4,0)--(7.2360679775,-2.351141009)--(3.2360679775,-2.351141009)--(0,0);
    \draw[line width=3,->] (0,0)--(2.5,0);
    \draw[line width=3,->] (7.2360679775,-2.351141009)--(5.21352549156,-0.881677878);
    \draw[line width=3,->>] (7.2360679775,-2.351141009)--(4.5360679775,-2.351141009);
    \draw[line width=3,->>] (0,0)--(2.1843458848,-1.5870201811896775);
    
    \foreach \i/\j in {1.5/0,6.0225424859/-1.46946313056}
    {
    \draw (\i,\j) node[shape=circle,fill=white,font={\Huge}]{\color{blue}$3$};
    }
    \foreach \i/\j in {5.9360679775/-2.351141009,1.051722092687/-0.76412082798}
    {
    \draw (\i,\j) node[shape=circle,fill=white,font={\Huge}]{\color{blue}$1$};
    }

    \coordinate (-A) at (0,0);
        \coordinate (-B) at (4,0);
        \coordinate (-C) at (7.2360679775,-2.351141009);
        \coordinate (-D) at (3.2360679775,-2.351141009);
    
}
}

\tikzset{proto2b/.pic={
    \draw[line width=1mm] (0,0)--(4,0)--(7.2360679775,-2.351141009)--(3.2360679775,-2.351141009)--(0,0);
    \draw[line width=3,->>] (0,0)--(2.7,0);
    \draw[line width=3,->>] (7.2360679775,-2.351141009)--(5.0517220926876,-0.76412082781);
    \draw[line width=3,->] (7.2360679775,-2.351141009)--(4.7360679775,-2.351141009);
    \draw[line width=3,->] (0,0)--(2.0225424859,-1.46946313073);
    
    \foreach \i/\j in {1.3/0,6.18434588481/-1.587020181019785}
    {
    \draw (\i,\j) node[shape=circle,fill=white,font={\Huge}]{\color{blue}$1$};
    }
    \foreach \i/\j in {5.7360679775/-2.351141009,1.21352549156/-0.8816778784387}
    {
    \draw (\i,\j) node[shape=circle,fill=white,font={\Huge}]{\color{blue}$3$};
    }

    \coordinate (-A) at (0,0);
        \coordinate (-B) at (4,0);
        \coordinate (-C) at (7.2360679775,-2.351141009);
        \coordinate (-D) at (3.2360679775,-2.351141009);
    
}
}

\begin{center}
\setlength{\tabcolsep}{2pt}
    \begin{longtable}[c]{c c c}
    \resizebox{65pt}{!}{
    \tikz{
    \pic (T1) [scale=1.2,rotate=0] {proto1a};
    \pic (T2) [scale=1.2,rotate=90] at (T1-B) {proto2a};
    \pic (T3) [xscale=-1.2,yscale=1.2,rotate=90] at (T1-D) {proto2a};

    \node[shape=circle, draw=red, fill=pink!70, text=black,font={\Huge}] at (T1-C) {3};
    \foreach \i/\j in {T1/-B,T1/-D,T2/-C}
    {
    \node[shape=circle, draw=black, fill=black,minimum size=7mm] at (\i\j) {};
    }
    }}

    & 
    \resizebox{140pt}{!}{
    \tikz{
    \pic (T1) [scale=1.2,rotate=0] {proto1b};
    \pic (T2) [scale=1.2,rotate=72] at (T1-A) {proto1b};
    \pic (T3) [scale=1.2,rotate=-72] at (T1-A) {proto1b};
    \pic (T4) [xscale=1.2,yscale=-1.2,rotate=18] at (T2-B) {proto2a};

    \node[shape=circle, draw=red, fill=red!60, text=black,font={\Huge}] at (T1-A) {4};
    \foreach \i/\j in {T1/-B,T1/-D,T4/-A,T4/-C}
    {
    \node[shape=circle, draw=black, fill=black,minimum size=7mm] at (\i\j) {};
    }
    }}

    & 
    \resizebox{140pt}{!}{
    \tikz{
    \pic (T1) [scale=1.2,rotate=0] {proto1b};
    \pic (T2) [scale=1.2,rotate=72] at (T1-A) {proto1b};
    \pic (T3) [scale=1.2,rotate=-72] at (T1-A) {proto1b};
    \pic (T4) [scale=1.2,rotate=144] at (T1-A) {proto1b};
    \pic (T5) [scale=1.2,rotate=-144] at (T1-A) {proto1b};

    \node[shape=circle, draw=red, fill=purple!50, text=black,font={\Huge}] at (T1-A) {5};
    \foreach \i/\j in {T1/-B,T1/-D,T2/-B,T4/-B,T5/-B}
    {
    \node[shape=circle, draw=black, fill=black,minimum size=7mm] at (\i\j) {};
    }
    }}\\
    
    \tiny Vertex Configuration 1 & \tiny Vertex Configuration 2 & \tiny Vertex Configuration 3\\
    
    \resizebox{100pt}{!}{
    \tikz{
    \pic (T1) [xscale=1.2,yscale=-1.2,rotate=18] {proto2a};
    \pic (T2) [xscale=1.2,yscale=-1.2,rotate=-36] at (T1-A){proto1a};
    \pic (T2) [xscale=1.2,yscale=-1.2,rotate=36] at (T1-C){proto1a};
    \node[shape=circle, draw=blue, fill=cyan!20, text=black,font={\Huge}] at (T1-B) {7};
    \foreach \i/\j in {T1/-A,T1/-C,T2/-C}
    {
    \node[shape=circle, draw=black, fill=black,minimum size=7mm] at (\i\j) {};}
    }
    }

     &
    \resizebox{140pt}{!}{
    \tikz{
    \pic (T1) [scale=1.2,rotate=0] {proto1a};
    \pic (T2) [scale=1.2,rotate=72] at (T1-A) {proto1a};
    \pic (T3) [scale=1.2,rotate=-72] at (T1-A) {proto1a};
    \pic (T4) [scale=1.2,rotate=144] at (T1-A) {proto1a};
    \pic (T5) [scale=1.2,rotate=-144] at (T1-A) {proto1a};

    \node[shape=circle, draw=blue, fill=cyan!80, text=black,font={\Huge}] at (T1-A) {15};
    \foreach \i/\j in {T1/-B,T1/-D,T2/-B,T4/-B,T5/-B}
    {
    \node[shape=circle, draw=black, fill=black,minimum size=7mm] at (\i\j) {};
    }
    }}
        & 
        \resizebox{140pt}{!}{
        \tikz{
        \pic (T1) [scale=1.2,rotate=72] {proto1a};
        \pic (T2) [scale=1.2,rotate=144] at (T1-A) {proto1a};
        \pic (T3) [scale=1.2,rotate=-144] at (T1-A) {proto1a};
        \pic (T4) [scale=1.2,rotate=-72] at (T1-A) {proto1a};
        \pic (T5) [scale=1.2,rotate=126] at (T1-A) {proto2a};
        \pic (T6) [xscale=-1.2,yscale=1.2,rotate=126] at (T1-A) {proto2a};

        \node[shape=circle, draw=blue, fill=blue!50, text=black,font={\Huge}] at (T1-A) {16};
    \foreach \i/\j in {T1/-B,T1/-D,T2/-B,T3/-B,T4/-B,T6/-D}
    {
    \node[shape=circle, draw=black, fill=black,minimum size=7mm] at (\i\j) {};}
        }}\\

    \tiny Vertex Configuration 4 & \tiny Vertex Configuration 5 & \tiny Vertex Configuration 6
    \end{longtable}
    \end{center}

    \begin{center}
\setlength{\tabcolsep}{2pt}
    \begin{tabular}{c c}
    \resizebox{175pt}{!}{
    \tikz{
    \pic (T1) [scale=1.2,rotate=0] {proto1a};
    \pic (T2) [scale=1.2,rotate=144] at (T1-A) {proto1a};
    \pic (T3) [scale=1.2,rotate=-144] at (T1-A) {proto1a};
    \pic (T4) [scale=1.2,rotate=162] at (T1-A) {proto2b};
    \pic (T5) [scale=1.2,rotate=198] at (T1-A) {proto2a};
    \pic (T6) [scale=1.2,rotate=54] at (T1-A) {proto2a};
    \pic (T7) [scale=1.2,rotate=18] at (T1-A) {proto2b};

    \node[shape=circle, draw=blue, fill=teal!70, text=black,font={\Huge}] at (T1-A) {17};
    \foreach \i/\j in {T1/-B,T1/-D,T2/-B,T4/-B,T5/-B,T7/-B,T7/-D}
    {
    \node[shape=circle, draw=black, fill=black,minimum size=7mm] at (\i\j) {};
    }
    }}

        & 
     \resizebox{145pt}{!}{
    \tikz{
    \pic (T1) [scale=1.2,rotate=0] {proto1a};
    \pic (T2) [scale=1.2,rotate=162] at (T1-A) {proto2b};
    \pic (T3) [scale=1.2,rotate=54] at (T1-A) {proto2a};
    \pic (T4) [scale=1.2,rotate=-144] at (T2-B) {proto1b};
    \pic (T5) [scale=1.2,rotate=144] at (T3-D) {proto1b};
    \node[shape=circle, draw=violet, fill=violet!40, text=black,font={\Huge}] at (T1-A) {11};
    \foreach \i/\j in {T1/-B,T1/-D,T2/-B,T4/-C,T5/-A}
    {
    \node[shape=circle, draw=black, fill=black,minimum size=7mm] at (\i\j) {};
    }
    }}\\

     \tiny Vertex Configuration 7 & \tiny Vertex Configuration 8
    \end{tabular}
\end{center}

    Next we must show that these vertex configurations are not adjacent to themselves.

     For configurations 1, 2, and 3 all arrows are double edge arrows pointing in. Thus, they cannot be adjacent to themselves, as there is no arrow pointing out.

     Configuration 4 is incident to single arrows pointing in and a double arrow pointing out and therefore, will only be adjacent to vertices with single arrows pointing out and a double arrow pointing in.

     Configurations 5, 6, and 7 are incident to only arrows pointing out and will thus only be adjacent to vertices with arrows pointing in.

     We will look at configuration 8 a little closer. This vertex is incident to two double arrows pointing out and cannot be adjacent to itself via those edges as those adjacent vertices have double arrows pointing in. This vertex configuration is also incident to three single arrows, two pointing out and one pointing in. When we follow the single arrows pointing out, the adjacent vertices are incident to two single arrows pointing in. Since configuration 8 only has one single arrow pointing in, it cannot be adjacent to itself via these edges. When we follow the single arrow pointing in, the adjacent vertex is incident to three single arrows pointing out. Since configuration 8 only has two single arrows pointing out it cannot be adjacent to itself via this edge either.
\end{proof}

\tikzset{proto1a/.pic={
    \draw[line width=1mm] (0,0)--(-2.351141009,3.2360679775)--(0,6.472133955)--(2.351141009,3.2360679775)--(0,0);
    \draw[line width=3,->] (0,0)--(-1.4694631307,2.0225424859);
    \draw[line width=3,->] (0,0)--(1.4694631307,2.0225424859);
    \draw[line width=3,->>] (-2.351141009,3.2360679775)--(-0.76412082781,5.4204138623);
    \draw[line width=3,->>] (2.351141009,3.2360679775)--(0.76412082781,5.4204138623);
    
    \foreach \i/\j in {-0.881677878/1.21352549,0.881677878/1.21352549}
    {
    \draw (\i,\j) node[shape=circle,fill=white,font={\Huge}]{\color{blue}$3$};
    }
    \foreach \i/\j in {-1.58702018/4.28779007,1.58702018/4.28779007}
    {
    \draw (\i,\j) node[shape=circle,fill=white,font={\Huge}]{\color{blue}$1$};
    }

    \coordinate (-A) at (0,0);
        \coordinate (-B) at (-2.351141009,3.2360679775);
        \coordinate (-C) at (0,6.472133955);
        \coordinate (-D) at (2.351141009,3.2360679775);
    
}
}

\tikzset{proto1b/.pic={
    \draw[line width=1mm] (0,0)--(-2.351141009,3.2360679775)--(0,6.472133955)--(2.351141009,3.2360679775)--(0,0);
    \draw[line width=3,->] (0,6.472135955)--(-1.4694631307,4.4495934691);
    \draw[line width=3,->] (0,6.472135955)--(1.4694631307,4.4495934691);
    \draw[line width=3,->>] (-2.351141009,3.2360679775)--(-0.76412082781,1.0517220927);
    \draw[line width=3,->>] (2.351141009,3.2360679775)--(0.76412082781,1.0517220927);
    
    \foreach \i/\j in {-0.881677878/5.258610465,0.881677878/5.258610465}
    {
    \draw (\i,\j) node[shape=circle,fill=white,font={\Huge}]{\color{blue}$3$};
    }
    \foreach \i/\j in {-1.58702018/2.184345885,1.58702018/2.184345885}
    {
    \draw (\i,\j) node[shape=circle,fill=white,font={\Huge}]{\color{blue}$1$};
    }

    \coordinate (-A) at (0,0);
        \coordinate (-B) at (-2.351141009,3.2360679775);
        \coordinate (-C) at (0,6.472133955);
        \coordinate (-D) at (2.351141009,3.2360679775);
    
}
}

\tikzset{proto2a/.pic={
    \draw[line width=1mm] (0,0)--(4,0)--(7.2360679775,-2.351141009)--(3.2360679775,-2.351141009)--(0,0);
    \draw[line width=3,->] (0,0)--(2.5,0);
    \draw[line width=3,->] (7.2360679775,-2.351141009)--(5.21352549156,-0.881677878);
    \draw[line width=3,->>] (7.2360679775,-2.351141009)--(4.5360679775,-2.351141009);
    \draw[line width=3,->>] (0,0)--(2.1843458848,-1.5870201811896775);
    
    \foreach \i/\j in {1.5/0,6.0225424859/-1.46946313056}
    {
    \draw (\i,\j) node[shape=circle,fill=white,font={\Huge}]{\color{blue}$3$};
    }
    \foreach \i/\j in {5.9360679775/-2.351141009,1.051722092687/-0.76412082798}
    {
    \draw (\i,\j) node[shape=circle,fill=white,font={\Huge}]{\color{blue}$1$};
    }

    \coordinate (-A) at (0,0);
        \coordinate (-B) at (4,0);
        \coordinate (-C) at (7.2360679775,-2.351141009);
        \coordinate (-D) at (3.2360679775,-2.351141009);
    
}
}

\tikzset{proto2b/.pic={
    \draw[line width=1mm] (0,0)--(4,0)--(7.2360679775,-2.351141009)--(3.2360679775,-2.351141009)--(0,0);
    \draw[line width=3,->>] (0,0)--(2.7,0);
    \draw[line width=3,->>] (7.2360679775,-2.351141009)--(5.0517220926876,-0.76412082781);
    \draw[line width=3,->] (7.2360679775,-2.351141009)--(4.7360679775,-2.351141009);
    \draw[line width=3,->] (0,0)--(2.0225424859,-1.46946313073);
    
    \foreach \i/\j in {1.3/0,6.18434588481/-1.587020181019785}
    {
    \draw (\i,\j) node[shape=circle,fill=white,font={\Huge}]{\color{blue}$1$};
    }
    \foreach \i/\j in {5.7360679775/-2.351141009,1.21352549156/-0.8816778784387}
    {
    \draw (\i,\j) node[shape=circle,fill=white,font={\Huge}]{\color{blue}$3$};
    }

    \coordinate (-A) at (0,0);
        \coordinate (-B) at (4,0);
        \coordinate (-C) at (7.2360679775,-2.351141009);
        \coordinate (-D) at (3.2360679775,-2.351141009);
    
}
}

\begin{center}
\resizebox{450pt}{!}{
    \tikz{
    \clip (-21.5,-21.5) rectangle (24,10);
    \pic (T1) [scale=1.2,rotate=0] {proto1b};
    \pic (T2) [scale=1.2,rotate=72] at (T1-A){proto1b};
    \pic (T3) [scale=1.2,rotate=-72] at (T1-A){proto1b};
    \pic (T4) [scale=1.2,rotate=144] at (T1-A){proto1b};
    \pic (T5) [scale=1.2,rotate=-144] at (T1-A){proto1b};
    \pic (T6) [scale=1.2,rotate=54] at (T2-C){proto2b};
    \pic (T7) [scale=1.2,rotate=-18] at (T1-C){proto2b};
    \pic (T8) [scale=1.2,rotate=126] at (T4-C){proto2b};
    \pic (T9) [scale=1.2,rotate=18] at (T4-C){proto2a};
    \pic (T10) [scale=1.2,rotate=90] at (T5-C){proto2a};
    \pic (T11) [scale=1.2,rotate=-36] at (T6-B){proto1b};
    \pic (T12) [scale=1.2,rotate=36] at (T7-B){proto1b};
    \pic (T13) [scale=1.2,rotate=-36] at (T7-B){proto1b};
    \pic (T14) [scale=1.2,rotate=-108] at (T7-B){proto1b};
    \pic (T15) [scale=1.2,rotate=-36] at (T10-D){proto1b};
    \pic (T16) [scale=1.2,rotate=-108] at (T10-D){proto1b};
    \pic (T17) [scale=1.2,rotate=180] at (T10-D){proto1b};
    \pic (T18) [scale=1.2,rotate=-108] at (T9-D){proto1b};
    \pic (T19) [scale=1.2,rotate=180] at (T9-D){proto1b};
    \pic (T20) [scale=1.2,rotate=108] at (T9-D){proto1b};
    \pic (T21) [scale=1.2,rotate=180] at (T8-B){proto1b};
    \pic (T22) [scale=1.2,rotate=108] at (T8-B){proto1b};
    \pic (T23) [scale=1.2,rotate=36] at (T8-B){proto1b};
    \pic (T24) [scale=1.2,rotate=108] at (T6-B){proto1b};
    \pic (T25) [scale=1.2,rotate=36] at (T6-B){proto1b};
    \pic (T26) [scale=1.2,rotate=18] at (T12-C){proto2b};
    \pic (T27) [scale=1.2,rotate=-54] at (T13-C){proto2b};
    \pic (T28) [scale=1.2,rotate=90] at (T14-C){proto2b};
    \pic (T29) [scale=1.2,rotate=-72] at (T28-A){proto1a};
    \pic (T30) [scale=1.2,rotate=-18] at (T28-A){proto2a};
    \pic (T31) [scale=1.2,rotate=-54] at (T28-A){proto2b};
    \pic (T32) [scale=1.2,rotate=-144] at (T30-D){proto1b};
    \pic (T33) [scale=1.2,rotate=234] at (T31-C){proto2b};
    \pic (T34) [scale=1.2,rotate=-72] at (T33-B){proto1b};
    \pic (T35) [scale=1.2,rotate=18] at (T17-C){proto2b};
    \pic (T36) [scale=1.2,rotate=-144] at (T17-C){proto1a};
    \pic (T37) [scale=1.2,rotate=-126] at (T17-C){proto2b};
    \pic (T38) [scale=1.2,rotate=-90] at (T17-C){proto2a};
    \pic (T39) [scale=1.2,rotate=144] at (T38-D){proto1b};
    \pic (T40) [scale=1.2,rotate=162] at (T19-C){proto2b};
    \pic (T41) [scale=1.2,rotate=-144] at (T40-B){proto1b};
    \pic (T42) [scale=1.2,rotate=-54] at (T21-C){proto2b};
    \pic (T43) [scale=1.2,rotate=144] at (T21-C){proto1a};
    \pic (T44) [scale=1.2,rotate=198] at (T21-C){proto2a};
    \pic (T45) [scale=1.2,rotate=162] at (T21-C){proto2b};
    \pic (T46) [scale=1.2,rotate=72] at (T45-B){proto1b};
    \pic (T47) [scale=1.2,rotate=-36] at (T46-C){proto1a};
    \pic (T48) [scale=1.2,rotate=90] at (T46-D){proto2b};
    \pic (T49) [scale=1.2,rotate=-126] at (T23-C){proto2b};
    \pic (T50) [scale=1.2,rotate=72] at (T23-C){proto1a};
    \pic (T51) [scale=1.2,rotate=90] at (T23-C){proto2b};
    \pic (T52) [xscale=-1.2,yscale=1.2,rotate=90] at (T23-C){proto2b};
    \pic (T53) [scale=1.2,rotate=0] at (T51-B){proto1b};
    \pic (T54) [scale=1.2,rotate=18] at (T53-D){proto2b};
    \pic (T55) [scale=1.2,rotate=0] at (T11-C){proto1a};
    \pic (T56) [scale=1.2,rotate=162] at (T11-C){proto2b};
    \pic (T57) [scale=1.2,rotate=54] at (T11-C){proto2a};
    \pic (T58) [scale=1.2,rotate=-72] at (T57-D){proto1b};
    \pic (T59) [scale=1.2,rotate=0] at (T27-B){proto1b};
    \pic (T60) [scale=1.2,rotate=-108] at (T59-C){proto1a};
    \pic (T61) [scale=1.2,rotate=18] at (T59-D){proto2b};
    \pic (T62) [scale=1.2,rotate=0] at (T61-B){proto1b};
    \pic (T63) [scale=1.2,rotate=-72] at (T61-B){proto1b};
    \pic (T64) [scale=1.2,rotate=180] at (T63-C){proto1a};
    \pic (T65) [scale=1.2,rotate=36] at (T29-C){proto1b};
    \pic (T66) [scale=1.2,rotate=-36] at (T29-C){proto1b};
    \pic (T67) [scale=1.2,rotate=-54] at (T66-C){proto2b};
    \pic (T68) [scale=1.2,rotate=-108] at (T29-C){proto1b};
    \pic (T69) [scale=1.2,rotate=180] at (T29-C){proto1b};
    \pic (T70) [scale=1.2,rotate=54] at (T30-C){proto2a};
    \pic (T71) [scale=1.2,rotate=-72] at (T70-D){proto1b};
    \pic (T72) [scale=1.2,rotate=-36] at (T32-C){proto1a};
    \pic (T73) [scale=1.2,rotate=-144] at (T70-D){proto1b};
    \pic (T74) [scale=1.2,rotate=18] at (T32-C){proto2a};
    \pic (T75) [scale=1.2,rotate=-18] at (T32-C){proto2b};
    \pic (T76) [scale=1.2,rotate=180] at (T32-C){proto1a};
    \pic (T77) [scale=1.2,rotate=-108] at (T75-B){proto1b};
    \pic (T78) [scale=1.2,rotate=144] at (T34-D){proto1a};
    \pic (T79) [scale=1.2,rotate=-54] at (T34-D){proto2b};
    \pic (T80) [scale=1.2,rotate=-72] at (T76-C){proto1b};
    \pic (T81) [scale=1.2,rotate=-144] at (T76-C){proto1b};
    \pic (T82) [scale=1.2,rotate=90] at (T81-C){proto2a};
    \pic (T83) [scale=1.2,rotate=54] at (T82-A){proto2b};
    \pic (T84) [scale=1.2,rotate=-108] at (T82-A){proto1a};
    \pic (T85) [xscale=1.2,yscale=-1.2,rotate=90] at (T82-A){proto2b};
    \pic (T86) [scale=1.2,rotate=-90] at (T82-A){proto2b};
    \pic (T87) [scale=1.2,rotate=108] at (T82-A){proto1a};
    \pic (T88) [scale=1.2,rotate=-108] at (T36-C){proto1b};
    \pic (T89) [scale=1.2,rotate=180] at (T36-C){proto1b};
    \pic (T90) [scale=1.2,rotate=54] at (T89-C){proto2a};
    \pic (T91) [scale=1.2,rotate=108] at (T36-C){proto1b};
    \pic (T92) [scale=1.2,rotate=-18] at (T39-B){proto2a};
    \pic (T93) [scale=1.2,rotate=-108] at (T41-C){proto1a};
    \pic (T94) [xscale=-1.2,yscale=1.2,rotate=-90] at (T41-C){proto2b};
    \pic (T95) [scale=1.2,rotate=270] at (T41-C){proto2b};
    \pic (T96) [scale=1.2,rotate=108] at (T41-C){proto1a};
    \pic (T97) [scale=1.2,rotate=234] at (T41-D){proto2b};
    \pic (T98) [scale=1.2,rotate=-108] at (T43-C){proto1b};
    \pic (T99) [scale=1.2,rotate=180] at (T43-C){proto1b};
    \pic (T100) [scale=1.2,rotate=36] at (T43-C){proto1b};
    \pic (T101) [scale=1.2,rotate=108] at (T43-C){proto1b};
    \pic (T102) [scale=1.2,rotate=-18] at (T101-C){proto2a};
    \pic (T103) [scale=1.2,rotate=72] at (T102-D){proto1b};
    \pic (T104) [xscale=-1.2,yscale=1.2,rotate=-90] at (T103-C){proto2b};
    \pic (T105) [scale=1.2,rotate=270] at (T103-C){proto2b};
    \pic (T106) [scale=1.2,rotate=-36] at (T103-C){proto1a};
    \pic (T107) [scale=1.2,rotate=126] at (T103-C){proto2b};
    \pic (T108) [scale=1.2,rotate=90] at (T103-D){proto2b};
    \pic (T109) [scale=1.2,rotate=0] at (T108-B){proto1b};
    \pic (T110) [scale=1.2,rotate=72] at (T108-B){proto1b};
    \pic (T111) [scale=1.2,rotate=234] at (T47-A){proto2a};
    \pic (T112) [scale=1.2,rotate=198] at (T47-A){proto2b};
    \pic (T113) [scale=1.2,rotate=36] at (T47-A){proto1a};
    \pic (T114) [scale=1.2,rotate=0] at (T47-B){proto1a};
    \pic (T115) [scale=1.2,rotate=162] at (T113-D){proto2b};
    \pic (T116) [scale=1.2,rotate=108] at (T50-C){proto1b};
    \pic (T117) [scale=1.2,rotate=36] at (T50-C){proto1b};
    \pic (T118) [scale=1.2,rotate=-36] at (T50-C){proto1b};
    \pic (T119) [scale=1.2,rotate=18] at (T117-C){proto2b};
    \pic (T120) [scale=1.2,rotate=-90] at (T117-C){proto2a};
    \pic (T121) [scale=1.2,rotate=0] at (T120-D){proto1b};
    \pic (T122) [scale=1.2,rotate=72] at (T119-B){proto1b};
    \pic (T123) [scale=1.2,rotate=0] at (T119-B){proto1b};
    \pic (T124) [scale=1.2,rotate=-72] at (T119-B){proto1b};
    \pic (T125) [scale=1.2,rotate=72] at (T54-B){proto1b};

    \foreach \i/\j in {T26/-B,T27/-B,T32/-A,T33/-B,T75/-B,T38/-D,T41/-A,T46/-A,T47/-C,T53/-A,T54/-B}
    {
    \node[shape=circle, draw=red, fill=pink!70, text=black,font={\Huge}] at (\i\j) {3};
    }
    \foreach \i/\j in {T6/-B,T7/-B,T10/-D,T8/-B,T9/-D,T61/-B,T70/-D,T76/-C,T106/-C,T119/-B}
    {
    \node[shape=circle, draw=red, fill=red!60, text=black,font={\Huge}] at (\i\j) {4};
    }
    \foreach \i/\j in {T1/-A,T29/-C,T36/-C,T98/-A,T50/-C}
    {
    \node[shape=circle, draw=red, fill=purple!50, text=black,font={\Huge}] at (\i\j) {5};
    }
    \foreach \i/\j in {T4/-B,T4/-D,T2/-D,T1/-D,T3/-D,T26/-D,T13/-D,T61/-D,T28/-D,T66/-D,T29/-D,T69/-B,T31/-D,T16/-D,T74/-B,T18/-D,T35/-D,T75/-D,T80/-D,T79/-D,T88/-D,T89/-D,T36/-D,T19/-D,T97/-D,T98/-B,T99/-D,T101/-D,T106/-B,T44/-B,T110/-D,T21/-D,T23/-B,T114/-B,T114/-D,T120/-B,T118/-B,T50/-D,T24/-D,T25/-D}
    {
    \node[shape=circle, draw=blue, fill=cyan!20, text=black,font={\Huge}] at (\i\j) {7};
    }
    \foreach \i/\j in {}
    {
    \node[shape=circle, draw=blue, fill=cyan!80, text=black,font={\Huge}] at (\i\j) {15};
    }
    \foreach \i/\j in {T32/-C,T39/-C,T109/-C}
    {
    \node[shape=circle, draw=blue, fill=blue!50, text=black,font={\Huge}] at (\i\j) {16};
    }
    \foreach \i/\j in {T12/-C,T14/-C,T18/-C,T81/-C,T20/-C,T23/-C}
    {
    \node[shape=circle, draw=blue, fill=teal!70, text=black,font={\Huge}] at (\i\j) {17};
    }
    \foreach \i/\j in {T1/-C,T2/-C,T4/-C,T5/-C,T3/-C,T26/-C,T59/-D,T61/-C,T69/-C,T16/-C,T35/-C,T81/-D,T91/-C,T39/-D,T41/-D,T101/-C,T108/-C,T22/-C,T47/-B,T117/-C,T52/-C,T51/-C}
    {
    \node[shape=circle, draw=violet, fill=violet!40, text=black,font={\Huge}] at (\i\j) {11};
    }
    }}
\end{center}

\subsection{Domino Variant}

\begin{proposition}
    Consider the non-periodic tiling of the plane with the \href{https://tilings.math.uni-bielefeld.de/substitution/domino-variant/}{domino variant} in \cite{TilingEncyclopedia}. This tiling is explored in \cite{DomVar} and \cite{DomVar2}. We identify the following supertile which is half of a level 2 supertile. Assign external edges a weight of 3, internal ''long" edges a weight of 1, and internal''short" edges a weight of 2. (Alternatively, colour the corners yellow, the middle external vertices blue, and the rest of the external vertices red. Assign weights of 3 to all external edges, weights of 1 to internal edges incident to a red vertex, and weighs of 2 to internal edges incident to a blue vertex. Then assign the remaining edges a weight of 2.) This provides a solution to the 1-2-3 problem.
\end{proposition}

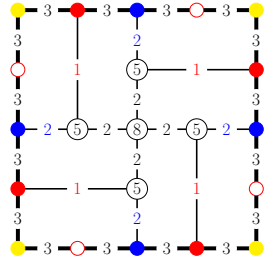
\tikzset{dominovar/.pic={
\draw[line width=2mm] (0,0)--(0,4)--(4,4)--(4,0)--(0,0);
\draw[line width=0.8mm] (0,2)--(4,2);
\draw[line width=0.8mm] (2,4)--(2,0);
\draw[line width=0.8mm] (1,4)--(1,2);
\draw[line width=0.8mm] (3,0)--(3,2);
\draw[line width=0.8mm] (0,1)--(2,1);
\draw[line width=0.8mm] (2,3)--(4,3);
\foreach \i/\j in {0/0.5,0/1.5,0/2.5,0/3.5,0.5/4,1.5/4,2.5/4,3.5/4,4/3.5,4/2.5,4/1.5,4/0.5,0.5/0,1.5/0,2.5/0,3.5/0}
{
\draw (\i,\j) node[shape=circle,fill=white,font={\Huge}] {\color{black}$3$};
}
\foreach \i/\j in {1/3,3/3,3/1,1/1}
{
\draw (\i,\j) node[shape=circle,fill=white,font={\Huge}] {\color{red}$1$};
}
\foreach \i/\j in {1.5/2,2.5/2,2/1.5,2/2.5}
{
\draw (\i,\j) node[shape=circle,fill=white,font={\Huge}] {\color{black}$2$};
}
\foreach \i/\j in {0.5/2,3.5/2,2/0.5,2/3.5}
{
\draw (\i,\j) node[shape=circle,fill=white,font={\Huge}] {\color{blue}$2$};
}
\foreach \i/\j in {2/3,3/2,2/1,1/2}
{
\node[shape=circle, draw=black, fill=white, text=black,font={\Huge}] at (\i,\j) {$5$};
}
\node[shape=circle, draw=black, fill=white, text=black,font={\Huge}] at (2,2) {$8$};
\foreach \i/\j in {0/0,0/4,4/4,4/0}
{\node[shape=circle, draw=yellow, fill=yellow, text=black,minimum size=7mm] at (\i,\j) {};}
\foreach \i/\j in {0/2,2/4,4/2,2/0}
{\node[shape=circle, draw=blue, fill=blue, text=black,minimum size=7mm] at (\i,\j) {};}
\foreach \i/\j in {1/4,4/3,3/0,0/1}
{\node[shape=circle, draw=red, fill=red, text=black,minimum size=7mm] at (\i,\j) {};}
\foreach \i/\j in {0/3,3/4,4/1,1/0}
{\node[shape=circle, draw=red, fill=white, text=black,minimum size=7mm] at (\i,\j) {};}
}}

\begin{center}
    \resizebox{100pt}{!}{
    \tikz{
    \pic[scale=3] {dominovar};
    }
    }
\end{center}

\begin{proof}
    As the supertiles are combined, the corners, coloured yellow, will be a junction of 4 supertiles and will sum to 12. The middle external vertices, coloured blue, will be degree 4 and sum up to 10. The solid red vertices will sum to either 7 or 8. The open red vertices will either not exist or sum up to 7. Each colour has a distinct sum and are never adjacent with themselves. Also, notice that the external sums are distinct from the internal sums. Thus no conflicts arise.
\end{proof}

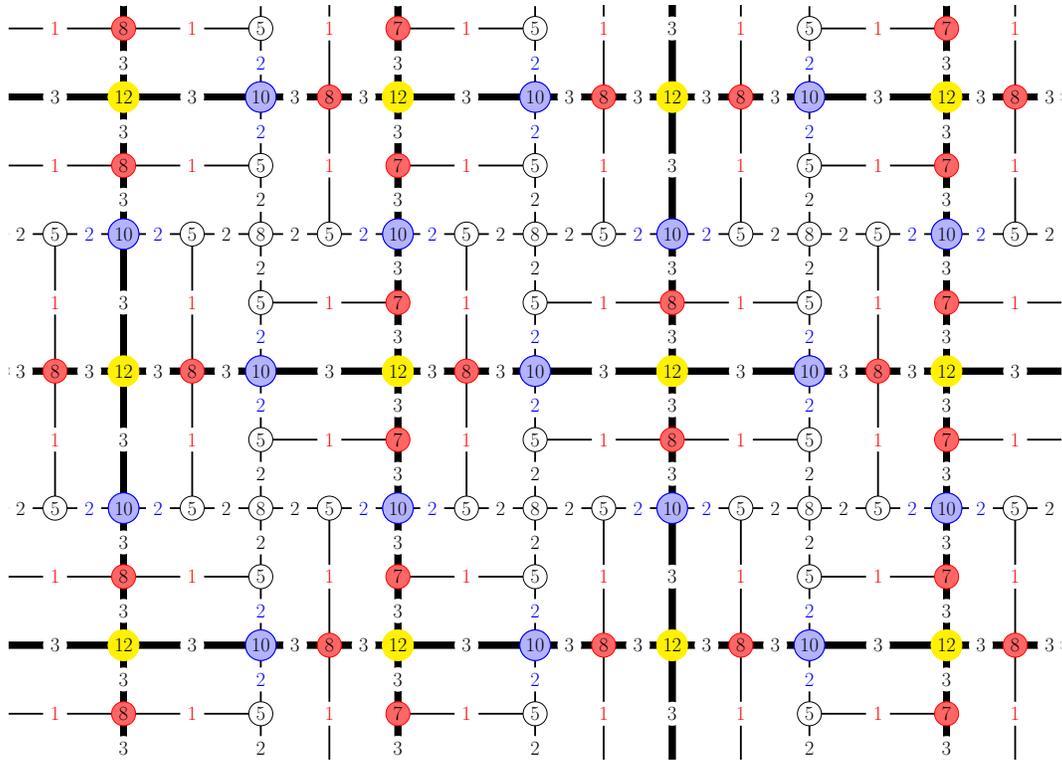
\tikzset{dominovar/.pic={
\draw[line width=3mm] (0,0)--(0,4)--(4,4)--(4,0)--(0,0);
\draw[line width=0.8mm] (0,2)--(4,2);
\draw[line width=0.8mm] (2,4)--(2,0);
\draw[line width=0.8mm] (1,4)--(1,2);
\draw[line width=0.8mm] (3,0)--(3,2);
\draw[line width=0.8mm] (0,1)--(2,1);
\draw[line width=0.8mm] (2,3)--(4,3);
\foreach \i/\j in {0/0.5,0/1.5,0.5/4,1.5/4,4/3.5,4/2.5,2.5/0,3.5/0}
{
\draw (\i,\j) node[shape=circle,fill=white,font={\Huge}] {\color{black}$3$};
}
\foreach \i/\j in {1/3,3/3,3/1,1/1}
{
\draw (\i,\j) node[shape=circle,fill=white,font={\Huge}] {\color{red}$1$};
}
\foreach \i/\j in {1.5/2,2.5/2,2/1.5,2/2.5}
{
\draw (\i,\j) node[shape=circle,fill=white,font={\Huge}] {\color{black}$2$};
}
\foreach \i/\j in {0.5/2,3.5/2,2/0.5,2/3.5}
{
\draw (\i,\j) node[shape=circle,fill=white,font={\Huge}] {\color{blue}$2$};
}
\foreach \i/\j in {2/3,3/2,2/1,1/2}
{
\node[shape=circle, draw=black, fill=white, text=black,font={\Huge}] at (\i,\j) {$5$};
}
\node[shape=circle, draw=black, fill=white, text=black,font={\Huge}] at (2,2) {$8$};
\foreach \i/\j in {0/0,0/4,4/4,4/0}
{\node[shape=circle, draw=yellow, fill=yellow, text=black,font={\Huge}] at (\i,\j) {$12$};}
\foreach \i/\j in {0/2,2/4,4/2,2/0}
{\node[shape=circle, draw=blue, fill=blue!30, text=black,font={\Huge}] at (\i,\j) {$10$};}
\foreach \i/\j in {1/4,4/3,3/0,0/1}
{\node[shape=circle, draw=red, fill=red, text=black,minimum size=7mm] at (\i,\j) {};}

\coordinate (-A) at (0,0);
\coordinate (-B) at (0,4);
\coordinate (-C) at (4,4);
\coordinate (-D) at (4,0);
\coordinate (-a) at (0,2.5);
\coordinate (-b) at (0,3);
\coordinate (-c) at (0,3.5);
\coordinate (-d) at (2.5,4);
\coordinate (-e) at (3,4);
\coordinate (-f) at (3.5,4);
\coordinate (-g) at (4,1.5);
\coordinate (-h) at (4,1);
\coordinate (-i) at (4,0.5);
\coordinate (-j) at (0.5,0);
\coordinate (-k) at (1,0);
\coordinate (-l) at (1.5,0);
\coordinate (-aa) at (0,1);
\coordinate (-bb) at (1,4);
\coordinate (-cc) at (4,3);
\coordinate (-dd) at (3,0);
}}

\begin{center}
    \resizebox{400pt}{!}{
    \tikz{
    \clip (7,7) rectangle (53,40);
    \pic (T1) [scale=3] {dominovar};
    \pic (T2) [xscale=-3,yscale=3] at (T1-C) {dominovar};
    \pic (T3) [xscale=3,yscale=3] at (T2-C) {dominovar};
    \pic (T4) [xscale=-3,yscale=3] at (T3-C) {dominovar};
    \pic (T5) [xscale=-3,yscale=3] at (T4-B) {dominovar};
    \pic (T6) [xscale=3,yscale=3] at (T5-C) {dominovar};
    \pic (T7) [xscale=-3,yscale=3] at (T6-C) {dominovar};
    \pic (T8) [xscale=3,yscale=3] at (T7-C) {dominovar};
    \pic (T9) [xscale=-3,yscale=3,rotate=180] at (T1-C) {dominovar};
    \pic (T10) [xscale=3,yscale=3] at (T2-A) {dominovar};
    \pic (T11) [xscale=-3,yscale=3] at (T10-C) {dominovar};
    \pic (T12) [xscale=3,yscale=3] at (T4-A) {dominovar};
    \pic (T13) [xscale=3,yscale=3] at (T5-A) {dominovar};
    \pic (T14) [xscale=-3,yscale=3] at (T13-C) {dominovar};
    \pic (T15) [xscale=3,yscale=3] at (T7-A) {dominovar};
    \pic (T16) [xscale=-3,yscale=3] at (T15-C) {dominovar};
    \pic (T17) [xscale=-3,yscale=3,rotate=180] at (T9-D) {dominovar};
    \pic (T18) [xscale=3,yscale=3] at (T10-D) {dominovar};
    \pic (T19) [xscale=-3,yscale=3] at (T18-C) {dominovar};
    \pic (T20) [xscale=3,yscale=3] at (T12-D) {dominovar};
    \pic (T21) [xscale=3,yscale=3] at (T13-D) {dominovar};
    \pic (T22) [xscale=-3,yscale=3] at (T21-C) {dominovar};
    \pic (T23) [xscale=3,yscale=3] at (T15-D) {dominovar};
    \pic (T24) [xscale=-3,yscale=3] at (T23-C) {dominovar};
    \pic (T25) [xscale=3,yscale=3] at (T17-C) {dominovar};
    \pic (T26) [xscale=-3,yscale=3] at (T25-C) {dominovar};
    \pic (T27) [xscale=3,yscale=3] at (T19-A) {dominovar};
    \pic (T28) [xscale=-3,yscale=3] at (T27-C) {dominovar};
    \pic (T29) [xscale=-3,yscale=3] at (T28-B) {dominovar};
    \pic (T30) [xscale=3,yscale=3] at (T29-C) {dominovar};
    \pic (T31) [xscale=-3,yscale=3] at (T30-C) {dominovar};
    \pic (T32) [xscale=3,yscale=3] at (T31-C) {dominovar};
    \pic (T33) [xscale=3,yscale=3] at (T25-D) {dominovar};
    \pic (T34) [xscale=-3,yscale=3] at (T33-C) {dominovar};
    \pic (T35) [xscale=3,yscale=3] at (T34-C) {dominovar};
    \pic (T36) [xscale=-3,yscale=3] at (T35-C) {dominovar};
    \pic (T37) [xscale=-3,yscale=3] at (T36-B) {dominovar};
    \pic (T38) [xscale=3,yscale=3] at (T37-C) {dominovar};
    \pic (T39) [xscale=-3,yscale=3] at (T38-C) {dominovar};
    \pic (T40) [xscale=3,yscale=3] at (T39-C) {dominovar};
    \pic (T41) [xscale=-3,yscale=3,rotate=180] at (T33-C) {dominovar};
    \pic (T42) [xscale=3,yscale=3] at (T41-A) {dominovar};
    \pic (T43) [xscale=-3,yscale=3] at (T42-C) {dominovar};
    \pic (T44) [xscale=3,yscale=3] at (T43-C) {dominovar};
    \pic (T45) [xscale=3,yscale=3] at (T44-B) {dominovar};
    \pic (T46) [xscale=-3,yscale=3] at (T45-C) {dominovar};
    \pic (T47) [xscale=3,yscale=3] at (T46-C) {dominovar};
    \pic (T48) [xscale=-3,yscale=3] at (T47-C) {dominovar};
    \pic (T49) [xscale=3,yscale=3] at (T41-C) {dominovar};
    \pic (T50) [xscale=-3,yscale=3] at (T49-C) {dominovar};
    \pic (T51) [xscale=3,yscale=3] at (T50-C) {dominovar};
    \pic (T52) [xscale=-3,yscale=3] at (T51-C) {dominovar};
    \pic (T53) [xscale=-3,yscale=3] at (T52-B) {dominovar};
    \pic (T54) [xscale=3,yscale=3] at (T53-C) {dominovar};
    \pic (T55) [xscale=-3,yscale=3] at (T54-C) {dominovar};
    \pic (T56) [xscale=3,yscale=3] at (T55-C) {dominovar};

    \foreach \i/\j in {T1/-a,T1/-c,T1/-e,T1/-h,T1/-j,T1/-l,T2/-i,T2/-g,T2/-e,T2/-b,T3/-a,T3/-c,T3/-e,T3/-h,T4/-i,T4/-g,T5/-j,T5/-l,T4/-b,T5/-i,T5/-g,T5/-e,T5/-b,T6/-a,T6/-c,T6/-e,T6/-h,T7/-i,T7/-g,T7/-e,T7/-b,T8/-a,T8/-c,T8/-d,T8/-f,T8/-h,T9/-k,T9/-i,T9/-g,T9/-d,T9/-f,T10/-e,T10/-g,T10/-i,T11/-e,T11/-c,T11/-a,T13/-j,T13/-l,T12/-g,T12/-i,T13/-e,T13/-g,T13/-i,T14/-e,T14/-c,T14/-a,T15/-e,T15/-g,T15/-i,T16/-f,T16/-d,T16/-c,T16/-a,T17/-a,T17/-c,T17/-d,T17/-f,T17/-h,T17/-k,T18/-a,T18/-c,T18/-e,T18/-h,T19/-i,T19/-g,T19/-e,T19/-b,T20/-a,T20/-c,T21/-j,T21/-l,T20/-h,T21/-a,T21/-c,T21/-e,T21/-h,T22/-i,T22/-g,T22/-e,T22/-b,T23/-e,T23/-h,T24/-i,T24/-g,T24/-d,T24/-f,T24/-b,T25/-j,T25/-l,T25/-e,T25/-g,T25/-i,T26/-e,T26/-c,T26/-a,T27/-e,T27/-g,T27/-i,T29/-j,T29/-l,T28/-c,T28/-a,T29/-e,T29/-c,T29/-a,T30/-e,T30/-g,T30/-i,T31/-e,T31/-c,T31/-a,T32/-d,T32/-f,T32/-g,T32/-i,T33/-l,T33/-j,T33/-c,T33/-a,T33/-e,T33/-h,T34/-g,T34/-i,T34/-e,T34/-b,T35/-a,T35/-c,T35/-e,T35/-h,T36/-i,T36/-g,T37/-j,T37/-l,T36/-b,T37/-i,T37/-g,T37/-e,T37/-b,T38/-a,T38/-c,T38/-e,T38/-h,T39/-i,T39/-g,T39/-e,T39/-b,T40/-a,T40/-c,T40/-d,T40/-f,T40/-h,T41/-d,T41/-f,T41/-k,T41/-h,T42/-e,T42/-h,T43/-e,T43/-b,T45/-j,T45/-l,T44/-h,T45/-e,T45/-h,T46/-e,T46/-b,T47/-e,T47/-h,T48/-f,T48/-d,T48/-b,T49/-j,T49/-l,T49/-e,T49/-h,T50/-e,T50/-b,T51/-e,T51/-h,T53/-j,T53/-l,T52/-b,T53/-e,T53/-b,T54/-e,T54/-h,T55/-e,T55/-b,T56/-d,T56/-f,T56/-h}
    {
    \draw (\i\j) node[shape=circle,fill=white,font={\Huge}] {\color{black}$3$};
    }
    \foreach \i/\j in {T1/-b,T1/-aa,T1/-dd,T1/-k,T2/-h,T2/-cc,T3/-aa,T3/-b,T4/-h,T4/-cc,T4/-e,T4/-bb,T5/-h,T5/-cc,T6/-aa,T6/-b,T7/-h,T7/-cc,T8/-aa,T8/-b,T8/-bb,T8/-e,T9/-h,T9/-cc,T9/-e,T10/-cc,T10/-h,T11/-b,T11/-aa,T12/-bb,T12/-e,T12/-cc,T12/-h,T13/-cc,T13/-h,T14/-b,T14/-aa,T15/-cc,T15/-h,T16/-e,T16/-bb,T16/-b,T16/-aa,T17/-bb,T17/-e,T20/-bb,T20/-e,T24/-e,T24/-bb,T25/-k,T25/-dd,T25/-cc,T25/-h,T26/-b,T26/-aa,T27/-cc,T27/-h,T28/-e,T28/-bb,T28/-b,T28/-aa,T29/-b,T29/-aa,T30/-cc,T30/-h,T31/-b,T31/-aa,T32/-e,T32/-bb,T32/-h,T32/-cc,T33/-dd,T33/-k,T36/-e,T36/-bb,T40/-bb,T40/-e,T41/-bb,T41/-e,T44/-bb,T44/-e,T48/-e,T48/-bb,T49/-k,T49/-dd,T52/-e,T52/-bb,T56/-bb,T56/-e}
    {
    \node[shape=circle, draw=red, fill=red!60, text=black,font={\Huge}] at (\i\j) {$7$};
    }
    \foreach \i/\j in {T1/-bb,T1/-cc,T2/-bb,T2/-aa,T3/-bb,T3/-cc,T4/-aa,T5/-bb,T5/-aa,T6/-bb,T6/-cc,T7/-bb,T7/-aa,T8/-cc,T9/-dd,T10/-bb,T11/-bb,T13/-bb,T14/-bb,T15/-bb,T17/-cc,T17/-dd,T18/-bb,T18/-cc,T19/-bb,T19/-aa,T20/-cc,T21/-bb,T21/-cc,T22/-bb,T22/-aa,T23/-bb,T23/-cc,T24/-aa,T25/-bb,T26/-bb,T27/-bb,T29/-bb,T30/-bb,T31/-bb,T33/-bb,T33/-cc,T34/-bb,T34/-aa,T35/-bb,T35/-cc,T36/-aa,T37/-bb,T37/-aa,T38/-bb,T38/-cc,T39/-bb,T39/-aa,T40/-cc,T41/-dd,T41/-cc,T42/-bb,T42/-cc,T43/-bb,T43/-aa,T44/-cc,T45/-bb,T45/-cc,T46/-bb,T46/-aa,T47/-bb,T47/-cc,T48/-aa,T49/-bb,T49/-cc,T50/-bb,T50/-aa,T51/-bb,T51/-cc,T52/-aa,T53/-bb,T53/-aa,T54/-bb,T54/-cc,T55/-bb,T55/-aa,T56/-cc}
    {
    \node[shape=circle, draw=red, fill=red!60, text=black,font={\Huge}] at (\i\j) {$8$};
    }
    }
    }
\end{center}

\section*{Acknowledgements}
A.C was supported by an NSERC USRA, C.R. was supported by the NSERC Discovery grant 2019-05430, and N.S. was supported by the NSERC Discovery grant 2024-04853.


\bibliographystyle{amsplain}

\end{document}